\documentclass[a4paper,abstracton]{scrartcl}
\usepackage{amsfonts}
\usepackage{amssymb}
\usepackage{amsmath}
\usepackage{amsthm}

\usepackage{mathtools}
\usepackage{multirow}

\usepackage[ruled,vlined,linesnumbered]{algorithm2e}

\usepackage{graphicx}
\usepackage{color}
\usepackage[natbibapa]{apacite}

\usepackage{subfigure}
\usepackage{placeins}

\usepackage{a4wide}

\newtheorem{theorem}{Theorem}

\newtheorem{corollary}[theorem]{Corollary}

\usepackage{authblk}

\allowdisplaybreaks

\newcommand{\X}{{\mathcal{X}}}
\newcommand{\cU}{{\mathcal{U}}}
\newcommand{\C}{{\mathcal{C}}}

\newcommand{\tr}{\intercal}

\title{Optimal Scenario Reduction for One- and Two-Stage Robust Optimization}
	
\author{Marc Goerigk}
\author{Mohammad Khosravi\footnote{Corresponding author. Email: mohammad.khosravi@uni-siegen.de}}
	
\affil{Network and Data Science Management, University of Siegen,\authorcr Unteres Schlo{\ss} 3, 57072 Siegen, Germany}

\date{}

\begin{document}

\maketitle

\begin{abstract}
Robust optimization typically follows a worst-case perspective, where a single scenario may determine the objective value of a given solution. Accordingly, it is a challenging task to reduce the size of an uncertainty set without changing the resulting objective value too much. On the other hand, robust optimization problems with many scenarios tend to be hard to solve, in particular for two-stage problems. Hence, a reduced uncertainty set may be central to find solutions in reasonable time. We propose scenario reduction methods that give guarantees on the performance of the resulting robust solution. Scenario reduction problems for one- and two-stage robust optimization are framed as optimization problems that only depend on the uncertainty set and not on the underlying decision making problem. Experimental results indicate that objective values for the reduced uncertainty sets are closely correlated to original objective values, resulting in better solutions than when using general-purpose clustering methods such as K-means.
\end{abstract}


\noindent\textbf{Acknowledgements:} Supported by the Deutsche Forschungsgemeinschaft (DFG) through grant GO 2069/1-1.
	
\section{Introduction}\label{introduction}

Most real-world decision making problems are affected by uncertainty. Depending on the available knowledge and decision maker preferences, several methods exist to include this uncertainty already in the optimization stage, including robust optimization \citep{ben2009robust} and stochastic optimization \citep{powell2019unified}. The complexity to solve these problems depends on the model that is used to describe possible outcomes. A natural choice is to list possible scenarios, in what is known as a discrete uncertainty set in the area of robust optimization. For example, we may have a list of observations available, which parameter values were attained in the past.

Unfortunately, discrete robust optimization problems tend to be hard to solve for discrete uncertainty sets \citep{kasperski2016robust}. This challenge becomes even greater if we allow multiple stages of decision making, potentially with discrete recourse decisions \citep{yanikouglu2019survey}. To improve the solvability of such problems, it would be of benefit to have a method available that reduces the size of the uncertainty set. As robust optimization problems typically consider a worst-case criterion, removing even a single scenario can have a significant impact on the objective value of solutions and thus the choice of an optimal robust solution as well. To the best of our knowledge, no current method offers a principled reduction of the scenario size for robust optimization. \cite{chassein2018scenario} calculate an approximation guarantee that is based on a reduction of the uncertainty set, but any clustering gives the same guarantee -- only their size is relevant. \cite{goerigk2019representative} introduce a method to represent a discrete uncertainty by using a single scenario.

In stochastic optimization, on the other hand, where the availability of a probability distribution allows us to better estimate the impact of scenario reductions, there already exists a tradition to study this problem systematically (see, e.g., \cite{dupavcova2003scenario} and \cite{heitsch2003scenario}). Recent examples include \cite{keutchayan2021problem}, where distance between scenarios is measured by their objective value; \cite{hewitt2021decision}, where the opportunity costs of predicting the wrong scenario is considered; or \cite{bertsimas2022optimization}, where a generalization of the Wasserstein distance called problem-dependent divergence is used to reduce the uncertainty set. Some such approaches, such as \cite{fairbrother2019problem}, even assume that we can sample additional scenarios from the probability distribution. There also exist approaches that cover distributionally robust optimization (see \cite{rahimian2019identifying}), where we would like to protect against the worst-case distribution from an ambiguity set. 

To the best of our knowledge, none of these approaches are tailored towards the specifics of robust optimization, where no probability distribution exists and a worst-case perspective in the problem parameters is used. We propose a method that is suitable for robust optimization. Our approach is to cluster scenarios in a way that optimizes a guarantee how much worse the resulting robust solution can perform on the original uncertainty set compared to an optimal robust solution for this original set. We thus provide instance-dependent approximation guarantees, which are based solely on the uncertainty set, not on the structure of the underlying decision making problem.

This means that our results are also related to current research in the area of approximation algorithms, which is particularly relevant for discrete robust problems. Most robust problems with discrete uncertainty are NP-hard, even for only two scenarios (see, for example, \cite{averbakh2001complexity}, where hardness of the selection problem is discussed). Many problems, including shortest path, minimum spanning tree and knapsack, allow the existence of an FPTAS if the number of scenarios is constant \citep{aissi2010general}. However, they are usually strongly NP-hard if the number of scenarios is unbounded. The selection problem, as an example, cannot be approximated within a constant factor in this case \citep{kasperski2013approximating}. A general-purpose approach for all one-stage robust problems is to optimize with respect to the average scenario, which gives an $N$-approximation, where $N$ is the number of scenarios \citep{aissi2009min}. This does not apply to two-stage problems, however, where approximating the problem becomes even harder. \cite{kasperski2017robust} show that in case of the selection problem, no approximation algorithm better than $\log n$ can exist. No FPTAS seems to be known for any two-stage problem (see \cite{kasperski2016robust}). In this paper, we circumvent such hardness results, as we give guarantees that are specific to the problem at hand, rather than guarantees that hold for any problem. However, our guarantees always apply to the exact robust solution to the reduced problem, which is usually still NP-hard to find. If a polynomial-time algorithm is needed, then the reduced problem would still be needed to be solved with a polynomial-time heuristic.

Our contributions are as follows. In Section~\ref{sec:onestage}, we develop a framework for one-stage robust optimization to calculate an approximation guarantee for any reduced uncertainty set. Using this measure, we derive an optimization problem to find a scenario reduction that gives the smallest possible approximation guarantee. Several approaches to solve this model are proposed, including an iterative method that is similar to the K-means algorithm. An advantage of this heuristic is that each iteration can be done in polynomial time, while the scenario reduction problem is shown to be NP-hard. We also give an upper bound on the worst possible guarantee this approach may give. We then turn to two-stage robust problems in Section~\ref{sec:twostage}. We show how the framework proposed for one-stage problems needs to be modified such that the reduced uncertainty set still yields an approximation guarantee. While the guarantee cannot be better than for one-stage problems, the additional restrictions mean that the reduction problem can be solved more efficiently, though it remains NP-hard. In computational experiments (see Section~\ref{sec:experiments}) we compare our clustering approaches with the popular K-means method. We show that the objective values with respect to our reduced uncertainty sets reach a better correlation to the original robust objective values. This stronger correlation also results in better robust solutions, which is tested using selection and vertex cover problems. We conclude our paper and point to further research questions in Section~\ref{sec:conclusions}.

\section{Optimal Clustering for One-Stage Robust Optimization}
\label{sec:onestage}

We write vectors in bold and use the notation $[n]$ to denote sets $\{1,\ldots,n\}$. We consider linear optimization problems over some set of feasible solutions $\X\subseteq\mathbb{R}^n_+$. In particular, if the cost vector $\pmb{c}\in\mathbb{R}^n_+$ is known, then the so-called nominal problem is to solve
\[ \min_{\pmb{x}\in\X} \pmb{c}^\tr \pmb{x} \]
To formulate the robust counterpart to this problem, we assume that an uncertainty set $\cU\subseteq\mathbb{R}^n_+$ can be identified that contains all possible cost vectors that we would like to protect against. The (one-stage) robust optimization problem is then to find some $\pmb{x}\in\X$ that optimizes the worst-case objective over $\cU$, i.e., to solve
\[ \min_{\pmb{x}\in\X} \max_{\pmb{c}\in\cU} \pmb{c}^\tr \pmb{x} \]
Throughout this paper, we assume that the uncertainty set consists of a list of $N$ explicitly listed scenarios, that is, we assume $\cU=\{\pmb{c}^1,\ldots,\pmb{c}^N\}$.

We first recall a previous approach from \cite{goerigk2019representative} to represent an uncertainty set $\cU$ through a single scenario that is constructed with the help of a linear program. It is based on the following observation: Let $\hat{\pmb{c}}\in \text{conv}(\cU)$ be a scenario with the property that $\pmb{c}^i \le t\hat{\pmb{c}}$ for all $i\in[N]$. Let $\hat{\pmb{x}}$ be an optimizer with respect to $\hat{\pmb{c}}$, and let $\pmb{x}^*$ be an optimal solution with respect to the original robust problem with $N$ scenarios. Then it holds that
\[ \max_{\pmb{c}\in\cU} \pmb{c}^\tr\hat{\pmb{x}} \le t\cdot\hat{\pmb{c}}^\tr\hat{\pmb{x}} \le t\cdot\hat{\pmb{c}}^\tr\pmb{x}^* \le t\cdot\max_{\pmb{c}\in\text{conv}(\cU)} \pmb{c}^\tr\pmb{x}^* = t\cdot\max_{\pmb{c}\in\cU} \pmb{c}^\tr\pmb{x}^* \]
hence, solving with respect to $\hat{\pmb{c}}$ gives a $t$-approximation.

We first extend this principle in the following way.

\begin{theorem}\label{th:onestage}
Let $\cU= \{\pmb{c}^1,\ldots,\pmb{c}^N\}\subseteq\mathbb{R}^n_+$, and let $\hat{\pmb{c}}\in\mathbb{R}^n_+$. Let $\alpha,\beta\ge0$ be such that
\begin{align*}
\forall i\in[N]:\  &\pmb{c}^i \le \alpha \hat{\pmb{c}} \\
\exists \pmb{c}'\in\text{conv}(\cU):\ &\hat{\pmb{c}} \le \beta \pmb{c}'
\end{align*}
Then, any optimizer with respect to $\hat{\pmb{c}}$ gives an $\alpha\beta$-approximation to the robust optimization problem with respect to $\cU$.
\end{theorem}
\begin{proof}
Using the same notation as before, we have
\[ \max_{\pmb{c}\in\cU} \pmb{c}^\tr\hat{\pmb{x}} \le \alpha\cdot \hat{\pmb{c}}^\tr\hat{\pmb{x}}
\le \alpha \cdot \hat{\pmb{c}}^\tr\pmb{x}^* \le \alpha\beta\cdot \pmb{c}'^\tr\pmb{x}^* \le \alpha\beta\cdot\max_{\pmb{c}\in\text{conv}(\cU)} \pmb{c}^\tr\pmb{x}^* = \alpha\beta\cdot\max_{\pmb{c}\in\cU} \pmb{c}^\tr\pmb{x}^* \]
\end{proof}

This result allows us to calculate the approximation guarantee of any scenario $\hat{\pmb{c}}\in\mathbb{R}^n_+$, by calculating $\alpha$ as the worst-case ratio between $\pmb{c}^i$ and $\pmb{c}$ and by finding the smallest value for $\beta$ by solving a small linear program. In Figure~\ref{fig:2dex1}, we show the corresponding approximation guarantees in a small $(n=2)$-dimensional example with $N=2$ scenarios, given by $\cU=\{(4,2)^\tr,(2,3)^\tr\}$. Guarantees above 3 are truncated for better readability.
\begin{figure}[htb]
\begin{center}
\includegraphics[width=0.5\textwidth]{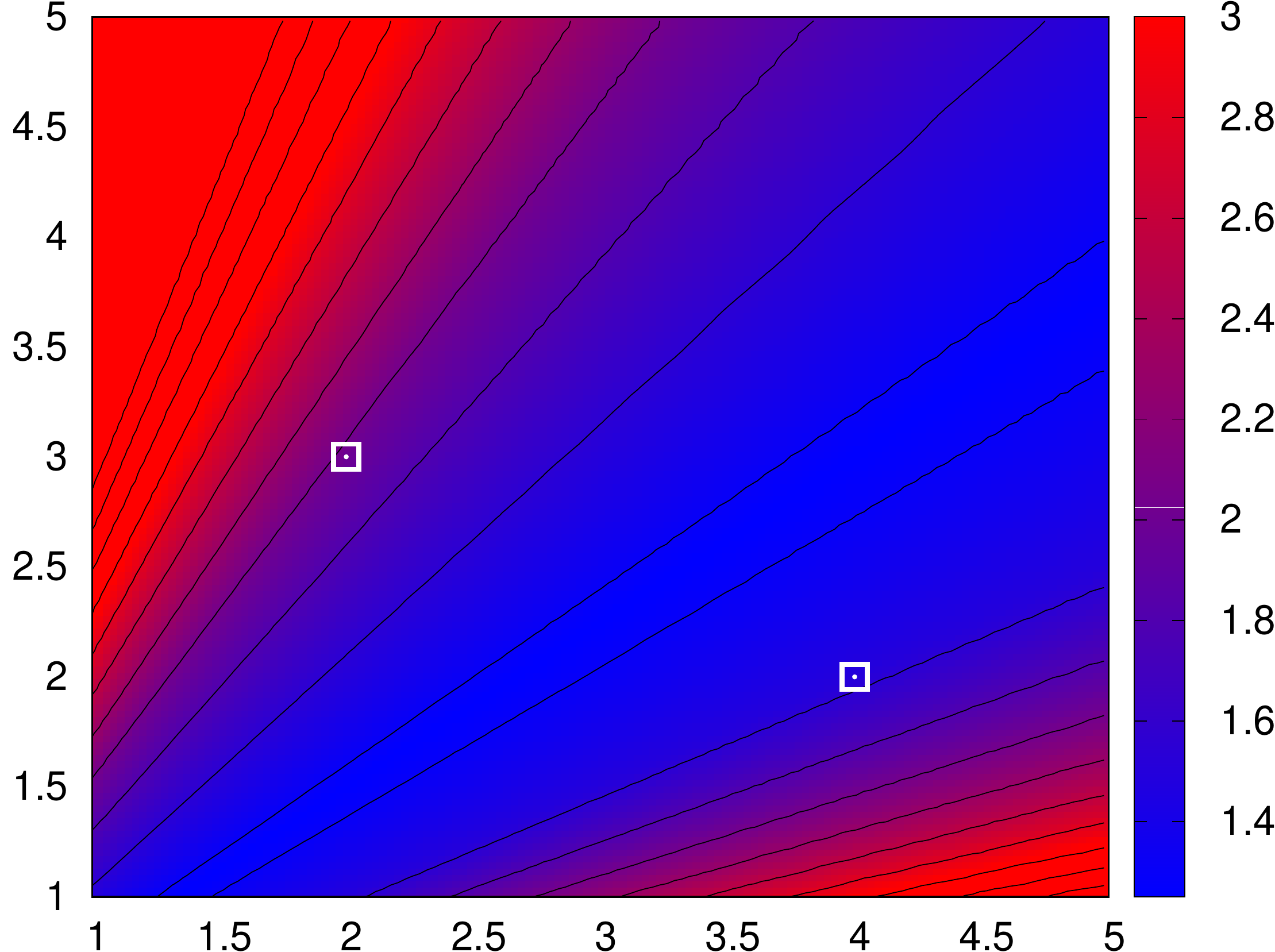}
\caption{Approximation guarantees for a one-stage example problem.}\label{fig:2dex1}
\end{center}
\end{figure}
An optimal approximation ratio of $1.25$ is attained, for example at $\hat{\pmb{c}}=(4,3)^\tr$. Note that all contour lines are straight. This is because any scaling $\lambda \hat{\pmb{x}}$ with $\lambda > 0$ of a point $\hat{\pmb{x}}$ results in the same approximation guarantee by adjusting $\alpha$ and $\beta$ accordingly. Hence, we may simply fix $\beta=1$, which recovers the setting from \cite{goerigk2019representative}. This observation will play a role when we extend our results to the two-stage setting.

We now consider the approximation guarantee when using multiple scenarios $\{\hat{\pmb{c}}^1,\ldots,\hat{\pmb{c}}^K\}=\mathcal{C}$. 
\begin{theorem}\label{th:appr1st}
Let $\cU= \{\pmb{c}^1,\ldots,\pmb{c}^N\}\subseteq\mathbb{R}^n_+$ and $\C = \{\hat{\pmb{c}}^1,\ldots,\hat{\pmb{c}}^K\} \subseteq\mathbb{R}^n_+$ be such that
\begin{align}
\forall i\in[N]\, \exists \hat{\pmb{c}} \in \text{conv}(\C):\ &\pmb{c}^i \le \alpha \hat{\pmb{c}} \label{a1}\\
\forall k\in[K]\, \exists \pmb{c}\in\text{conv}(\cU):\ &\hat{\pmb{c}}^k \le \beta \pmb{c} \label{a2}
\end{align}
for some $\alpha,\beta\ge 0$. Then, an optimal solution to the robust problem with respect to $\C$ gives an $\alpha\beta$-approximation to the robust problem with respect to $\cU$.
\end{theorem}
\begin{proof}
Let $\hat{\pmb{x}}$ be any robust optimizer with respect to $\C$ and let $\pmb{x}^*$ be any robust optimizer with respect to $\cU$.
Let $\pmb{c}^*$ denote a maximizer of $\max_{\pmb{c}\in\cU} \pmb{c}^\tr\hat{\pmb{x}}$, and let $\hat{\pmb{c}}\in \text{conv}(\C)$ be such that $\pmb{c}^* \le \alpha \hat{\pmb{c}}$ according to condition~\eqref{a1}. Furthermore, let $\hat{\pmb{c}}^*$ be a maximizer of $\max_{\pmb{c}\in\C} \pmb{c}^\tr\pmb{x}^*$, and let $\pmb{c}'\in\text{conv}(\cU)$ be such that $\hat{\pmb{c}}^* \le \beta \pmb{c}'$ according to condition~\eqref{a2}. Then it holds that
\[
\max_{\pmb{c}\in\cU} \pmb{c}^\tr\hat{\pmb{x}}
= \pmb{c}^{*\tr} \hat{\pmb{x}} 
\le  \alpha \cdot \hat{\pmb{c}}^\tr\hat{\pmb{x}} 
\le\alpha \cdot \max_{\pmb{c}\in\C} \pmb{c}^\tr\hat{\pmb{x}}
\le \alpha \cdot \max_{\pmb{c}\in\C} \pmb{c}^\tr \pmb{x}^*
= \alpha \cdot \hat{\pmb{c}}^{*\tr} \pmb{x}^*
\le \alpha\beta \cdot \pmb{c}'^\tr\pmb{x}^*
\le \alpha\beta  \cdot \max_{\pmb{c}\in\cU} \pmb{c}^\tr \pmb{x}^*
\]
\end{proof}
Let us assume that we have found some $\C$ with $\alpha$ and $\beta$ as in the conditions of Theorem~\ref{th:appr1st}. Similar to the single-scenario case, we can then replace this clustering with $\beta=1$ and equality in \eqref{a2}. To be more precise, for each $k\in[K]$, let $\tilde{\pmb{c}}^k\in\text{conv}(\cU)$ be such that $\hat{\pmb{c}}^k \le \beta\tilde{\pmb{c}}^k$. Let $\C'=\{\tilde{\pmb{c}}^k : k\in[K]\}$ and set $\alpha' = \alpha\beta$ and $\beta'=1$. Then it holds that for all $i\in[N]$, there is $\tilde{\pmb{c}}\in\text{conv}(\C')$ such that $\pmb{c}^i \le \alpha'\tilde{\pmb{c}}$ and for all $k\in[K]$ there exists $\pmb{c}\in\text{conv}(\cU)$ such that $\tilde{\pmb{c}}^k = \beta\pmb{c}$. Hence, we can consider $\beta=1$ without impairing the resulting approximation guarantee.

\begin{corollary}\label{cor3}
Let $\cU= \{\pmb{c}^1,\ldots,\pmb{c}^N\}$ be an uncertainty set, and let $\{\hat{\pmb{c}}^1,\ldots,\hat{\pmb{c}}^K\} = \C \subseteq \text{conv}(\cU)$. Let $\alpha\in\mathbb{R}_+$ such that for each $i\in[N]$, there is a scenario $\hat{\pmb{c}}\in\text{conv}(\mathcal{C})$ such that $\pmb{c}^i \le \alpha \hat{\pmb{c}}$. Then, an optimal solution to the robust problem with respect to $\C$ gives an $\alpha$-approximation to the robust problem with respect to $\cU$.
\end{corollary}

The conditions of Corollary~\ref{cor3} can be framed as an optimization problem. We use variables $\lambda_{ki}\ge 0$ with $\sum_{i\in[N]} \lambda_{ki}=1$ to define $\hat{\pmb{c}}^k$ as a convex combination of scenarios $\cU$. Furthermore, for each scenario $\pmb{c}^i \in\cU$ we use $\mu_{ik}\ge 0$ with $\sum_{k\in[K]} \mu_{ik} =1$ to define the convex combination of scenarios in $\mathcal{U}$ that dominates $\pmb{c}^i$. The resulting optimization problem is then as follows.
\begin{align}
\max\ &t \label{contstart} \\
\text{s.t. } & t c^i_j \le \sum_{k\in[K]} \sum_{\ell\in[N]} \mu_{ik}  \lambda_{k\ell} c^\ell_j &  \forall i\in[N], j\in[n]  \label{contNL}\\
& \sum_{i\in[N]} \lambda_{ki} = 1 & \forall k\in[K] \\
& \sum_{k\in[K]} \mu_{ik} = 1 & \forall i\in[N] \\
& t \ge 0 \\
& \lambda_{ki} \ge 0 & \forall k\in[K], i\in[N] \\
& \mu_{ik} \ge 0 & \forall i\in[N], k\in[K] \label{contend}
\end{align}
where the aggregated scenarios in $\mathcal{C}$ can be calculated using $\hat{c}^k_j = \sum_{i\in[N]} \lambda_{ki} c^i_j$.

\begin{corollary}
Any feasible solution to problem~(\ref{contstart}-\ref{contend}) gives a reduced scenario set $\mathcal{C}$, where an optimal solution for the robust optimization problem with respect to $\mathcal{C}$ is a $1/t$-approximation to the original robust optimization problem with respect to $\cU$.
\end{corollary}

Problem~(\ref{contstart}-\ref{contend}) consists of determining values for a matrix $M=(\mu_{ik})$ and a matrix $\Lambda=(\lambda_{ki})$ of dimensions $N\times K$ and $K\times N$, respectively. Intuitively, we can thus imagine a solution akin to a compression step from $N$ down to $K$ scenarios, and a subsequent decompression step from $K$ up to $N$ scenario again.

Note that problem~(\ref{contstart}-\ref{contend}) is non-linear, due to the multiplication of $\mu$ and $\lambda$ variables. State-of-the-art solvers such as Gurobi offer built-in methods to solve bilinear optimization problems of this type to optimality, based on spatial branching. In our experience, this approach quickly becomes intractable, even for small problems with $n<10$ and $N<10$. We discuss two problem variants and one heuristic to avoid this computational difficulty.

In these problem variants, we require either variables $\mu$ or variables $\lambda$ to be binary instead of continuous. We can then linearize constraint~\eqref{contNL} by introducing new variables $\tau_{ik\ell} = \mu_{ik}\cdot \lambda_{k\ell}$ with
\begin{align*}
& t c^j_i \le \sum_{k\in[K]} \sum_{\ell\in[N]} \tau_{ik\ell} c^\ell_j & \forall i\in[N], j\in[n] \\
& \tau_{ik\ell} \le \mu_{ik} & \forall i\in[N], k\in[K], \ell\in[N] \\
& \tau_{ik\ell} \le \lambda_{k\ell} & \forall i\in[N], k\in[K], \ell\in[N] 
\end{align*}
However, in both cases, more efficient models are possible.

First, we may consider variables $\mu_{ik}$ to be binary instead of continuous. This means that we forgo the possibility to dominate scenarios $\pmb{c}^i$ through convex combinations of scenarios from $\mathcal{C}$; instead, each scenario from $\cU$ is assigned a single scenario from $\mathcal{C}$. This is a clustering approach, where the scenarios from $\cU$ are grouped in $K$ distinct clusters, and each cluster defines a scenario $\hat{\pmb{c}}$. The advantage of this approach is that the nonlinearity in problem~(\ref{contstart}-\ref{contend}) can be linearized by using
\begin{align}
& t c^i_j \le \sum_{\ell\in[N]} \lambda_{k\ell} c^\ell_j + M(1-\mu_{ik}) & \forall i\in[N], j\in[n], k\in[K]
\end{align}
where $M\ge c^i_j$ for all $i\in[N],j\in[n]$ is a sufficiently large constant. We refer to this approach as IP-$\mu$.

Second, we may consider variables $\lambda_{ki}$ to be binary. The consequence is that our reduced scenario set $\mathcal{C}$ now consists of a subset of scenarios from $\cU$, instead of convex combinations of these scenarios. In other words, we consider a scenario reduction problem, where $N-K$ scenarios need to be removed. We can model this problem using binary variables $\lambda_i$ to denote whether scenario $i\in[N]$ is part of $\mathcal{C}$. The optimization problem we consider is thus:
\begin{align}
\max\ &t \label{ipl1}\\
\text{s.t. } & t \pmb{c}^i \le \sum_{\ell\in[N]} \mu_{i\ell} \pmb{c}^\ell & \forall i\in[N] \label{ipl2}\\
& \mu_{i\ell} \le \lambda_\ell & \forall i\in[N], \ell\in[N] \label{ipl3}\\
& \sum_{i\in[N]} \lambda_i = K \label{ipl4}\\
& \sum_{\ell\in[N]} \mu_{i\ell} = 1& \forall i\in[N] \label{ipl5}\\
& \lambda_i \in \{0,1\} & \forall i\in[N] \label{ipl6}\\
& \mu_{i\ell} \ge 0 & \forall i\in[N], \ell\in[N] \label{ipl7}\\
& t \ge 0 \label{ipl8}
\end{align}
We refer to this approach as IP-$\lambda$. Observe that constraints~(\ref{ipl3}-\ref{ipl7}) are the same as in previous models proposed to reduce scenarios in stochastic optimization, see, e.g., problem (20) in \cite{bertsimas2022optimization}, which itself is based on \cite{heitsch2003scenario} and \cite{rujeerapaiboon2018scenario}. This is not surprising, as they naturally model a choice of a subset of scenarios, where each of the existing scenarios needs to be put in relation to one of the scenarios from the chosen subset.

Finally, we may solve the continuous and non-linear problem~(\ref{contstart}-\ref{contend}) heuristically using an iterative approach that is explained later in this section. For simplicity, we refer to both this method as well as to the formulation (\ref{contstart}-\ref{contend}) as Cont, if the context is clear.

Unfortunately, all three problems are hard to solve, as the following result indicates. The proofs of this result can be found in Appendix~\ref{sec:proofs}.

\begin{theorem}
The decision versions of Cont, IP-$\mu$ and IP-$\lambda$ are NP-complete.
\end{theorem}

To illustrate these three methods, we present two datasets in Figure~\ref{fig:1stexamples}. In both cases all scenarios are non-dominated (notice that dominated scenarios would not influence our methods, i.e., they become filtered out automatically). In the left column, data points $\cU$ presented as black circles follow a convex curve, while in the right column, data points follow a concave curve. The red crosses represent the clustered scenario sets $\mathcal{C}$, where we reduced $N=20$ original scenarios down to $K=3$ scenarios.

\begin{figure}[htbp]
\begin{center}
\subfigure[Dataset 1, IP-$\mu$\label{fig1a}]{\includegraphics[width=0.4\textwidth]{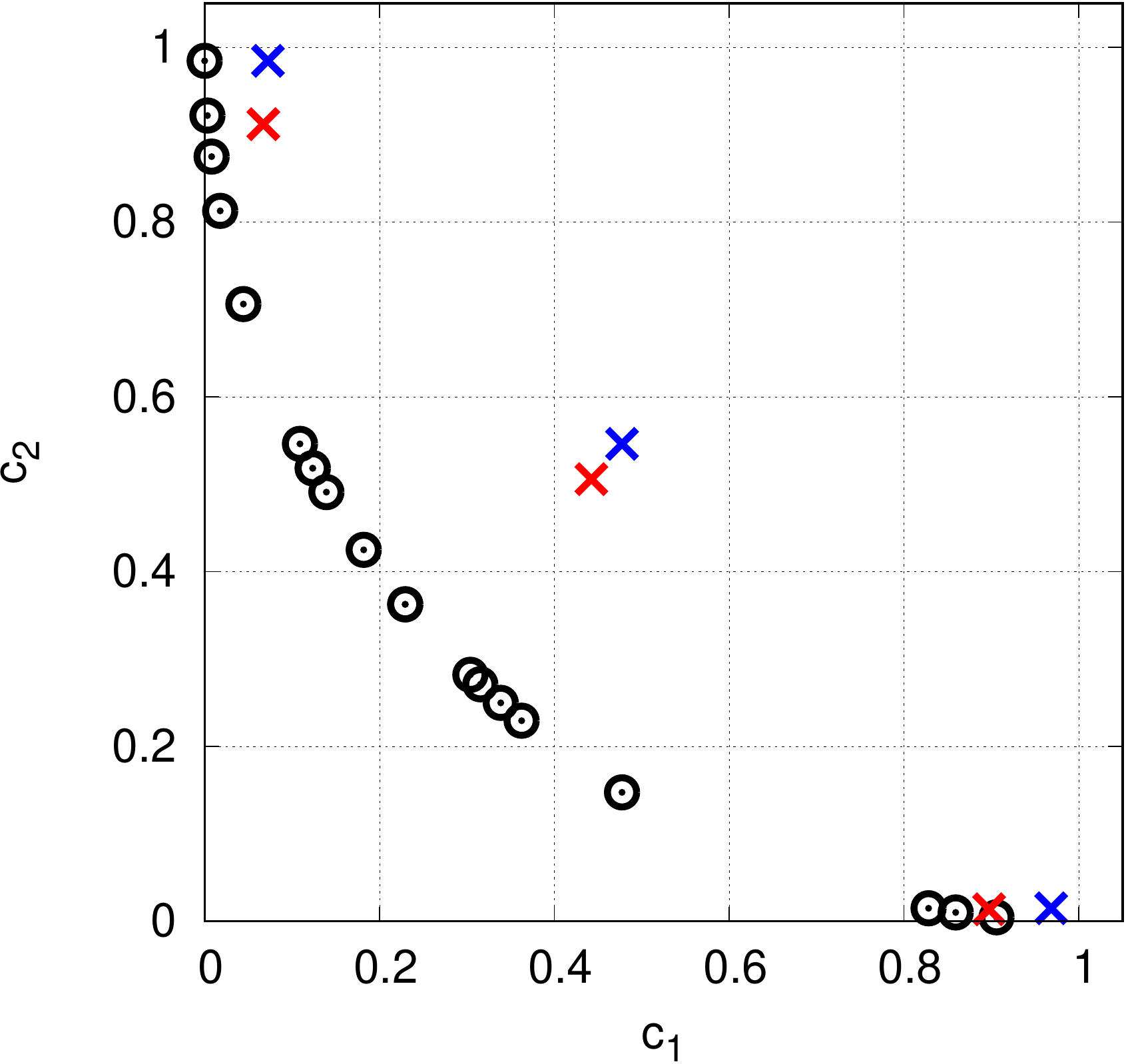}}\hspace{5mm}%
\subfigure[Dataset 2, IP-$\mu$\label{fig1b}]{\includegraphics[width=0.4\textwidth]{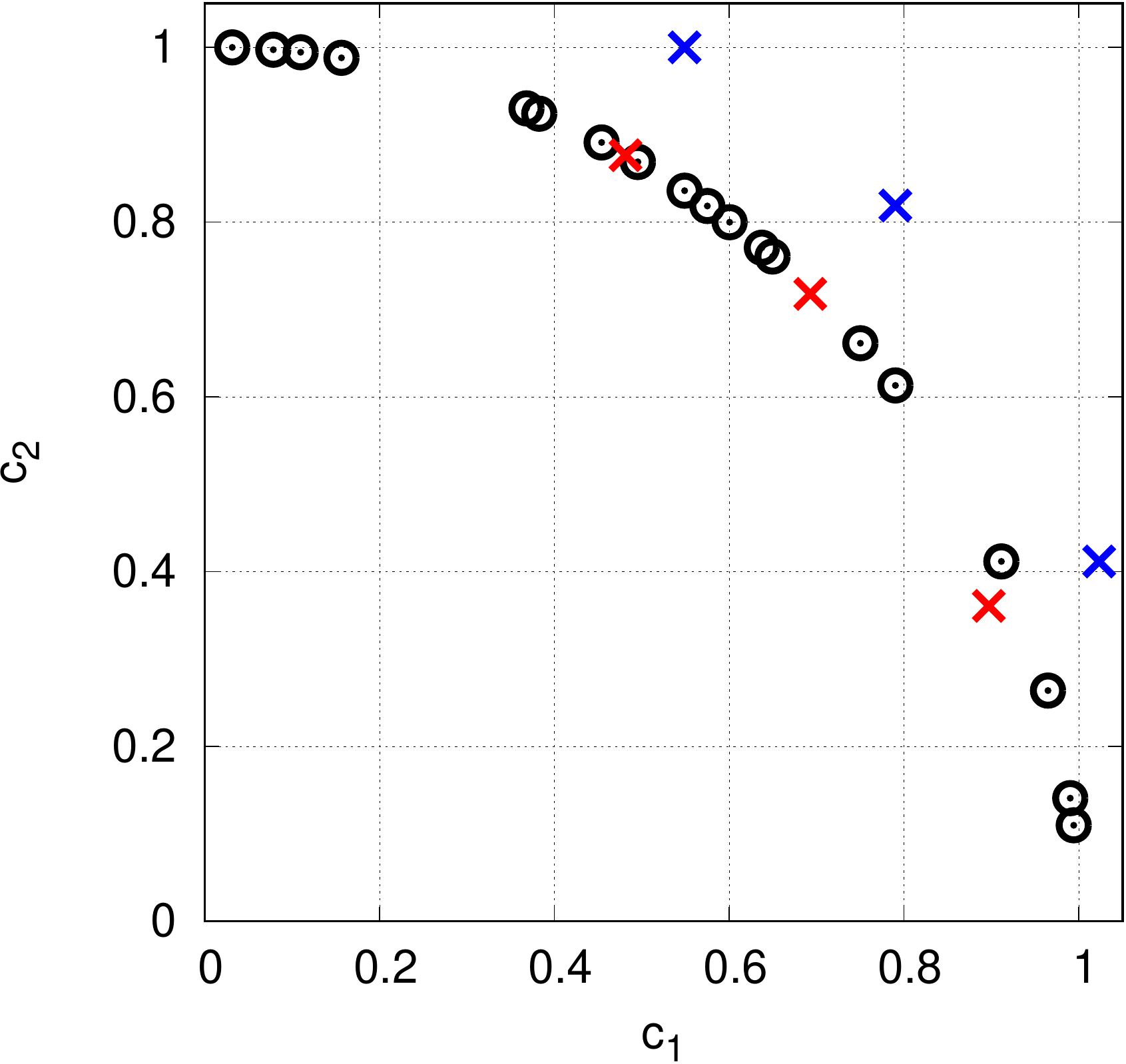}}
\subfigure[Dataset 1, IP-$\lambda$\label{fig1c}]{\includegraphics[width=0.4\textwidth]{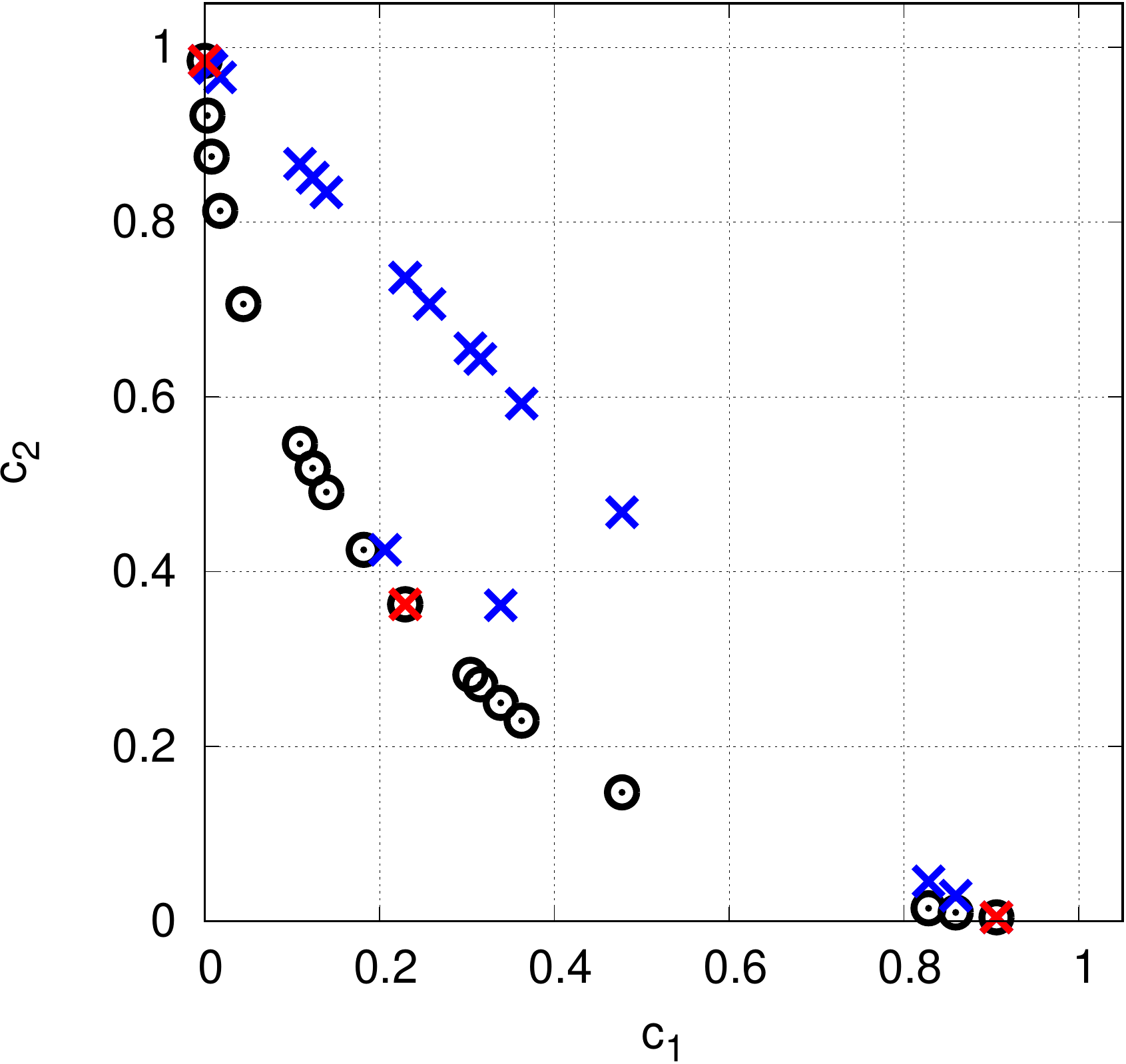}}\hspace{5mm}%
\subfigure[Dataset 2, IP-$\lambda$\label{fig1d}]{\includegraphics[width=0.4\textwidth]{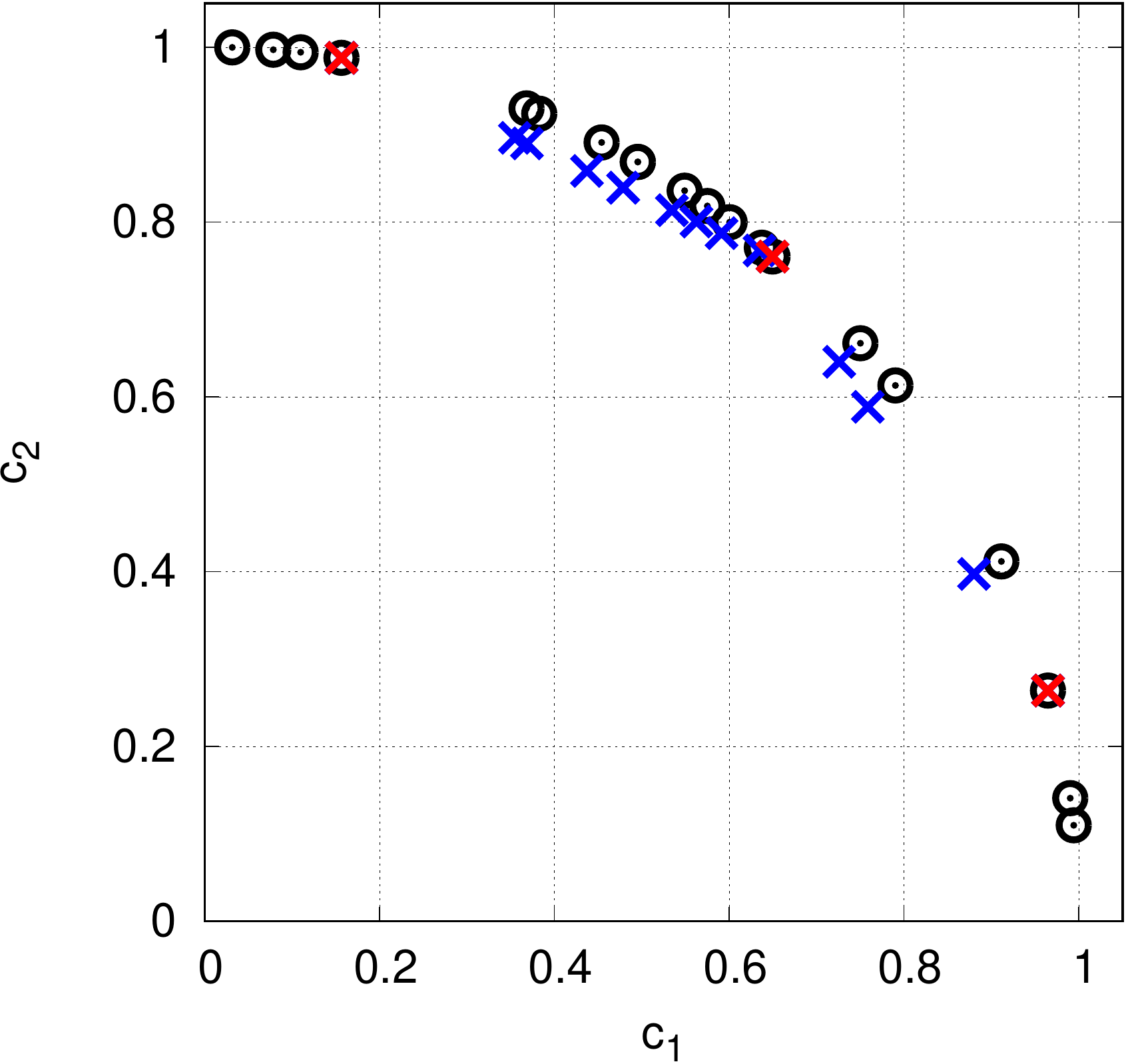}}
\subfigure[Dataset 1, Cont\label{fig1e}]{\includegraphics[width=0.4\textwidth]{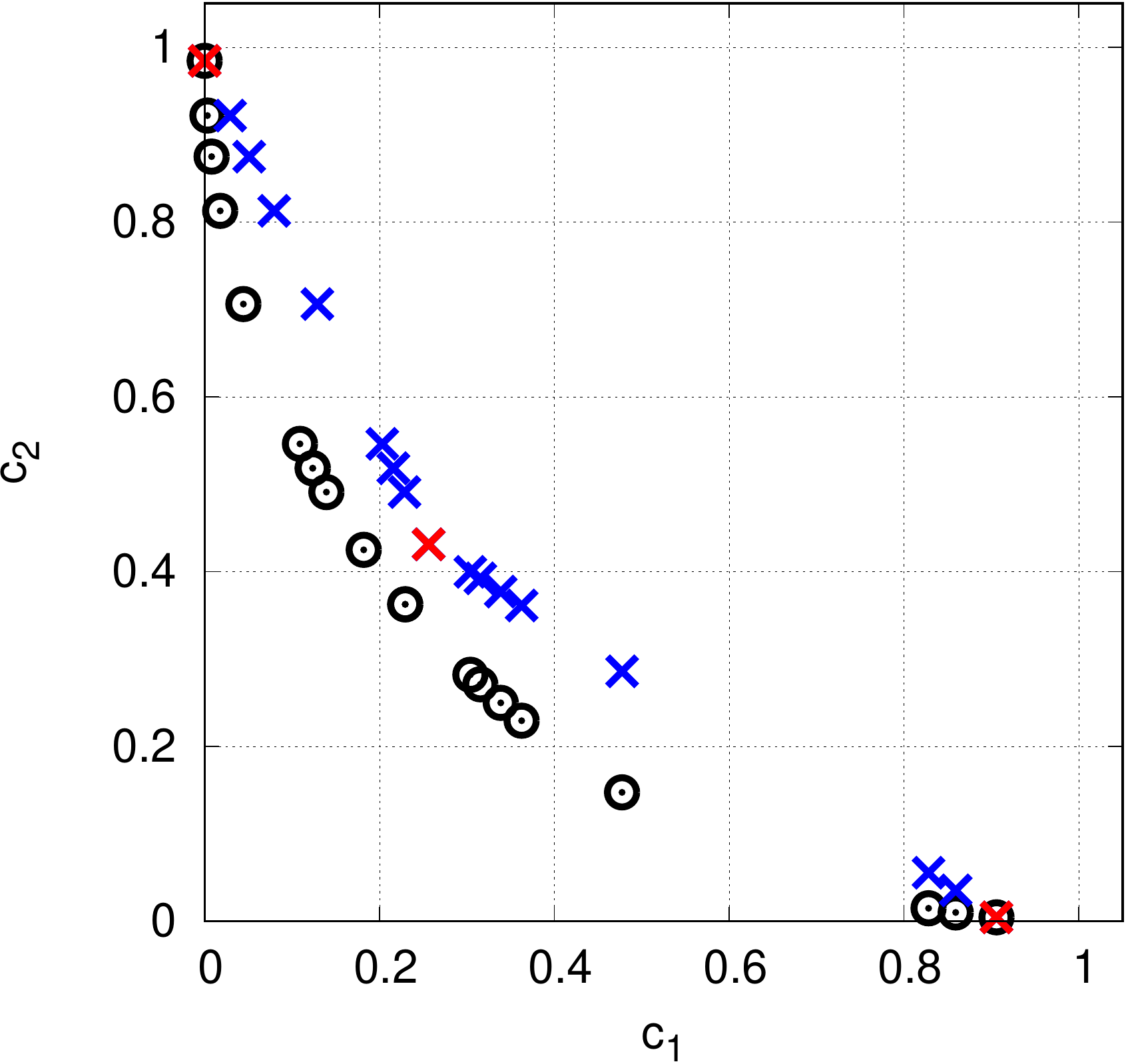}}\hspace{5mm}%
\subfigure[Dataset 2, Cont\label{fig1f}]{\includegraphics[width=0.4\textwidth]{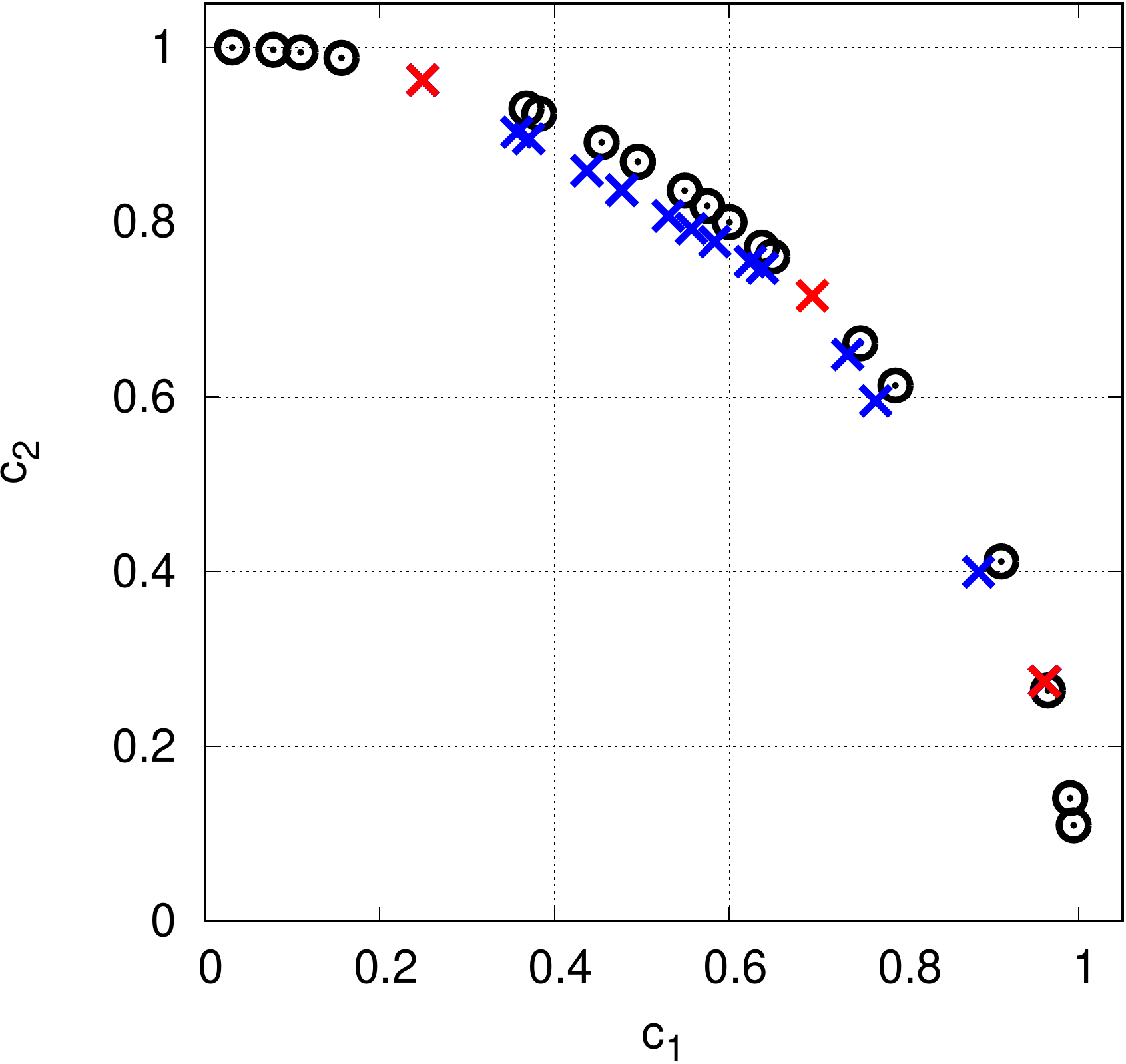}}
\caption{Clustering examples.}\label{fig:1stexamples}
\end{center}
\end{figure}

First consider Figures~\ref{fig1a} and \ref{fig1b}. In each case, the aggregated scenarios are convex combinations of the original scenarios. We represent by blue crosses the scaling of these scenarios such that each original scenario is dominated by one of the blue crosses. These scaling factors (1.079 for Figure~\ref{fig1a}, 1.141 for Figure~\ref{fig1b}) give the approximation guarantee of the reduced set.

Now consider Figures~\ref{fig1c} and \ref{fig1d}, which correspond to the solutions of IP-$\lambda$, i.e., we choose $K=3$ out of $N=20$ scenarios that result in the best approximation guarantee. The blue crosses indicate the convex combinations built from these three scenarios such that each black point becomes dominated (after scaling the blue crosses up by a sufficient factor). While we do not need any scaling at all for the first dataset, meaning that we have a guarantee that the resulting robust solution is optimal for the original problem (note that this could even be achieved with $K=2$), a scaling factor of 1.042 is necessary for the second dataset.

Finally, Figures~\ref{fig1e} and \ref{fig1f} show the solutions found by using method Cont. As for IP-$\mu$, the red crosses that indicate scenario set $\mathcal{U}$ are convex combinations of the original scenarios. As for IP-$\lambda$, we can use these for new convex combinations to dominate the original scenarios after scaling. While we still have a guarantee that the resulting robust solution is optimal in case of the first dataset, the second dataset needs a slightly smaller scaling factor in comparison to IP-$\lambda$, resulting in a 1.039-approximation guarantee.

Note that Theorem~\ref{th:appr1st} does not depend on the underlying problem, but only on the scenario data. To illustrate and compare the strength of the resulting bounds, we generate random data with $N=10$ scenarios and $n=10$. We test data generated by a uniform distribution in $\{1,\ldots,100\}$, and by a multivariate normal distribution that is truncated to $[1,100]$. In Figure~\ref{fig:ex1st} we show the corresponding results, averaged over 100 problem instances. On the horizontal axis is $K$, the size of the reduced uncertainty set $\mathcal{C}$. Note that these are guarantees that come with the clustering, without solving any robust optimization problem.

\begin{figure}[htb]
\begin{center}
\subfigure[Uniform data.]{\includegraphics[width=.48\textwidth]{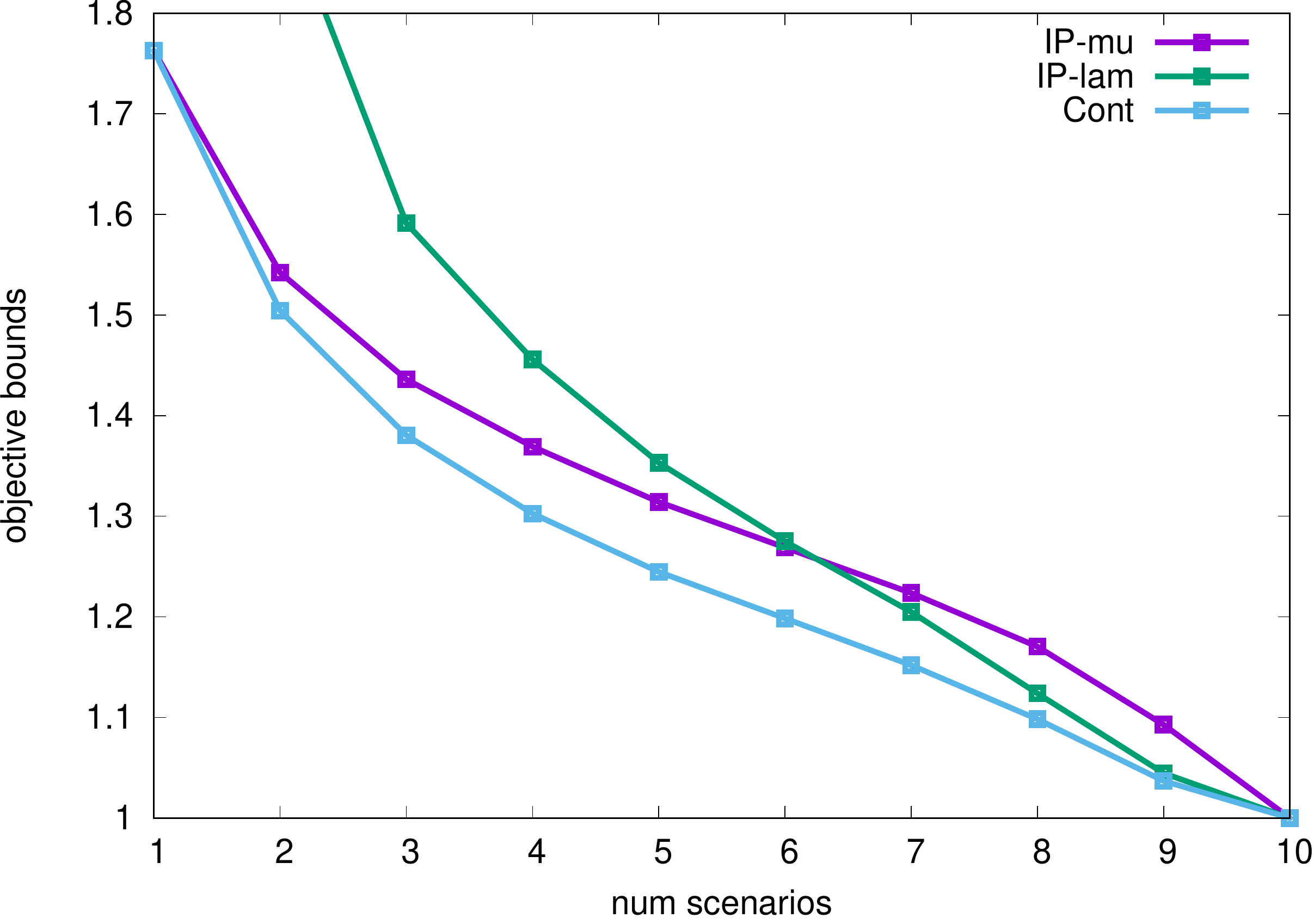}}\hfill
\subfigure[Normal data.]{\includegraphics[width=.48\textwidth]{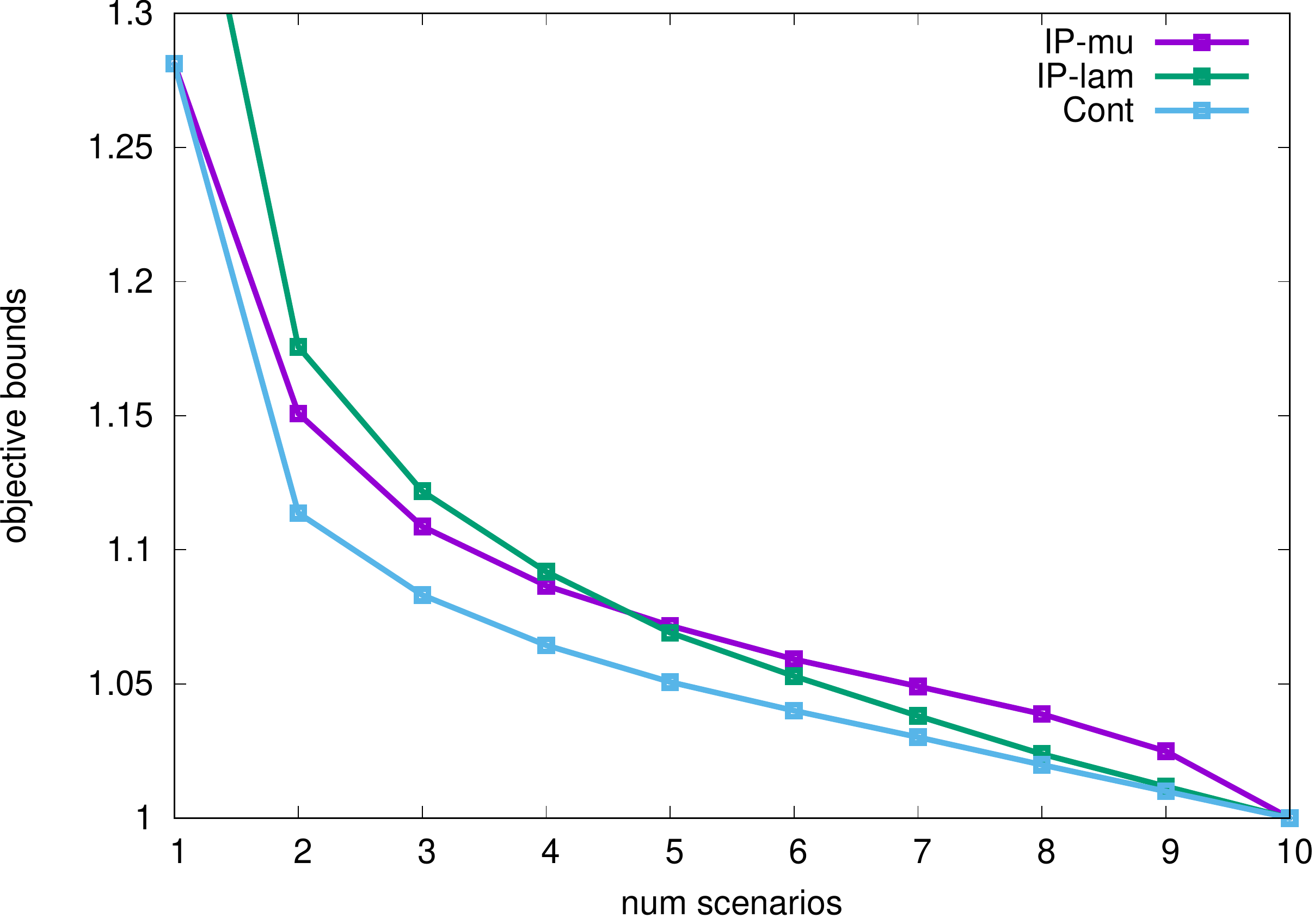}}
\end{center}
\caption{Approximation guarantees for randomly generated data with $n=N=10$.}\label{fig:ex1st}
\end{figure}

By construction, an optimal solution to the nonlinear, continuous problem gives a guarantee that is not worse than the guarantees found by IP-$\mu$ and IP-$\lambda$. Accordingly, note that the corresponding curve (in light blue) remains below the other two curves at all times. Note that for $K=1$, IP-$\mu$ and Cont give the same results. That is to be expected, as binary and continuous variables $\mu$ lead to the same result in this case (each scenario $\pmb{c}^i$ must be dominated by the one scenario $\hat{\pmb{c}}^1$ that is available). Furthermore, IP-$\lambda$ tends to perform considerably worse than the other methods for small $K$, as a small subset of scenarios is not sufficient to represent the original data (whereas convex combinations of scenarios are more suitable here). This changes as $K$ increases, when IP-$\lambda$ tends to outperform IP-$\mu$. Notice also that approximation guarantees for normally distributed data are considerably smaller than for uniformly distributed data.

Recall that the midpoint solution to a robust optimization problem (i.e., the solution that is found by solving a problem with a single scenario $\hat{\pmb{c}}=\sum_{i\in[N]}\pmb{c}^i/N$) gives an $N$-approximation. Using this observation, we can give a simple upper bound on the approximation guarantee found by our optimal clustering model.

\begin{corollary}\label{cor:ck}
Let $\cU=\{\pmb{c}^1,\ldots,\pmb{c}^N\}$, and let $C_1\cup C_2 \cup \ldots \cup C_K$ be any partition of $\cU$. Let $\cU'=\{\hat{\pmb{c}}^1,\ldots,\hat{\pmb{c}}^K\}$ be the average scenarios of each partition. Then, an optimizer of the robust problem with respect to $\cU'$ is a $\max_{k\in[K]} |C_k|$-approximation for the robust problem with respect to $\cU$.
\end{corollary}
\begin{proof}
The claim follows directly from Theorem~\ref{th:appr1st} and the observation that $\hat{\pmb{c}}^k \le |C_k| \pmb{c}$ for all $\pmb{c}\in C_k$.
\end{proof}

\begin{corollary}\label{cor:bound}
There is an optimal solution to problem~(\ref{contstart}-\ref{contend}) that gives an approximation guarantee of at most $\lceil N/K \rceil$.
\end{corollary}
\begin{proof}
The claim follows directly from Corollary~\ref{cor:ck} and by partitioning scenarios in clusters of size as uniformly as possible.
\end{proof}

Note that Corollary~\ref{cor:bound} applies to Cont and IP-$\mu$. It does not apply to IP-$\lambda$; indeed, a simple problem as $\cU=\{(1,0)^\tr,(0,1)^\tr\}$ with $K=1$ illustrates that it may not be possible to find a finite approximation guarantee with this approach, if there exists no scenario that is non-zero in each dimension.

We now describe a heuristic to solve Cont. Recall the K-means algorithm to find a clustering that minimizes the squared quadratic distance of each data point to its nearest center point: We first choose a random set of $K$ existing points as center points. To each point, we assign the nearest center point to find $K$ clusters. We calculate the midpoint of each cluster. This is repeated by iteratively assigning the nearest center to each point to find new clusters, and finding new center points as the midpoints of clusters, as long as the sum of distances keeps improving. As the result of the K-means algorithm depends on the set of starting centers, it is usually repeated multiple times and the best result is chosen.

A similar heuristic can be applied in our setting. We first choose a random subset as $\mathcal{C}$, i.e., for each $k\in[K]$, we choose a random $i\in[N]$ and set $\lambda_{ki} = 1$. With these starting scenarios, we solve a problem only in $\mu$ and $t$, i.e., we solve the following linear program:
\begin{align*}
\max\ & t \\
\text{s.t. } & tc^i_j \le \sum_{k\in[K]} \mu_{ik} \hat{c}^k_j & \forall i\in[N], j\in[n] \\
& \sum_{k\in[K]} \mu_{ik} = 1 & \forall i\in[N] \\
&t \ge  0 \\
& \mu_{ik} \ge 0 & \forall i\in[N], k\in[K]
\end{align*}
Note that this problem can be decomposed. If we use a separate variables $t_i$ for each cluster, we may use constraints
\[ t_i c^i_j \le \sum_{k\in[K]} \mu_{ik} \hat{c}^k_j \qquad \forall i\in[N], j\in[n] \]
and an objective function $\sum_{i\in[N]} t_i$ to minimize the sum of scaling factors. By construction, this will lead to the same worst-case guarantee, but has the advantage that we find the best variables $\mu_{ik}$ for each scenario $i\in[N]$, and not only for those scenarios where the scaling factor defines the worst case.

Having found a solution for $\mu$ this way, we again solve a linear program to determine new scenarios by their defining $\lambda$-values:
\begin{align*}
\max\ & t \\
\text{s.t. } & tc^i_j \le \sum_{k\in[K]} \sum_{\ell\in[N]} ( \mu_{ik}c^\ell_j) \lambda_{k\ell}  & \forall i\in[N], j\in[n] \\
& \sum_{i\in[N]} \lambda_{ki} = 1 & \forall k\in[K] \\
&t \ge  0 \\
& \lambda_{ki} \ge 0 & \forall k\in[K], i\in[N]
\end{align*}
We repeat this iterative heuristic, alternating between solving for $\mu$ and solving for $\lambda$, until the objective value does not improve. As our result may depend on the random starting scenarios, we repeat the process multiple times and choose the solution with the best guarantee in the end. As we only solve linear programs, each iteration remains solvable in polynomial time. We can still expect the K-means algorithm to be faster, as it does not even require the solution of linear programs in each iteration. We therefore expect a clustering that is better suited for the requirements of robust optimization, but at the cost of increased computational effort, when our method is applied instead of K-means.

\section{Optimal Clustering for Two-Stage Robust Optimization}
\label{sec:twostage}

We now consider two-stage robust optimization problems, where the decision maker has the opportunity to react once the scenario has been revealed. We split variables into here-and-now variables $\pmb{x}\in\mathbb{R}^{n_x}_+$ that need to be decided beforehand, and wait-and-see variables $\pmb{y}\in\mathbb{R}^n_+$ that can be decided later. Let $\X(\pmb{x}) = \{\pmb{y}\in\mathbb{R}^{n}_+ : (\pmb{x},\pmb{y}) \in\X\}$ be the set of feasible second-stage solutions and let $\X'=\{\pmb{x}\in\mathbb{R}^{n_x}_+ : \X(\pmb{x})\neq\emptyset\}$ denote the set of feasible first-stage solutions. The two-stage robust optimization problem is then to solve
\[ \min_{\pmb{x}\in\X'} \max_{\pmb{c}\in\cU} \min_{\pmb{y}\in\X(\pmb{x})} \pmb{C}^\tr\pmb{x} + \pmb{c}^\tr\pmb{y} \]

We first recall why such two-stage problems can be hard to approximate when recourse variables are discrete.
Consider an example with 
$\X= \{ (\pmb{x},\pmb{y})\in\{0,1\}^{2\times 2} : x_1 + x_2 + y_1 + y_2 = 1\}$
where we can buy at most one item in the first stage, and only buy an item in the second stage if we decided not to buy in the first stage. There are two scenarios with costs as in Table~\ref{tab:ex2st}. An optimal solution is not to buy anything in the first stage. As it is possible to pack an item with costs 0 in each scenario, the objective value of this solution is 0. Now consider any solution with respect to a scenario $\hat{\pmb{c}} = \lambda \pmb{c}^1 + (1-\lambda)\pmb{c}^2$ for some $\lambda\in(0,1)$. As $M\to\infty$, it becomes cheaper to buy an item in the first stage, which gives an objective value of 1. In particular, solving with respect to the midpoint scenario does not give an approximation guarantee.

\begin{table}[htb]
\begin{center}
\begin{tabular}{r|rr}
 & 1 & 2 \\
 \hline
$C_j$ & 1 & 1 \\
$c^1_j$ & $M$ & 0 \\
$c^2_j$ & 0 & $M$
\end{tabular}
\caption{Example two-stage problem. Each column represents costs of one item.}\label{tab:ex2st}
\end{center}
\end{table}

Furthermore, observe that in two-stage problems, we have that the objective value with respect to $\cU$ and with respect to $\text{conv}(\cU)$ are in general not the same. Intuitively, this means that a clustering approach needs to balance two effects: on the one hand, by choosing scenarios from the convex hull, we may increase the objective value; on the other hand, by choosing less than $N$ scenarios, we may decrease the objective value.

We first consider the case how to find a good approximation using a single scenario.

\begin{theorem}\label{th:twostage}
Let $\cU= \{\pmb{c}^1,\ldots,\pmb{c}^N\}\subseteq\mathbb{R}^n_+$, and let $\hat{\pmb{c}}\in\mathbb{R}^n_+$. Let $\alpha,\beta\ge 1$ be such that
\begin{align*}
\forall i\in[N]:\  &\pmb{c}^i \le \alpha \hat{\pmb{c}} \\
\exists \pmb{c}'\in \cU:\ &\hat{\pmb{c}} \le \beta \pmb{c}'
\end{align*}
Then, any optimizer with respect to $\hat{\pmb{c}}$ gives an $\alpha\beta$-approximation to the two-stage robust optimization problem with respect to $\cU$.
\end{theorem}
Note that there are two differences to the conditions stated in Theorem~\ref{th:onestage}: Parameters $\alpha$ and $\beta$ need to be greater or equal one instead of zero; and the scenario $\pmb{c}'$ needs to be in $\cU$ instead of $\text{conv}(\cU)$.
\begin{proof}
Let $\hat{\pmb{c}}$ as assumed, and let $\hat{\pmb{x}}$ be a minimizer to the two-stage problem with $\hat{\pmb{c}}$ as the only scenario. We denote by $\pmb{x}^*$ a minimizer with respect to $\cU$. Then we can estimate:
\begin{align}
\max_{\pmb{c}\in\cU} \min_{\pmb{y}\in\X(\hat{\pmb{x}})} \pmb{C}^\tr\hat{\pmb{x}} + \pmb{c}^\tr\pmb{y} &= \pmb{C}^\tr\hat{\pmb{x}} +  \max_{\pmb{c}\in\cU} \min_{\pmb{y}\in\X(\hat{\pmb{x}})} \pmb{c}^\tr\pmb{y} \label{2stproof-1}\\
&\le \pmb{C}^\tr\hat{\pmb{x}} + \min_{\pmb{y}\in\X(\hat{\pmb{x}})} \alpha \hat{\pmb{c}}^\tr \pmb{y} \label{2stproof-2} \\
&\le \alpha \left( \pmb{C}^\tr\hat{\pmb{x}} + \min_{\pmb{y}\in\X(\hat{\pmb{x}})} \hat{\pmb{c}}^\tr \pmb{y} \right) \label{2stproof-3} \\
&\le \alpha \left( \pmb{C}^\tr\pmb{x}^* + \min_{\pmb{y}\in\X(\pmb{x}^*)} \hat{\pmb{c}}^\tr \pmb{y} \right) \label{2stproof-4} \\
&\le \alpha \left( \pmb{C}^\tr\pmb{x}^* + \beta \max_{\pmb{c}\in\cU} \min_{\pmb{y}\in\X(\pmb{x}^*)} \pmb{c}^\tr\pmb{y} \right) \label{2stproof-5} \\
& \le \alpha\beta \left( \pmb{C}^\tr\pmb{x}^* + \max_{\pmb{c}\in\cU} \min_{\pmb{y}\in\X(\pmb{x}^*)} \pmb{c}^\tr\pmb{y} \right) \label{2stproof-6}
\end{align}
In equation~\eqref{2stproof-1}, we move the constant part $\pmb{C}^\tr\hat{\pmb{x}}$ out of the max-min problem. Let $\pmb{c}^j\in\cU$ be a maximizer of this max-min problem. Due to the assumptions, it holds that $\pmb{c}^j\le \alpha \hat{\pmb{c}}$, which gives the estimate~\eqref{2stproof-2}. As $\alpha \ge 1$ and $\pmb{C}^\tr\hat{\pmb{x}}\ge 0$, we can conclude \eqref{2stproof-3}. By definition $\hat{\pmb{x}}$ is the optimizer with respect to the single scenario $\hat{\pmb{c}}$. In particular, its objective value with respect to this uncertainty set is not larger than the objective value of solution $\pmb{x}^*$, hence \eqref{2stproof-4} follows. Let $\pmb{c}^j$ be a maximizer of $\max_{\pmb{c}\in\cU} \min_{\pmb{y}\in\X(\pmb{x}^*)} \pmb{c}^\tr\pmb{y}$. By construction, $\hat{\pmb{c}} \le \beta\pmb{c}^j$. Hence, estimate \eqref{2stproof-5} is valid. Finally, as $\beta \ge 1$ and $\pmb{C}^\tr\pmb{x}^*\ge 0$, we reach \eqref{2stproof-6}. We conclude that the robust objective value of solution $\hat{\pmb{x}}$ is at most $\alpha\beta$ times the optimal robust objective value, which completes the proof.
\end{proof}

We revisit the small example from Section~\ref{sec:onestage} with $\cU=\{(4,2)^\tr,(2,3)^\tr\}$. As before, we can calculate tight values for $\alpha$ and $\beta$, given any candidate scenario $\hat{\pmb{c}}$. The corresponding approximation guarantees are presented in Figure~\ref{fig:2dex2}. Guarantees above 3 are again truncated for better readability.
\begin{figure}[htb]
\begin{center}
\includegraphics[width=0.5\textwidth]{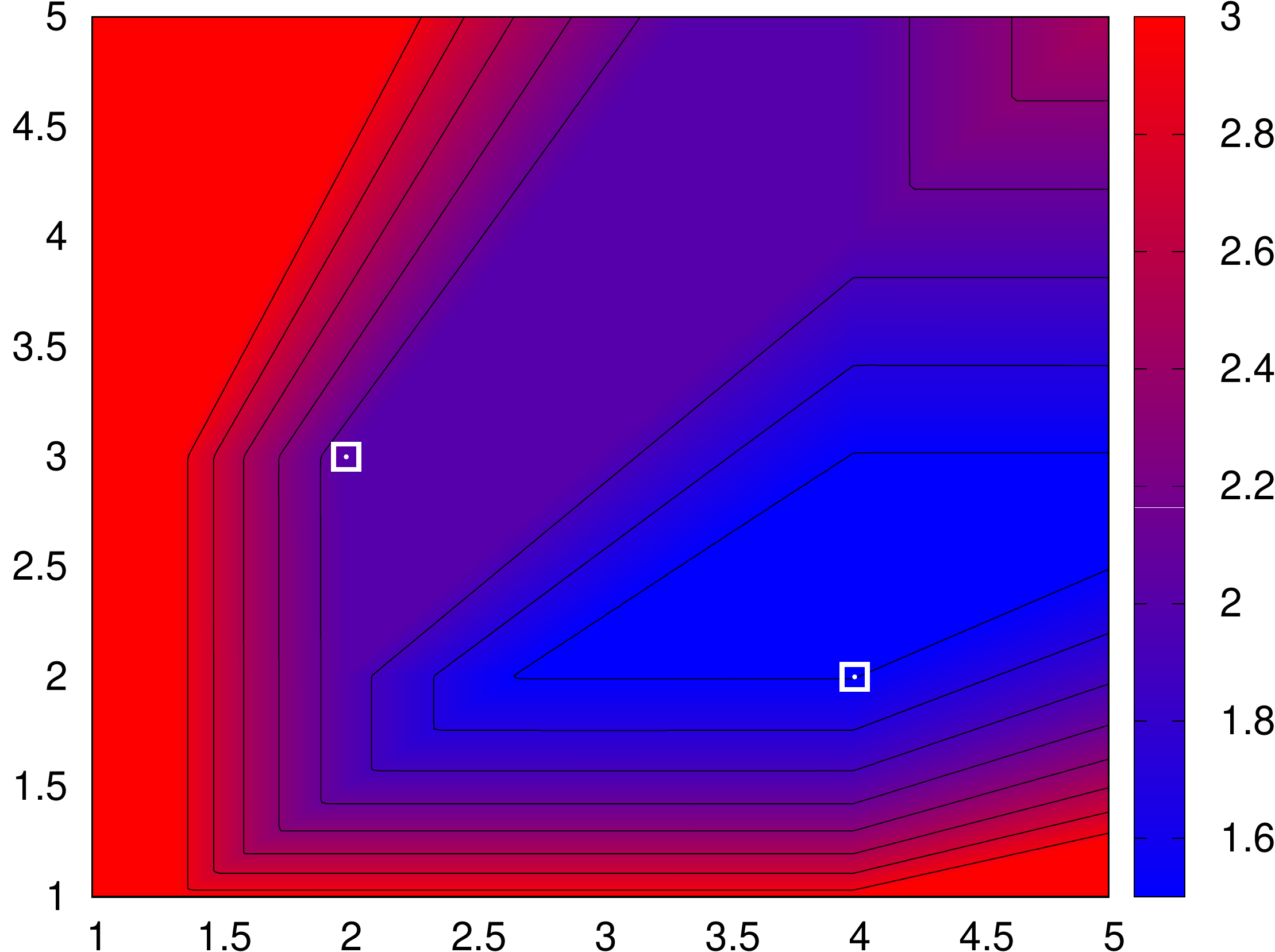}
\caption{Approximation guarantees for a two-stage example problem.}\label{fig:2dex2}
\end{center}
\end{figure}
Note that contour lines are not straight anymore, as scaling scenario $\hat{\pmb{c}}$ does not necessarily result in the same approximation guarantee, because $\alpha$ and $\beta$ are bounded by one. The optimal approximation guarantee is 1.5 and is attained, for example, in $\hat{\pmb{c}}=(4,2)^\tr$ (and also by scenarios in the vicinity of this point). Note that $\hat{\pmb{c}}\in\cU$.

Indeed, it can be seen that there is always an optimal choice in $\cU$. Consider any fixed choice of $\hat{\pmb{c}}$, $\pmb{c}'$, $\alpha,\beta\ge 1$ for which it holds that
\[ \pmb{c}^i \le \alpha \hat{\pmb{c}} \le \alpha\beta\pmb{c}' \qquad \forall i\in[N] \]
In this case, we can set $\alpha' = \alpha\beta$ and $\beta' = 1$ and see that
\[ \pmb{c}^i \le \alpha' \pmb{c}' \le \alpha'\beta' \pmb{c}' \qquad \forall i\in[N] \]
As $\alpha\beta=\alpha'\beta'$, scenario $\pmb{c}'$ thus gives the same approximation guarantee as $\hat{\pmb{c}}$.

We now extend these observations to multiple representative scenarios.

\begin{theorem}\label{th2stage}
Let $\cU=\{\pmb{c}^1,\ldots,\pmb{c}^N\}\subseteq\mathbb{R}^n_+$, and let $\mathcal{C}=\{\hat{\pmb{c}}^1,\ldots,\hat{\pmb{c}}^K\}\subseteq\mathbb{R}^n_+$ be such that
\begin{align*}
& \forall i\in[N]\, \exists k\in[K] \text{ such that } \pmb{c}^i \le \alpha \hat{\pmb{c}}^k \\
& \forall k\in[K]\, \exists i\in[N] \text{ such that } \hat{\pmb{c}}^k \le \beta \pmb{c}^i
\end{align*}
for some $\alpha\ge 1$ and $\beta\ge 1$. Then, an optimal solution to the two-stage robust problem with respect to $\mathcal{C}$ gives an $\alpha\beta$-approximation to the two-stage robust problem with respect to $\cU$.
\end{theorem}
\begin{proof}
Using the same notation as before, we can estimate as follows.
\begin{align}
\max_{\pmb{c}\in\cU} \min_{\pmb{y}\in\X(\hat{\pmb{x}})} \pmb{C}^\tr\hat{\pmb{x}} + \pmb{c}^\tr\pmb{y} &= \pmb{C}^\tr\hat{\pmb{x}} + \max_{\pmb{c}\in\cU} \min_{\pmb{y} \in\X(\hat{\pmb{x}})} \pmb{c}^\tr\pmb{y} \\
& \le \pmb{C}^\tr\hat{\pmb{x}} + \max_{\hat{\pmb{c}}\in\mathcal{C}} \min_{\pmb{y}\in\X(\hat{\pmb{x}})} \alpha \hat{\pmb{c}}^\tr\pmb{y} \\
& \le \alpha \left( \pmb{C}^\tr\hat{\pmb{x}} + \max_{\hat{\pmb{c}}\in\mathcal{C}} \min_{\pmb{y}\in\X(\hat{\pmb{x}})} \alpha \hat{\pmb{c}}^\tr\pmb{y} \right) \\ 
& \le \alpha \left( \pmb{C}^\tr\pmb{x}^* + \max_{\hat{\pmb{c}}\in\mathcal{C}} \min_{\pmb{y}\in\X(\pmb{x}^*)} \hat{\pmb{c}}^\tr\pmb{y} \right) \\ 
& \le \alpha \left( \pmb{C}^\tr\pmb{x}^* + \beta \max_{\pmb{c}\in\cU} \min_{\pmb{y}\in\X(\pmb{x}^*)} \pmb{c}^\tr\pmb{y} \right) \\
& \le \alpha\beta \left( \pmb{C}^\tr\pmb{x}^* +  \max_{\pmb{c}\in\cU} \min_{\pmb{y}\in\X(\pmb{x}^*)} \pmb{c}^\tr\pmb{y} \right)
\end{align}
\end{proof}

We can reformulate the bound from Theorem~\ref{th2stage} to an optimization problem that aims at finding a clustering that gives an optimal approximation guarantee.
\begin{align*}
\min\  & \alpha \cdot \beta\\
\text{s.t. } & \pmb{c}^i \le \alpha \sum_{k\in[K]} \mu_{ik} \hat{\pmb{c}}^k & \forall i\in[N] \\
& \hat{\pmb{c}}^k \le \beta \sum_{i\in[N]} \lambda_{ki} \pmb{c}^i & \forall k\in[K] \\
& \hat{\pmb{c}}^k \in \mathbb{R}^n_+ & \forall k\in[K] \\
& \alpha, \beta \ge 1 \\
& \mu_{ik} \in\{0,1\} & \forall i\in[N], k\in[K] \\
& \lambda_{ki} \in \{0,1\} & \forall k\in[K], i\in[N]
\end{align*}
Consider any solution of this problem, and consider the set of scenarios $\pmb{c}^i$ that are assigned to some $\hat{\pmb{c}}^k$, i.e., the set $\{i\in[N] : \mu_{ik} = 1\}$. Using the same arguments as in the single-scenario case on this set, we find that there exists a solution with the same approximation guarantee where $\hat{\pmb{c}}^k = \sum_{\ell} \lambda_{k\ell}\pmb{c}^\ell$. Hence, an optimal solution exists where $\mathcal{C}\subseteq\cU$. This allows us to reformulate the above problem to the following mixed-integer linear optimization problem:
\begin{align}
\max\ &t \label{ipstart}\\
\text{s.t. } & t \pmb{c}^i \le \sum_{\ell\in[N]} \mu_{i\ell} \pmb{c}^\ell & \forall i\in[N] \label{consimp}\\
& \mu_{i\ell} \le \lambda_\ell & \forall i\in[N], \ell\in[N] \label{ip2}\\
& \sum_{i\in[N]} \lambda_i = K \label{ip3}\\
& \sum_{\ell\in[N]} \mu_{i\ell} = 1& \forall i\in[N] \label{ip4}\\
& \lambda_i \in \{0,1\} & \forall i\in[N] \label{ip5}\\
& \mu_{i\ell} \in\{0,1\} & \forall i\in[N], \ell\in[N]\label{ip6} \\
& t \ge 0 \label{ipend}
\end{align}
Observe that the difference to problem IP-$\lambda$ is that variables $\mu$ are required to be binary. This allows us to further simplify the optimization problem. Consider any choice of $\mu$ variables and some fixed $i$. Let $\ell$ be the index for which $\mu_{i\ell} = 1$. Constraints~\eqref{consimp} then become 
$t c^i_j \le c^\ell_j$ for all $j\in[n]$. We would like to increase $t$ as far as possible. Let $d_{i\ell} = \min_{j\in[n]} c^\ell_j / c^i_j$. Constraints~\eqref{consimp} can thus be substituted by
\begin{equation}\label{conbetter}
t \le \sum_{\ell\in[N]} d_{i\ell}\mu_{i\ell} \qquad \forall i\in[N]
\end{equation}
This way, we do not only reduce the number of constraints. A further advantage is that we can now relax variables $\mu$: there is always an optimal solution where the largest amongst the $d_{i\ell}$ values is assigned to scenario $i\in[N]$. Note that this is not the case for formulation~(\ref{ipstart}-\ref{ipend}). The problem
\begin{align*}
\max\ &t \\
\text{s.t. } & (\ref{ip2}-\ref{ip5}), \eqref{ipend}, \eqref{conbetter} \\
& \mu_{i\ell} \ge 0 & \forall i\in[N], \ell\in[N]\\
\end{align*}
is referred to simply as IP in the following.

As in the one-stage case, it is possible to calculate approximation guarantees that are independent of the underlying optimization problem. In Figure~\ref{fig:ex2st} we calculate the resulting average approximation guarantees using the same uniformly and normally distributed data as was used in Figure~\ref{fig:ex1st}. 
\begin{figure}[htb]
\begin{center}
\subfigure[Uniform data.]{\includegraphics[width=.48\textwidth]{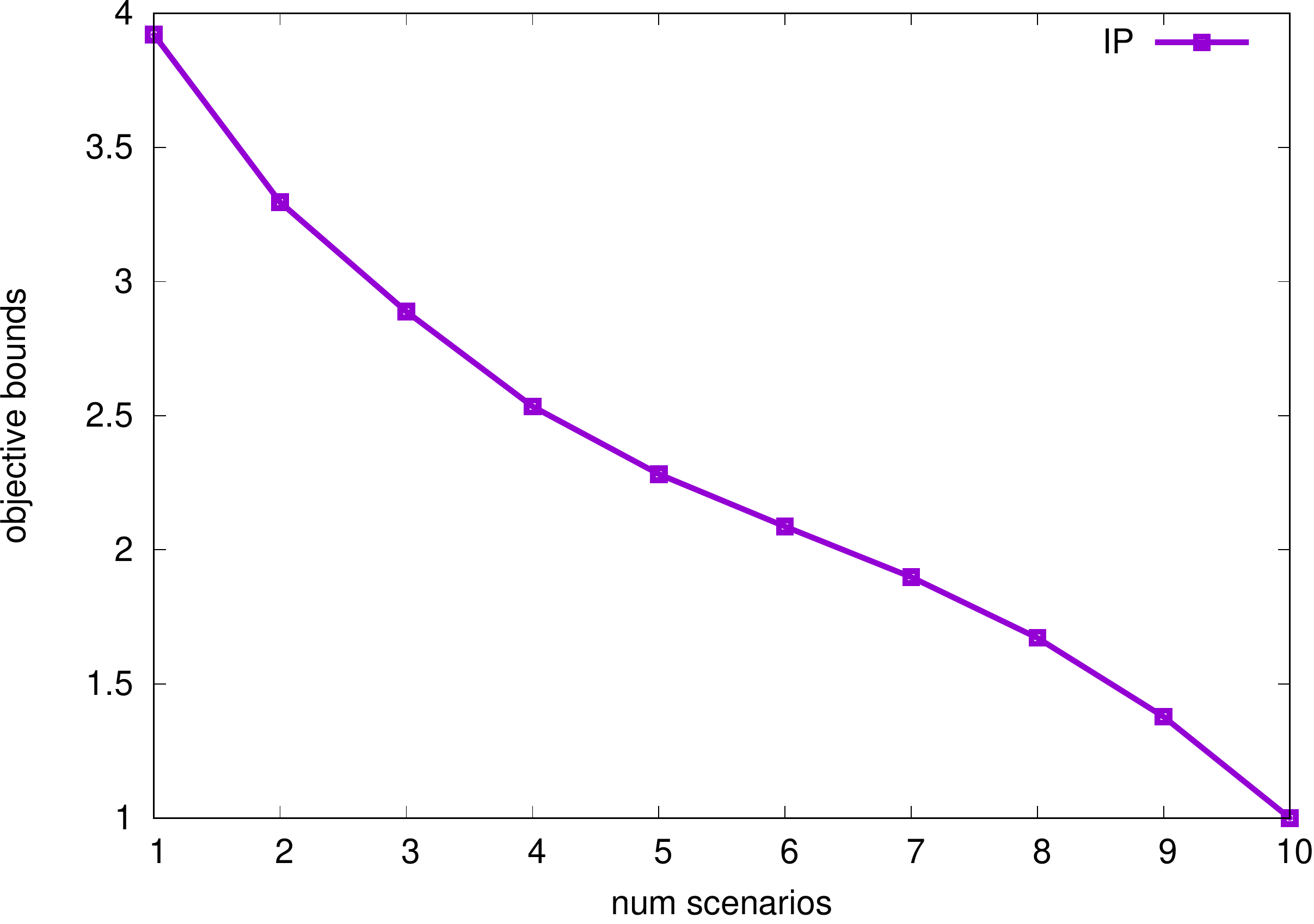}}\hfill
\subfigure[Normal data.]{\includegraphics[width=.48\textwidth]{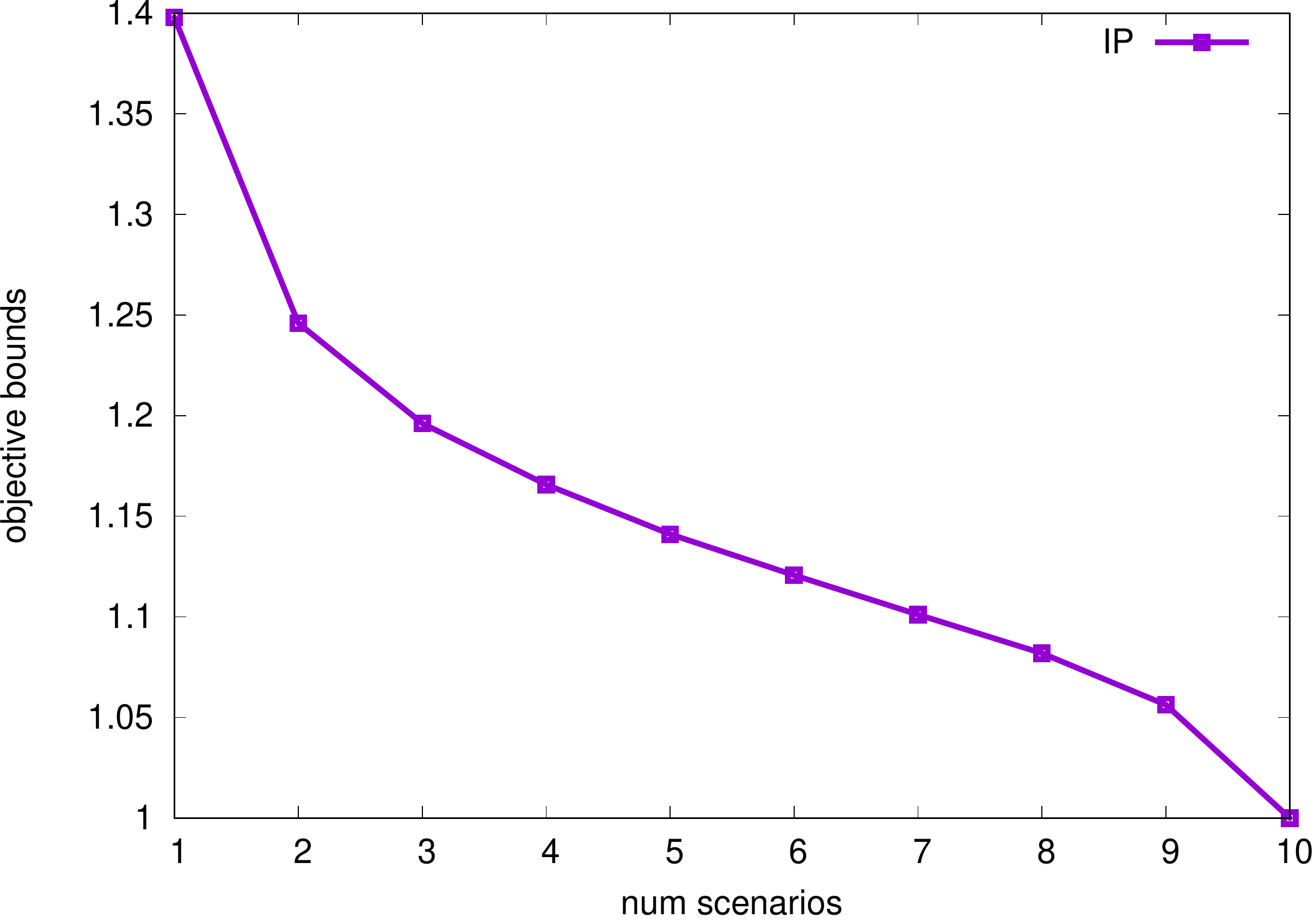}}
\end{center}
\caption{Approximation ratios for randomly generated data, two-stage problems.}\label{fig:ex2st}
\end{figure}
As any solution feasible to IP is also feasible to IP-$\lambda$ and Cont, it is natural that approximation guarantees have become larger. This reflects the intuition that two-stage problems are harder to approximate than one-stage problems. Still, guarantees remain meaningful, in particular for normally distributed data and when considering that no algorithm with any a-priori guarantee exists. Interestingly, the reduced degrees of freedom in the design of scenario clusterings also mean the clustering problems are easier to solve. We conclude this section with the observation that the clustering problem remains NP-hard, though; the proof of this result is found in the appendix.

\begin{theorem}
The decision version of IP is NP-complete.
\end{theorem}

\section{Experiments}
\label{sec:experiments}

We conduct two types of experiments. In the first experiment, we compare the quality of aggregation when using Cont with that of K-means on different types of data. Here, quality is measured as correlation with robust objective values with respect to the original uncertainty set. In the second experiment, we consider specific optimization problems and assess the quality of aggregation using the resulting robust solutions.

\subsection{Experiment 1}

The reduced scenario set $\C$ is supposed to represent the set $\cU$ as well as possible for robust optimization problems with the worst-case objective. Hence, we would like the worst-case objective value $\max_{\pmb{c}\in\C} \pmb{c}^\tr\pmb{x}$ to be as closely correlated to $\max_{\pmb{c}\in\cU} \pmb{c}^\tr\pmb{x}$ as possible for any $\pmb{x}\in\mathbb{R}^n_+$. Observe that it is not relevant if the objective values with respect to $\C$ and with respect to $\cU$ are similar, as we can scale an uncertainty set with an arbitrary factor $\lambda>0$ and find the same robust solution. If the objective value with respect to $\C$ is perfectly correlated to the original objective value, then we will rank the performance of robust solutions in the same way.

In this experiment, we fix $n=10$, $N=100$ and $K=5$, that is, our aim is to reduce a problem with 100 scenarios down to 5 scenarios. We use Cont (with the best out of 10 repetitions and a limit of 20 on the number of iterations per repetition) as well as K-means (with the best out of 1000 repetitions). We consider four different types of scenarios: In $\cU_1$, each scenario coefficient $c^i_j$ is chosen uniformly and independently from $\{1,\ldots,100\}$. For $\cU_2$, we first create scenarios in the same way, but each scenario has a $5\%$ probability to be multiplied with a factor 2. This way, we create a uniform uncertainty set that contains some outliers. For $\cU_3$, we follow the idea of budgeted uncertainty. Each item $j$ has a nominal cost $c_j$ and a deviation $d_j$ chosen uniformly from $\{1,\ldots,100\}$. To generate a scenario, each item is given its nominal cost $c_j$, except for three items chosen at random, which have costs $c_j+d_j$. Finally, for $\cU_4$ we create a uniform uncertainty set as in $\cU_1$, but normalize all scenarios with respect to the 2-norm, so that they lie on a sphere. We then multiply each scenario with a random value in $[0.9,1.1]\cdot10^4$ to create $\cU_4$.

For each type, we create 50 uncertainty sets. For each set, we sample 100 random vectors $\pmb{x}\in\mathbb{R}^n_+$, by choosing $x_j\in[0,1]$ uniformly and independently. We calculate the corresponding robust objective values for the original set, for Cont, and for K-means, which gives a total of 5000 data points for each method.

In Figure~\ref{fig:corr}, we show the resulting robust objective values. Note that both Cont and K-means choose scenarios in the convex hull of $\cU$. Hence, objective values will always be not larger than objective values with respect to $\cU$.  Underneath each plot, we give the Pearson correlation value $\rho$.

\begin{figure}[htbp]
\begin{center}
\subfigure[$\cU_1$, Cont, $\rho=98.2\%$\label{fig:exp1a}]{\includegraphics[width=0.248\textwidth]{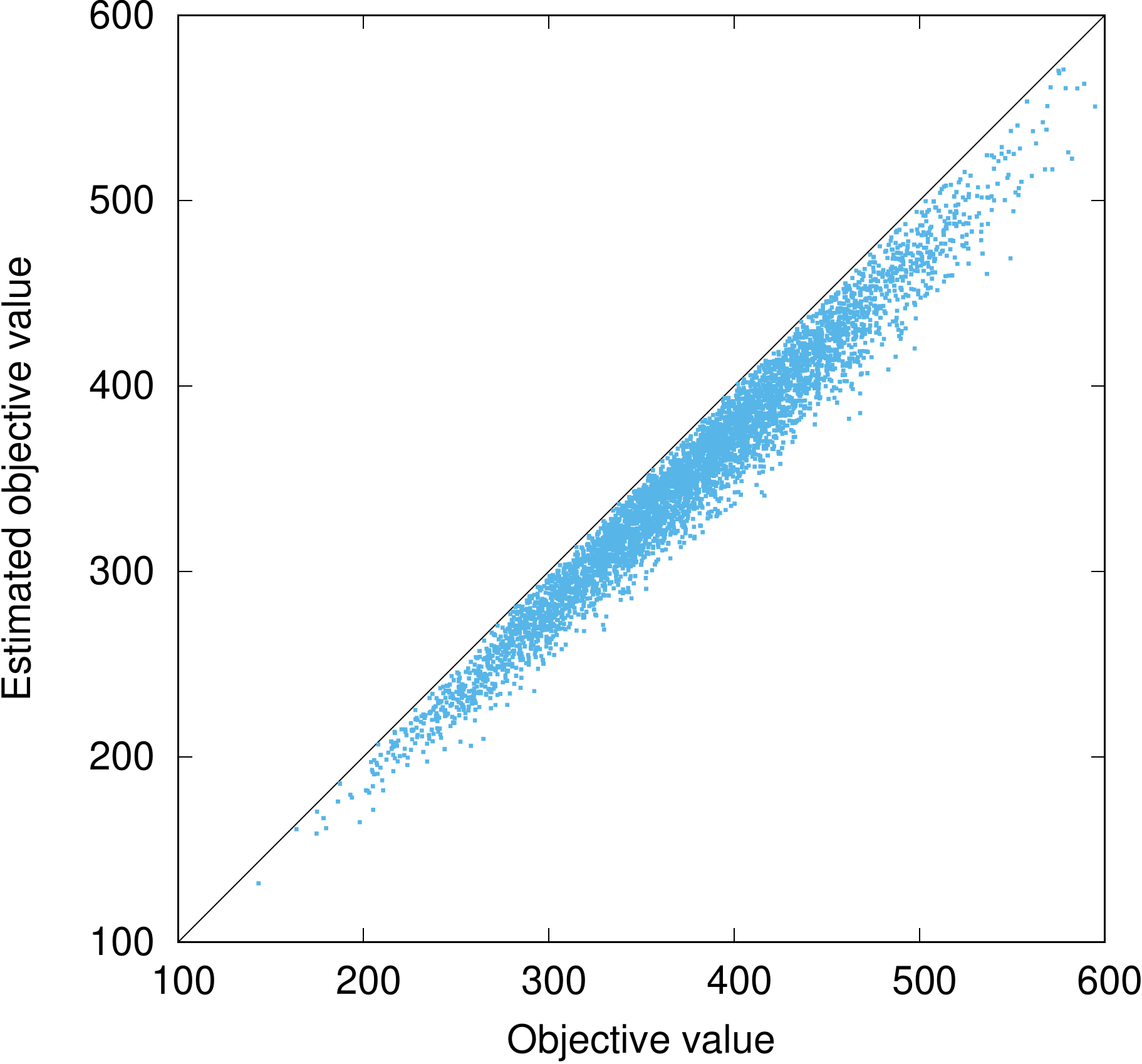}}%
\subfigure[$\cU_1$, KM, $\rho=95.5\%$\label{fig:exp1b}]{\includegraphics[width=0.248\textwidth]{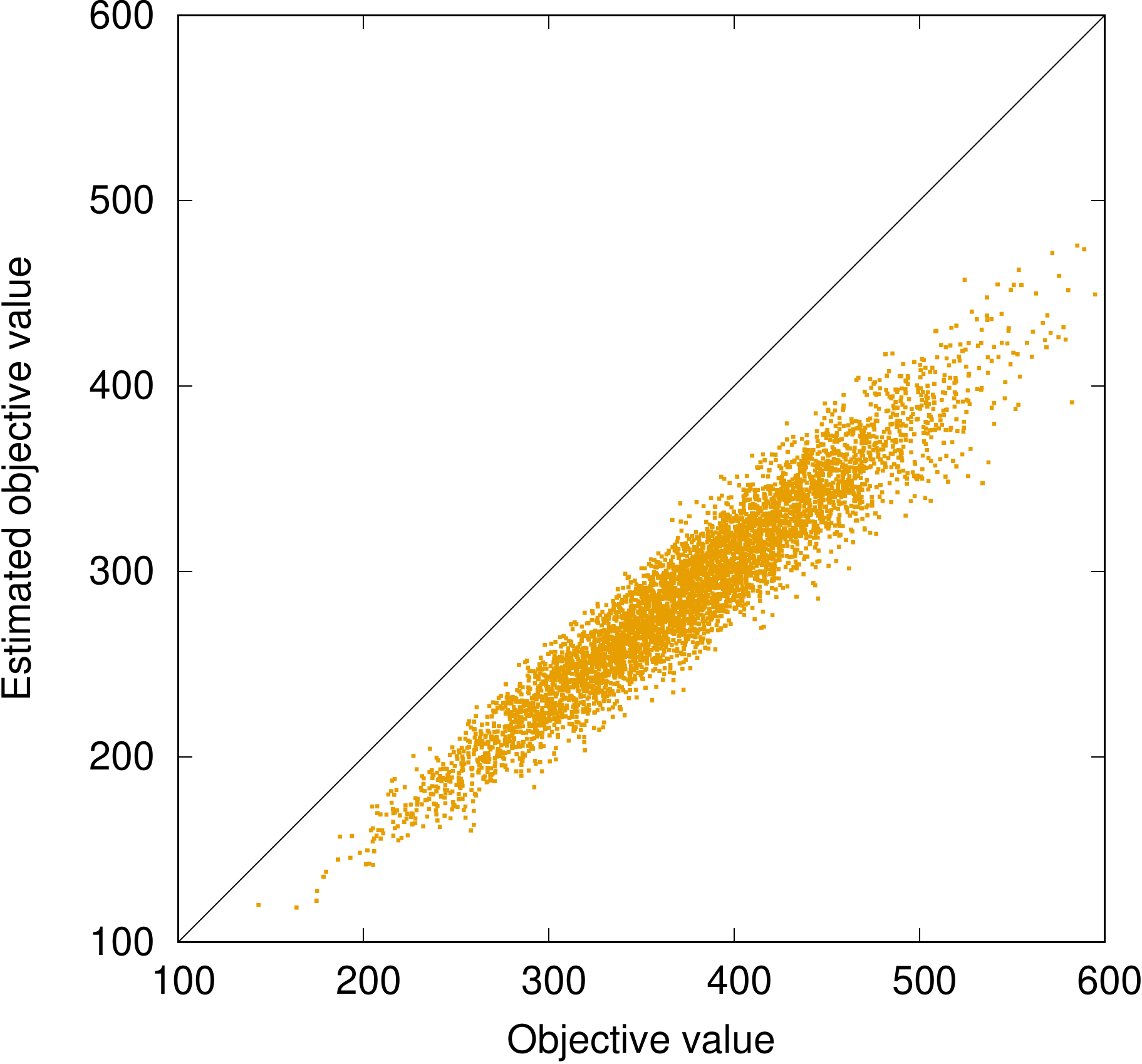}}
\subfigure[$\cU_2$, Cont, $\rho=99.1\%$\label{fig:exp1c}]{\includegraphics[width=0.248\textwidth]{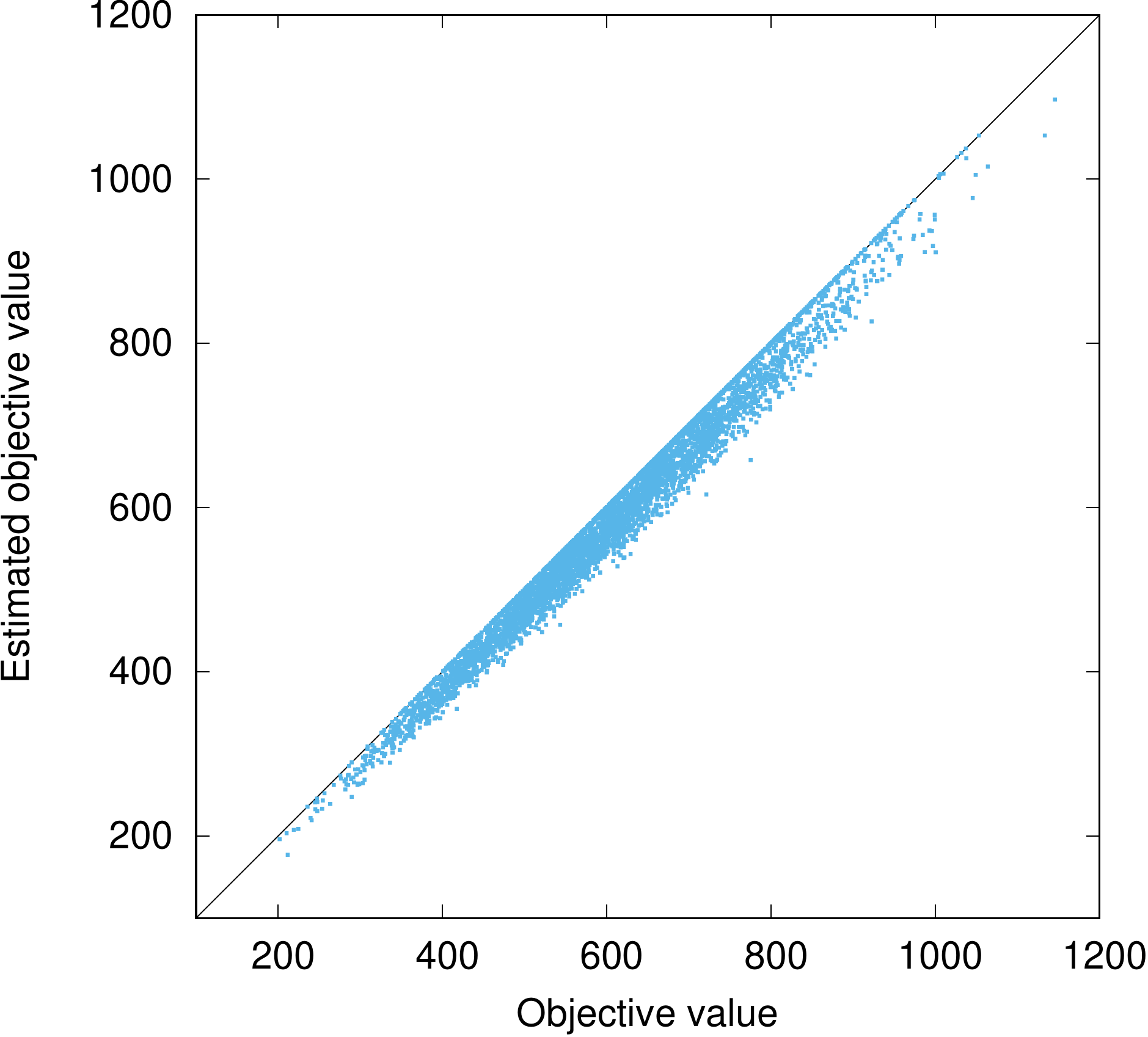}}%
\subfigure[$\cU_2$, KM, $\rho=86.1\%$\label{fig:exp1d}]{\includegraphics[width=0.248\textwidth]{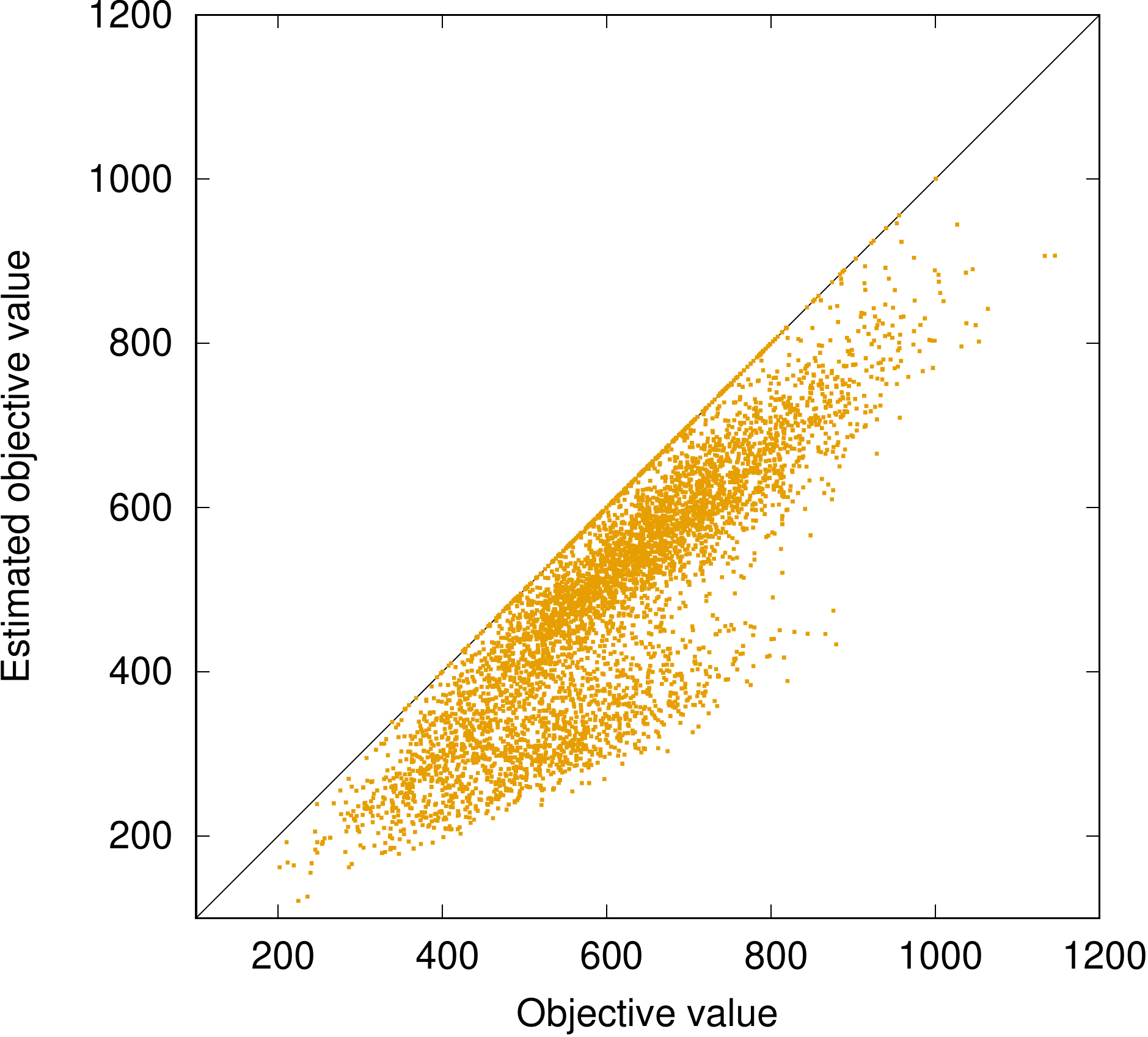}}

\subfigure[$\cU_3$, Cont, $\rho=98.6\%$\label{fig:exp1e}]{\includegraphics[width=0.248\textwidth]{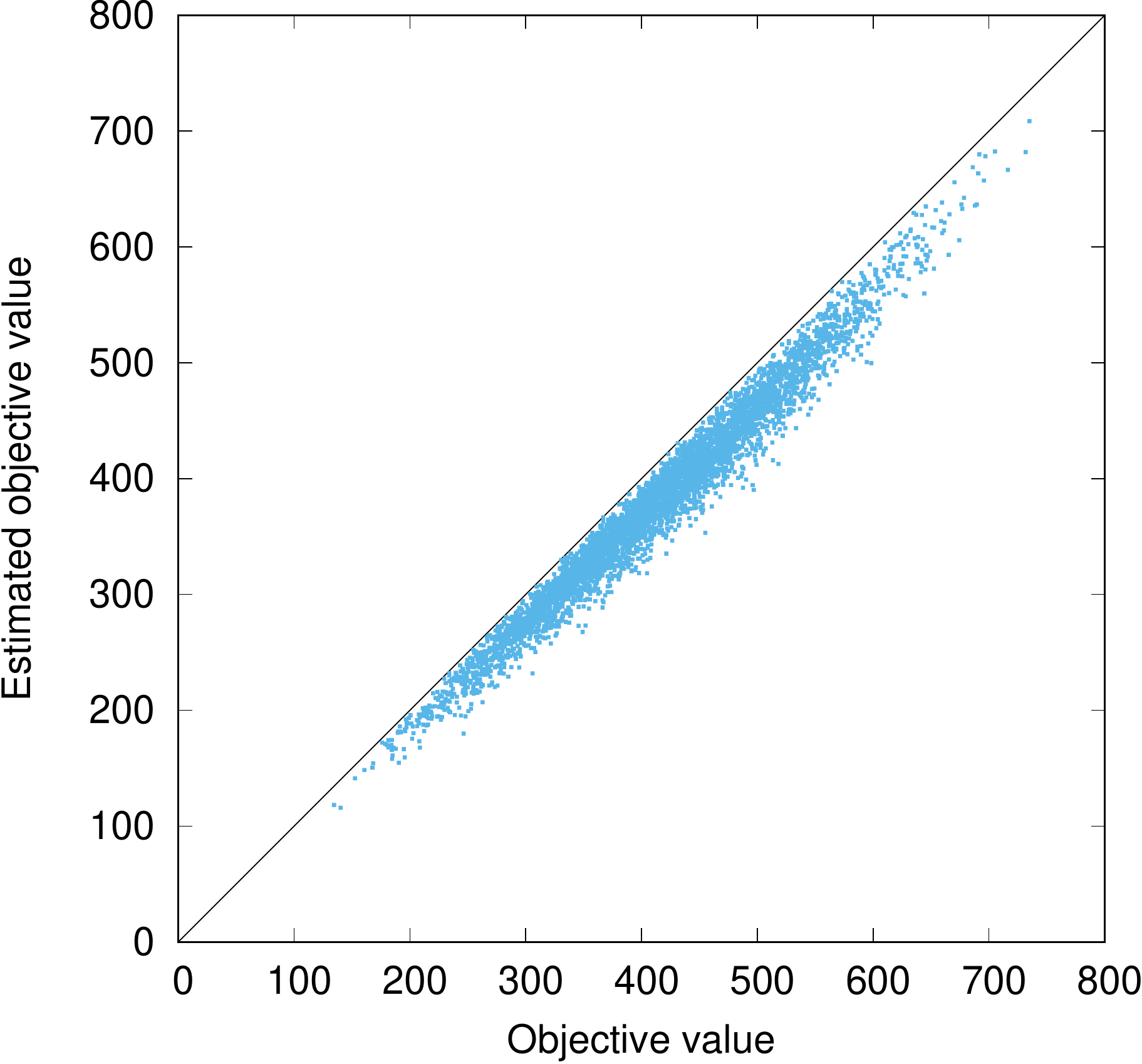}}%
\subfigure[$\cU_3$, KM, $\rho=98.2\%$\label{fig:exp1f}]{\includegraphics[width=0.248\textwidth]{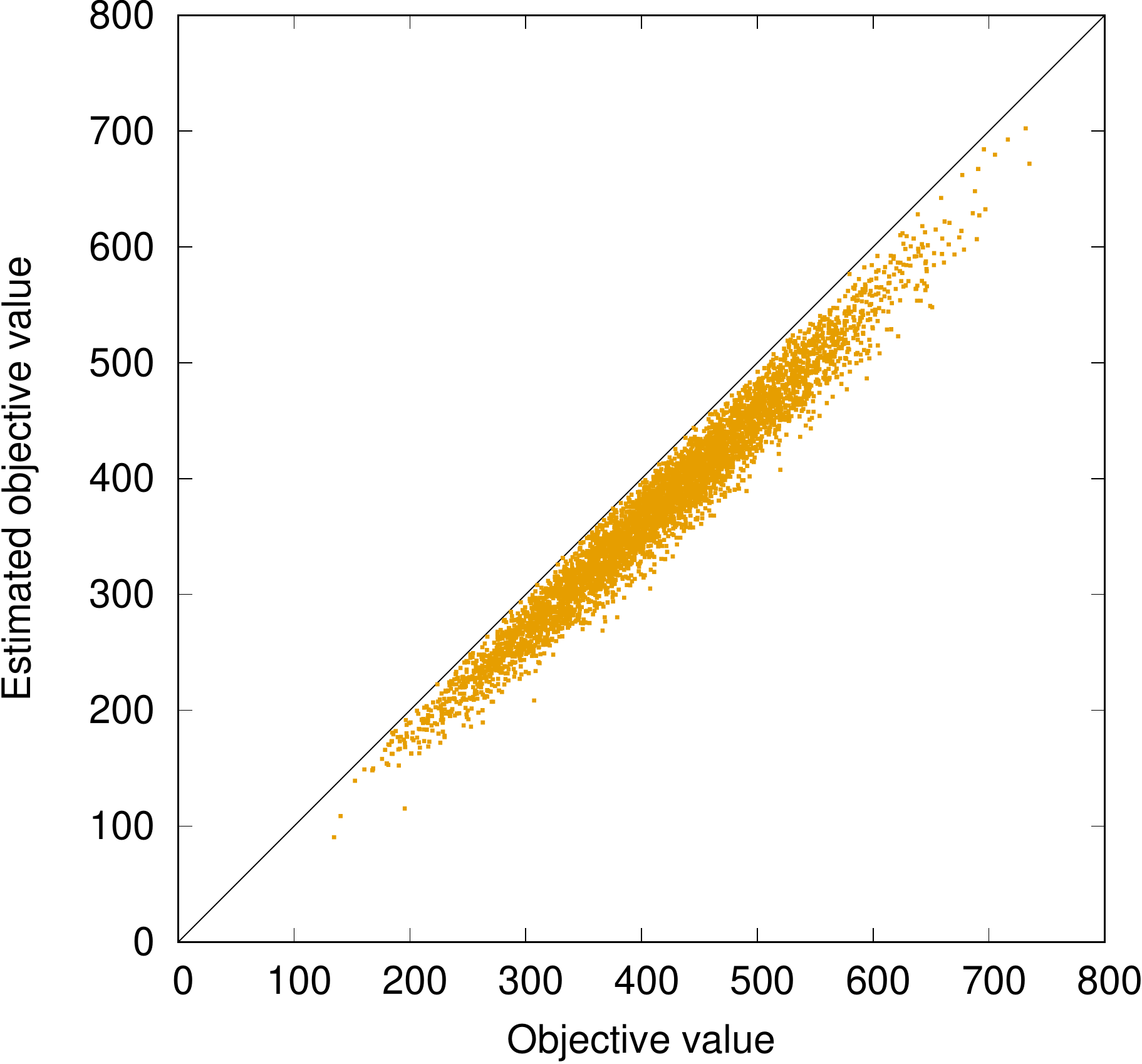}}
\subfigure[$\cU_4$, Cont, $\rho=97.0\%$\label{fig:exp1g}]{\includegraphics[width=0.248\textwidth]{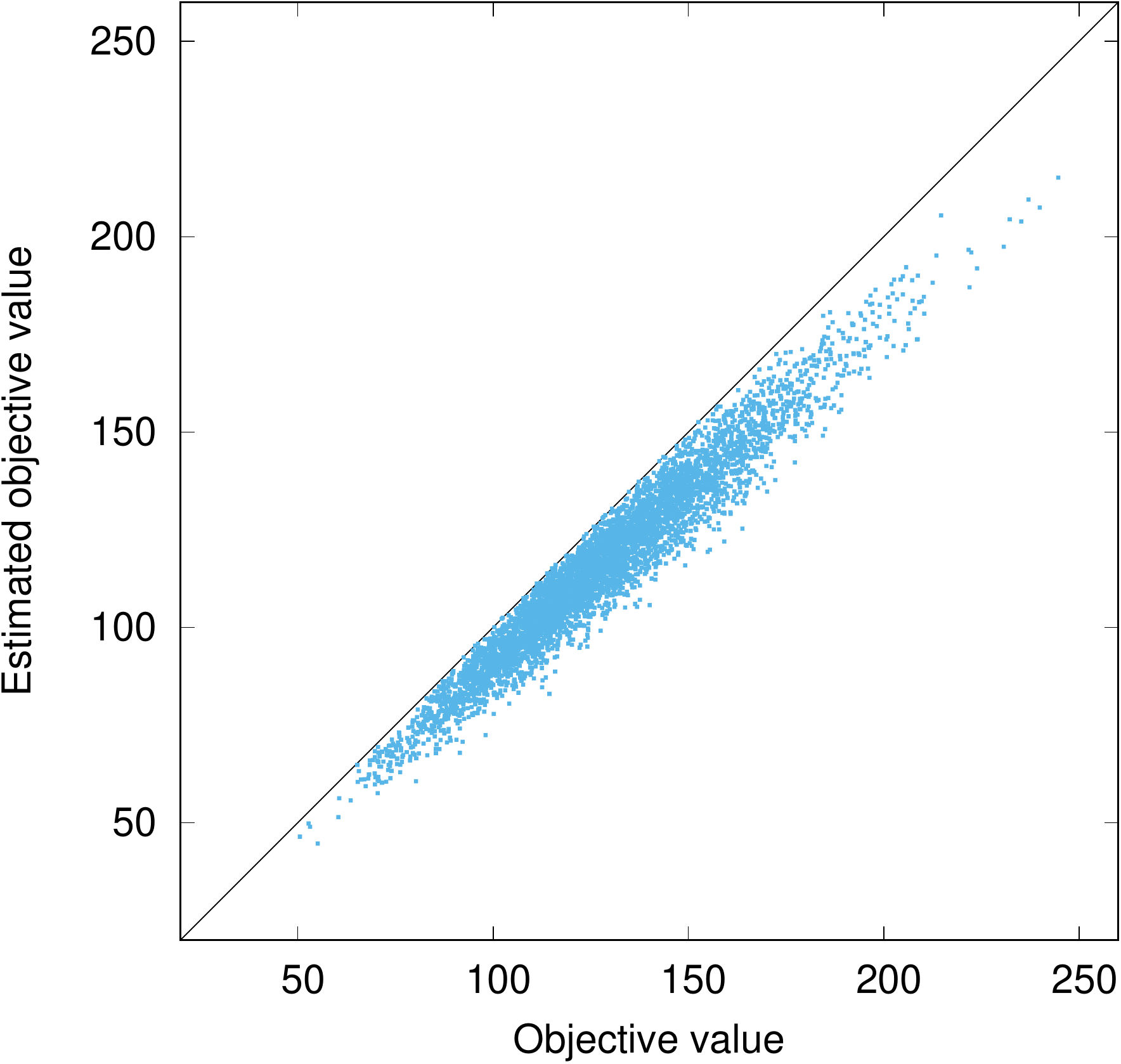}}%
\subfigure[$\cU_4$, KM, $\rho=83.0\%$\label{fig:exp1h}]{\includegraphics[width=0.248\textwidth]{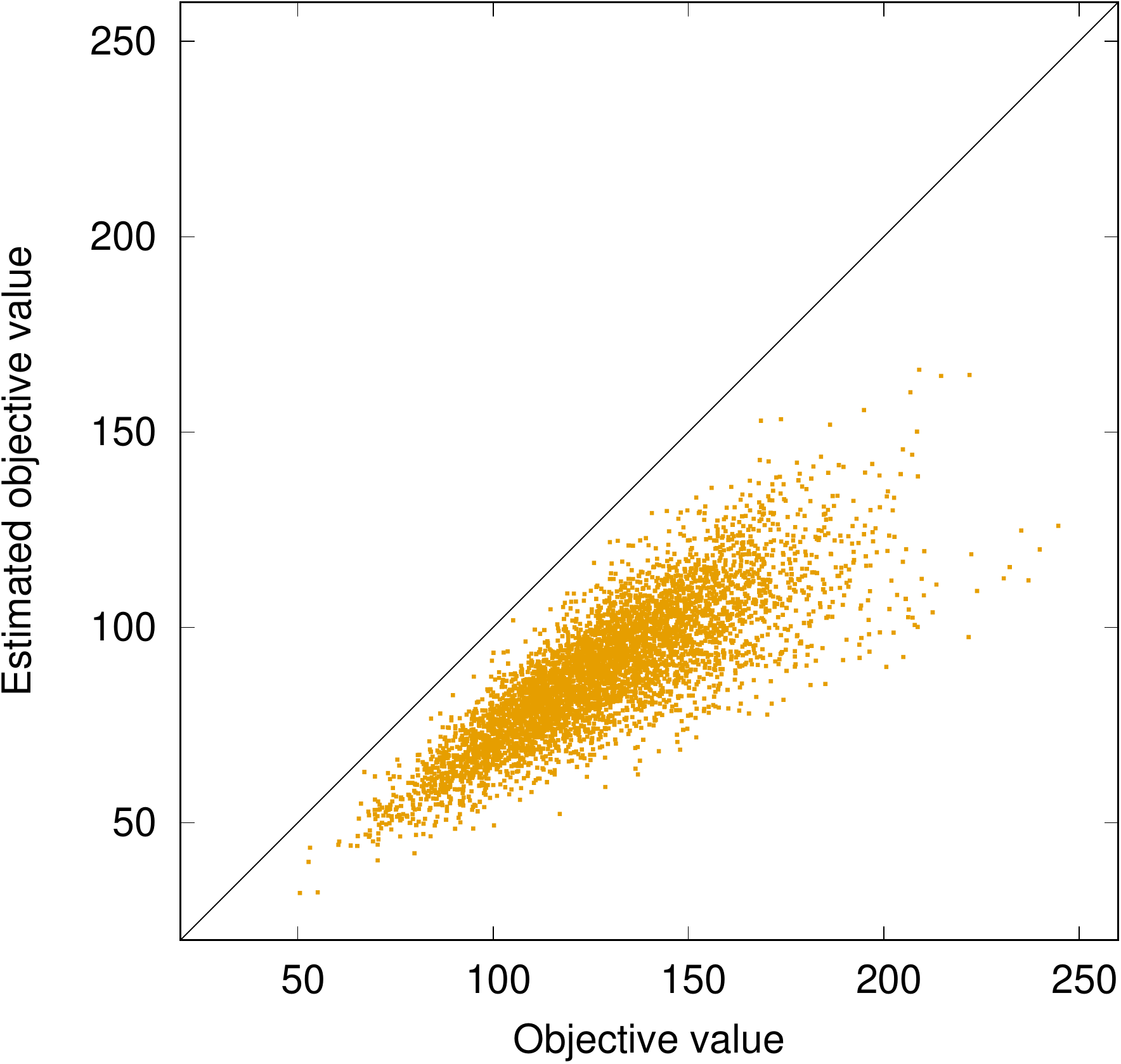}}
\caption{Experiment 1, correlation between estimated and actual objective values. KM stands for K-means.}\label{fig:corr}
\end{center}
\end{figure}

First consider Figures~\ref{fig:exp1a} and \ref{fig:exp1b} which show the case of uniform uncertainty sets $\cU_1$. Using Cont, we reach a correlation of $98.2\%$, while K-means results in a correlation of $95.5\%$. Both methods give a reasonable approximation of robust costs, while the spread for K-means is larger, and performance thus slightly worse. If we add in outliers (see Figures~\ref{fig:exp1c} and \ref{fig:exp1d}), we observe that the performance of Cont becomes better (increasing the correlation to $99.1\%$), as these outliers tend to dominate the calculation of robust objective values. At the same time, performance of K-means declines down to a correlation of $86.1\%$.

Budgeted uncertainty sets $\cU_3$ can be well-approximated with both methods (see Figures~\ref{fig:exp1e} and \ref{fig:exp1f}) with correlation values $98.6\%$ for Cont and $98.2\%$ for K-means, respectively. Finally, consider scenarios along the 2-ball which we generate for $\cU_4$ (see Figures~\ref{fig:exp1g} and \ref{fig:exp1h}). Intuitively, such scenarios should be hard to approximate, as they are spread out evenly and of similar size. Indeed, we find that Cont and K-means have the lower correlation values amongst all uncertainty set with $97.0\%$ for Cont and $83.0\%$ for K-means.

These results show that Cont gives a better approximation of robust objective values than K-means on all settings that we considered. In some cases ($\cU_1$ and $\cU_3$), both methods give good results. Our method becomes particularly useful if there are outliers in the data ($\cU_1$) or if the data is structured along the 2-ball ($\cU_4$). These results are independent of the underlying optimization problem.

\subsection{Experiment 2}

\subsubsection{Setup}

While our scenario reduction methods give guarantees on the performance of the resulting solutions (see Figures~\ref{fig:ex1st} and \ref{fig:ex2st}), it is possible (and likely) that the actual objective values of the resulting solutions are closer to the optimal objective values than the guarantees suggest. To this end, we conduct experiments on one- and two-stage robust optimization problems to analyze the performance of the resulting robust solutions in more detail.

Scenario data $c^i_j$ is always generated independently and uniformly from $\{1,\ldots,100\}$. For the two-stage problems, we generate first-stage costs in the same way. Note that we do not include outliers in this experiment, which would likely give our scenario reduction methods a significant advantage over K-means according to the results of the previous experiment.

Our methods can be applied to any set of feasible solutions $\X$, provided that all variables remain non-negative. The first problem we consider is the selection problem where the set of feasible solutions is given as
\[ \X = \{ \pmb{x}\in\{0,1\}^n : \sum_{i\in[n]} x_i = p\} \]
for some integer $p$. We always set $p=n/2$. The second problem is the vertex cover problem, where we are given an undirected graph $G=(V,E)$ and want to choose a subset $S$ of nodes with minimum sum of node costs such that each node $i\in V$ is either in $S$ or has a neighbor in $S$. That is, the set of feasible solutions is given as
\[ \X = \{ \pmb{x}\in\{0,1\}^n : x_i + \sum_{ \{i,j\} \in E} x_j \ge 1 \ \forall i\in V\} \]
To generate graphs, we create each of the possible $n(n-1)/2$ edges independently with a probability of $10/n$, such that on average, each node has approximately 10 neighbors. We consider the selection and the vertex cover problem, as the former has only a single constraint and therefore allows arbitrary item combinations, while the latter problem has more constraints and thus has more complex interactions between items regarding feasibility.

We consider smaller ($n=20$) as well as larger ($n=150$) problems in combination with fewer ($N=10$) and more ($N=50$) scenarios. For each combination of $n$ and $N$ and each type of optimization problem, we generate 250 instances and always average results. Each instance is solved using all scenarios to find the optimal objective value (or an estimate due to time limits). Additionally, we use $K$ from 1 to 10 to calculate approximate solutions with the methods proposed in this paper. For each solution calculated this way, we evaluate the actual objective value with respect to the full uncertainty set and form the ratio between this value and the optimal value. As a comparison method, we use K-means, as it is one of the most widely applied clustering methods.

Our code is implemented in C++ and available online\footnote{\url{https://github.com/goerigk/robust-scenario-reduction-code/}}. All optimization problems are solved using CPLEX 12.8. Each optimization problem in CPLEX is limited to a time limit of 60 seconds. We use an Intel Xeon Gold 5220 CPU running at 2.20GHz, and restrict each method to a single thread.

\subsubsection{Results on One-Stage Robust Optimization}
\label{subsec:expone}

We compare the performance of our methods IP-$\mu$, IP-$\lambda$ and Cont with that of a K-means clustering approach. Cont is repeated 10 times, while K-means is repeated 1000 times. After the repetitions, the best result is used (according to the approximation guarantee when using Cont, and sum of squared distances when using K-means). To limit the computation time of Cont, we restrict reach repetition to three iterations (i.e., from the random starting solution, we optimize three times for $\mu$ and three times for $\lambda$).

We first show aggregation times in Figure~\ref{fig:1st-agg}. Note the logarithmic vertical axis. As K-means avoids the solution of any optimization problem with CPLEX, it has the clear advantage of being the fastest clustering method. IP-$\mu$ is usually slower than IP-$\lambda$. While Cont has the advantage that it only solves linear programs, it needs to do so repeatedly, which leads to aggregation times that are close to 100 seconds on average on the very largest instances.

\begin{figure}[htbp]
\begin{center}
\subfigure[$n=20$, $N=10$]{\includegraphics[width=0.3\textwidth]{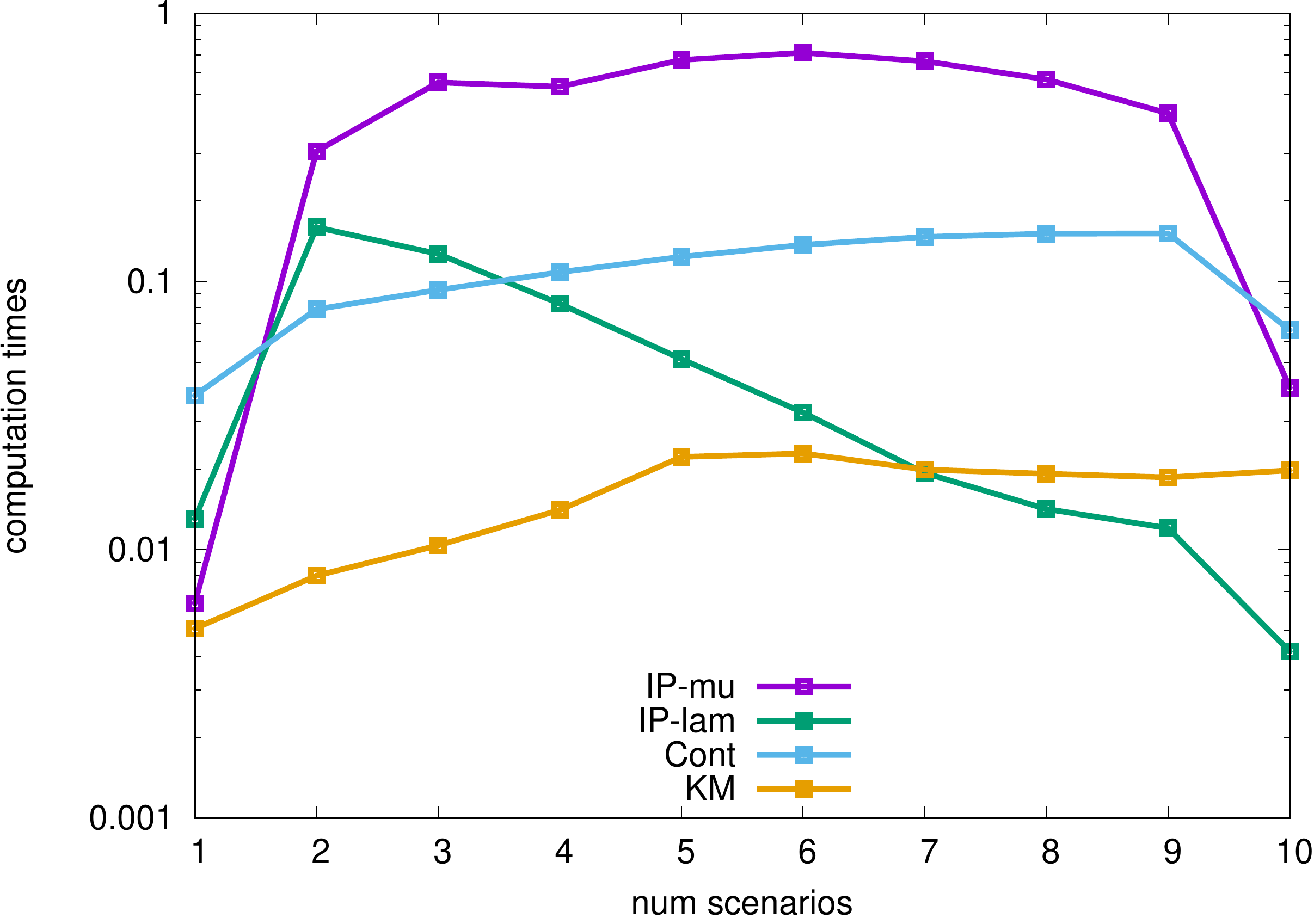}}%
\subfigure[$n=150$, $N=10$]{\includegraphics[width=0.3\textwidth]{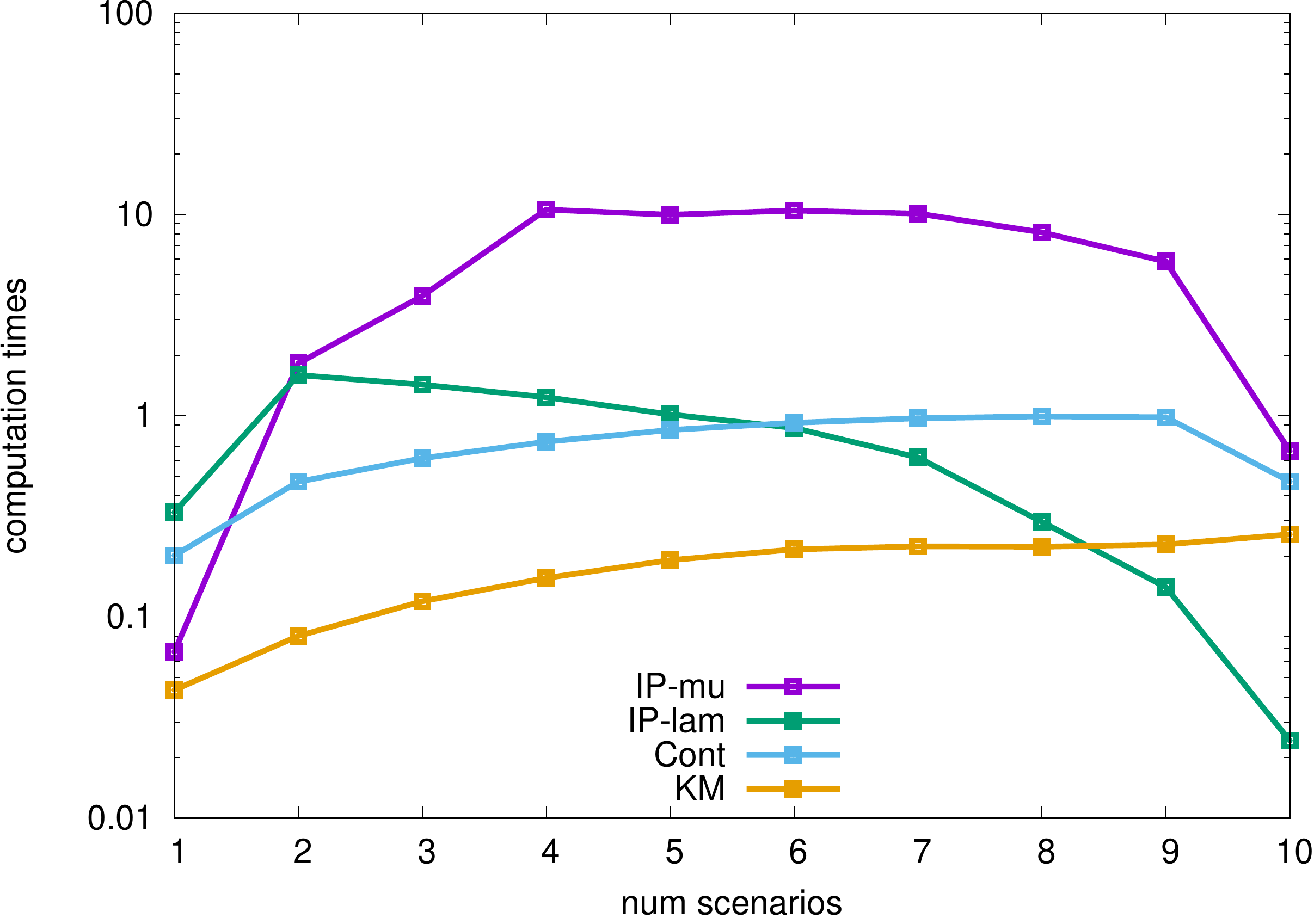}}
\subfigure[$n=20$, $N=50$]{\includegraphics[width=0.3\textwidth]{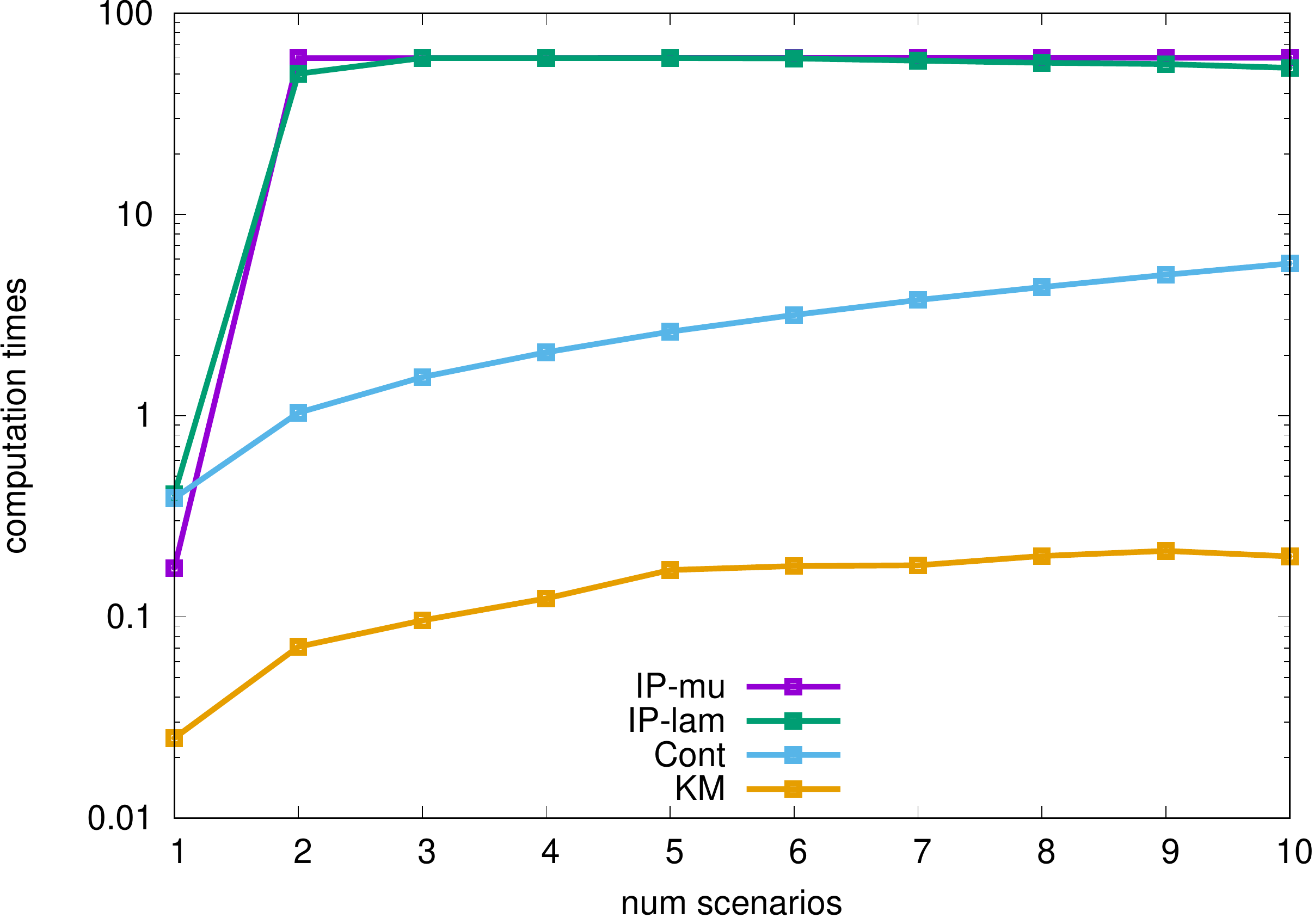}}%
\subfigure[$n=150$, $N=50$]{\includegraphics[width=0.3\textwidth]{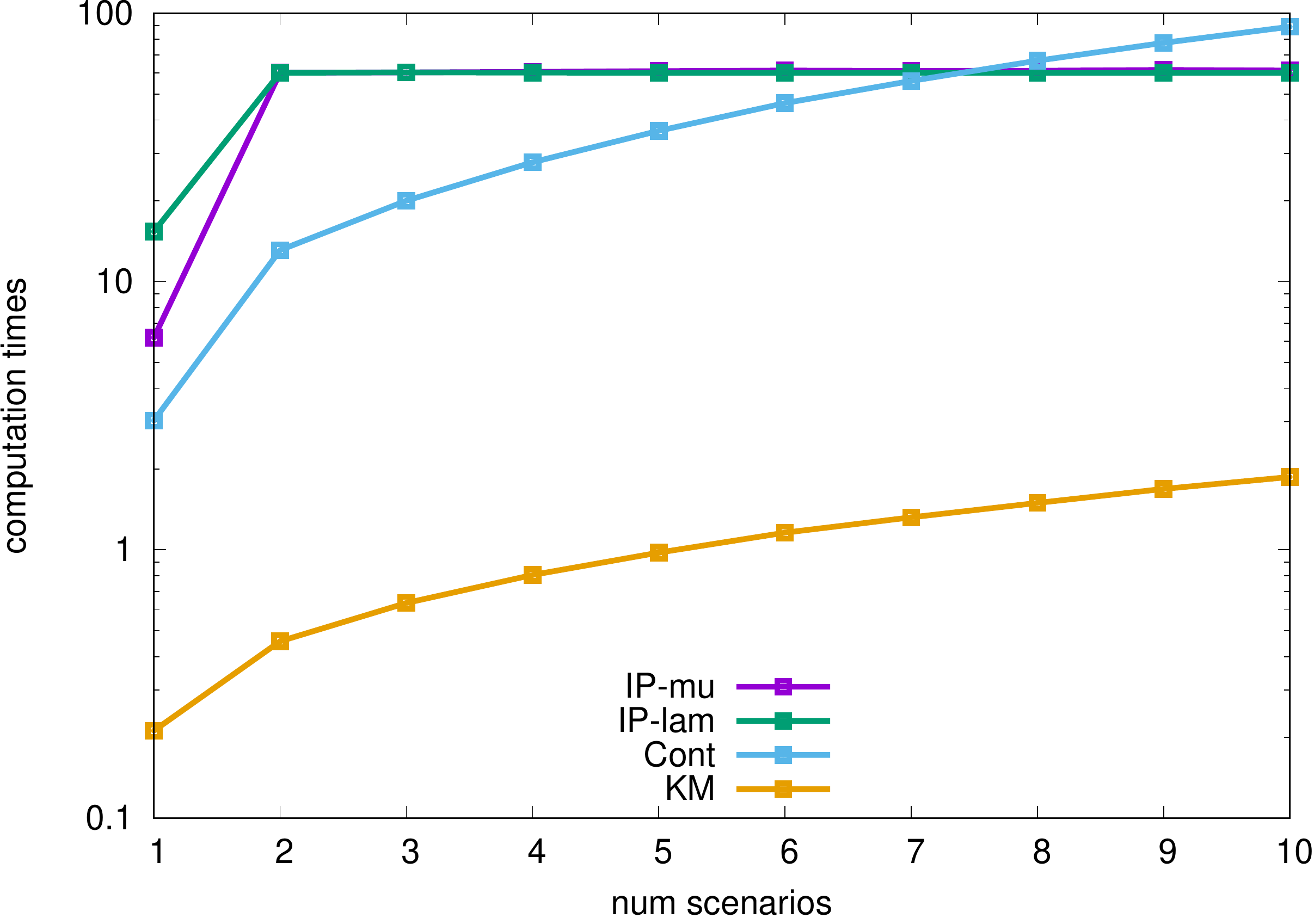}}
\end{center}
\caption{One-stage aggregation times.}\label{fig:1st-agg}
\end{figure}

We now consider the average ratio between the objective value of the robust solution found after scenario reduction, and the optimal robust solution with respect to the original uncertainty set. In Figure~\ref{fig:1st-obj-sel}, we present these values for the selection problem. For problems with $N=10$, IP-$\mu$ and Cont give better results than K-means the majority of cases. Note that Cont is sometimes outperformed by IP-$\mu$ or IP-$\lambda$, which is due to the iteration limit for each run of the method. On problems with $N=50$, there is a distinct advantage of Cont over K-means that increases with the number of cluster centers $K$.

\begin{figure}[htbp]
\begin{center}
\subfigure[$n=20$, $N=10$]{\includegraphics[width=0.3\textwidth]{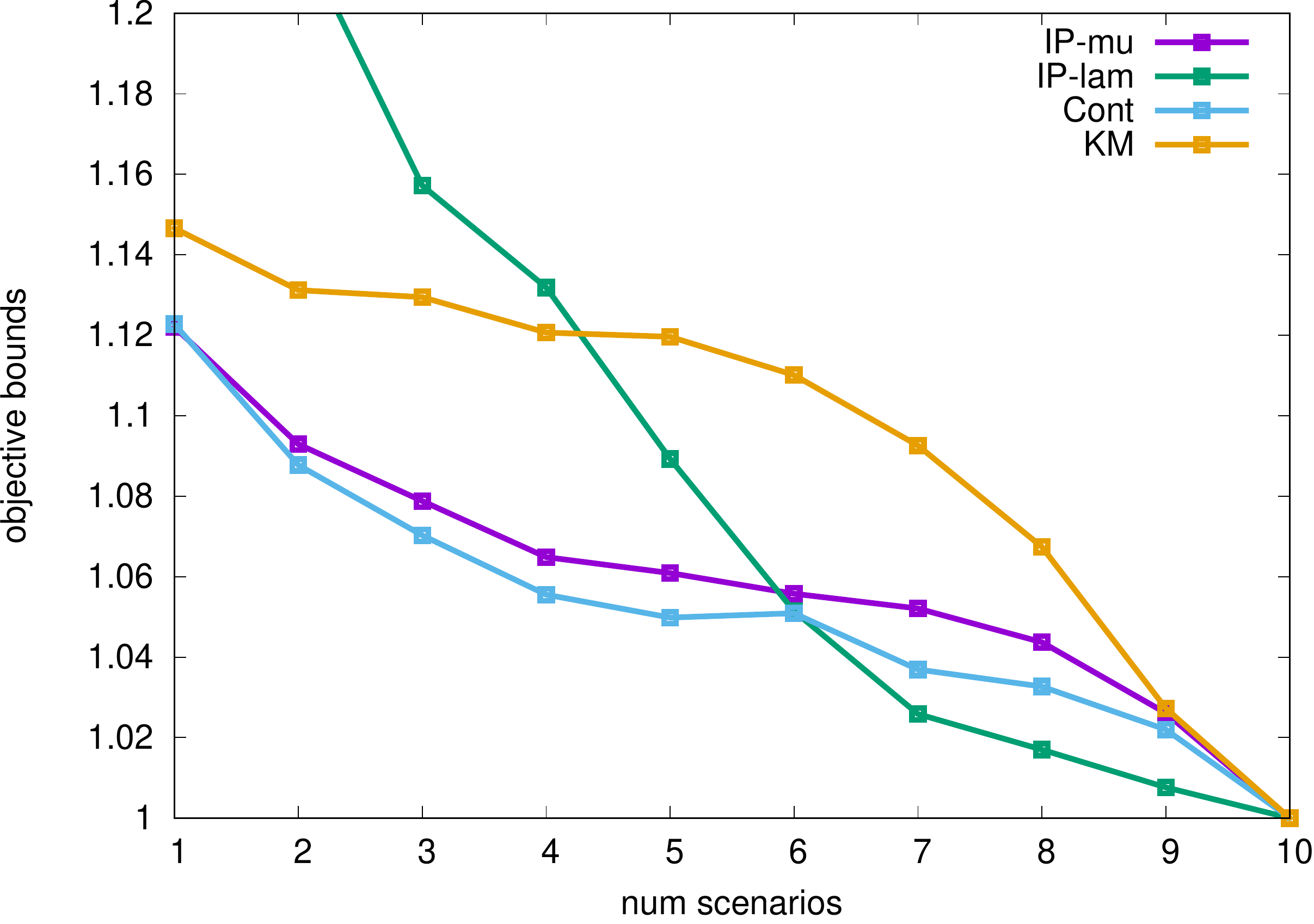}}%
\subfigure[$n=150$, $N=10$]{\includegraphics[width=0.3\textwidth]{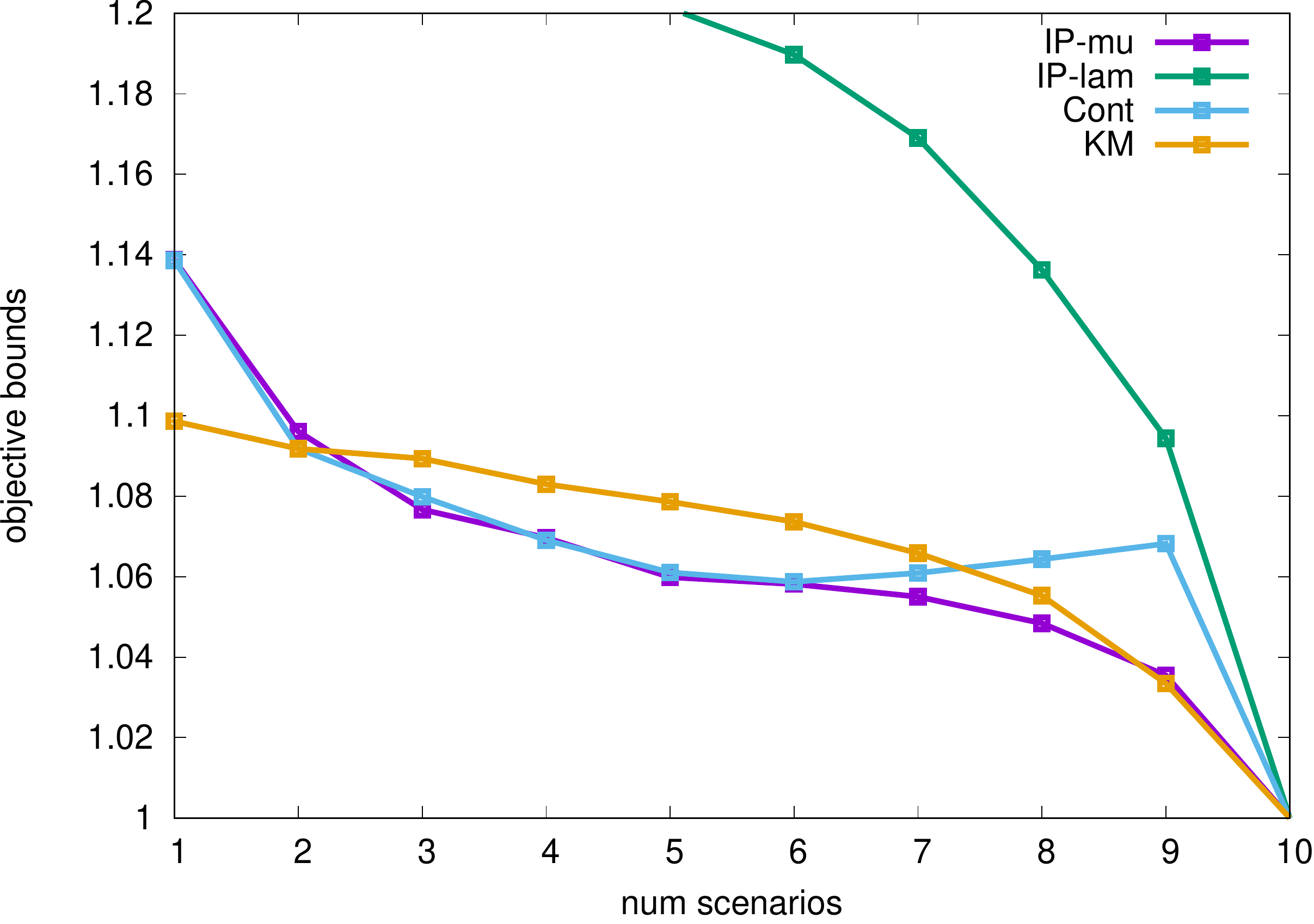}}
\subfigure[$n=20$, $N=50$]{\includegraphics[width=0.3\textwidth]{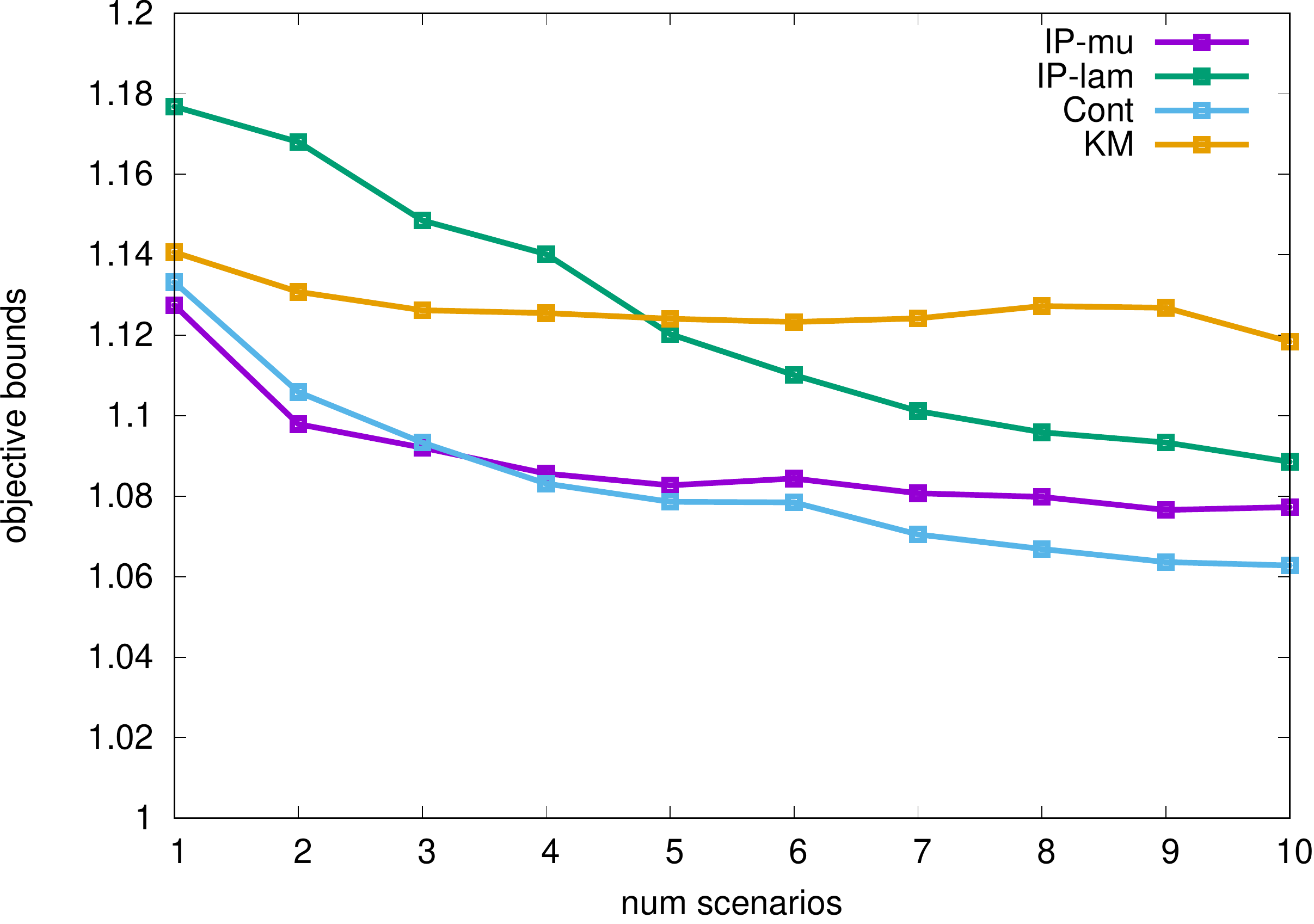}}%
\subfigure[$n=150$, $N=50$]{\includegraphics[width=0.3\textwidth]{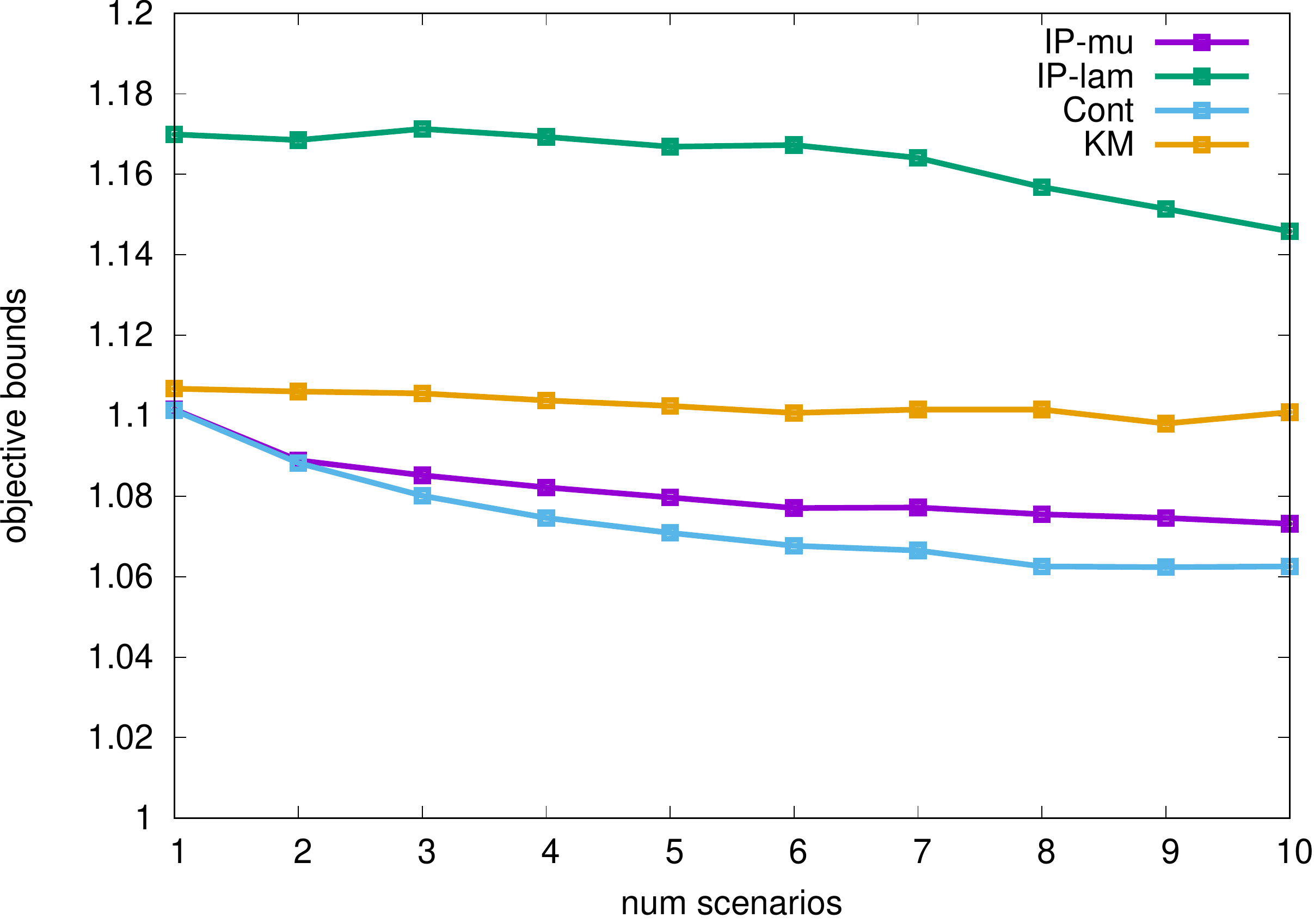}}
\end{center}
\caption{One-stage selection, average objective ratios.}\label{fig:1st-obj-sel}
\end{figure}

In Figure~\ref{fig:1st-obj-vc}, we compare the performance of robust solutions in the case of vertex cover problems. Qualitatively, the performance is similar to the case of selection problems.

\begin{figure}[htbp]
\begin{center}
\subfigure[$n=20$, $N=10$]{\includegraphics[width=0.3\textwidth]{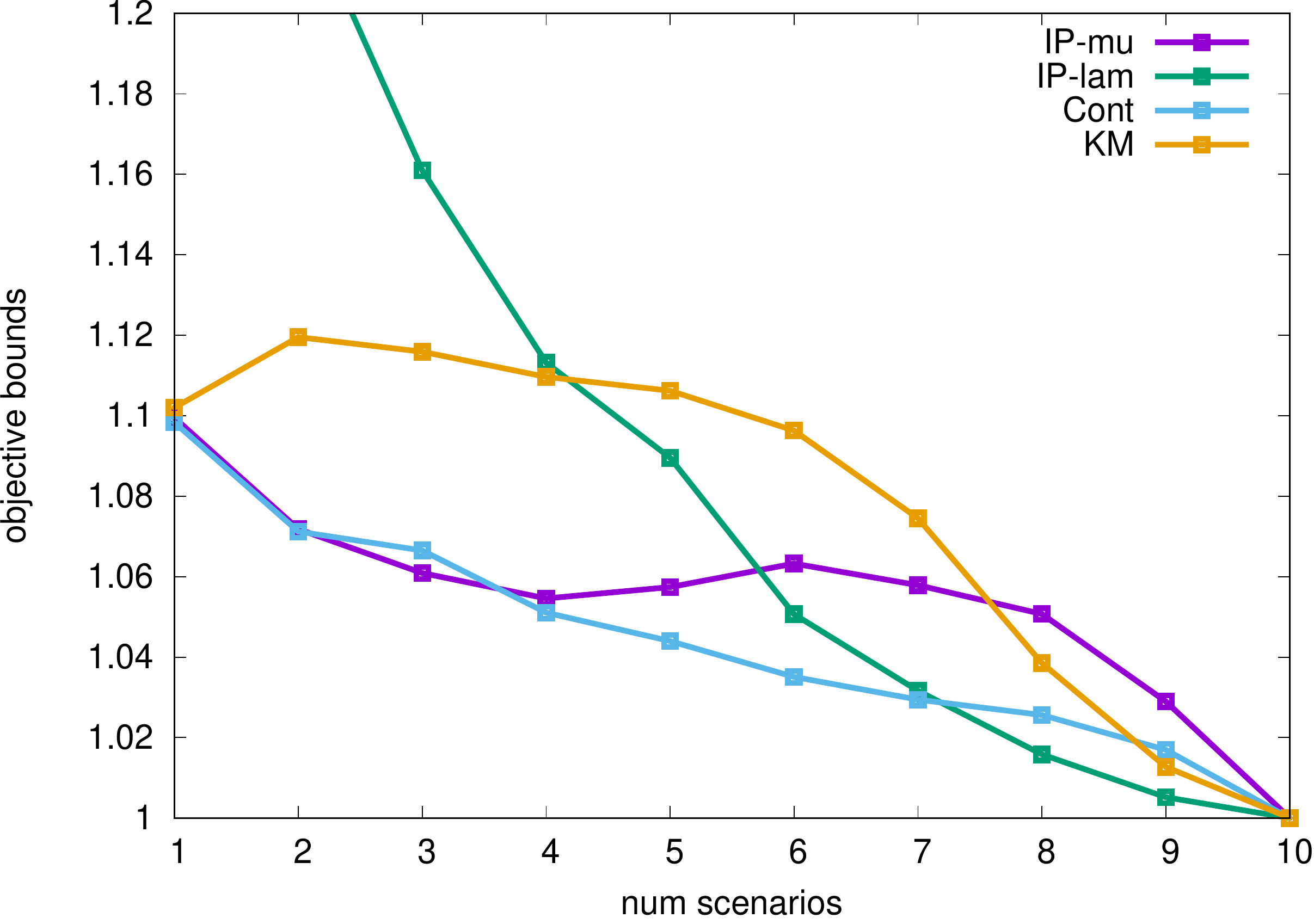}}%
\subfigure[$n=150$, $N=10$]{\includegraphics[width=0.3\textwidth]{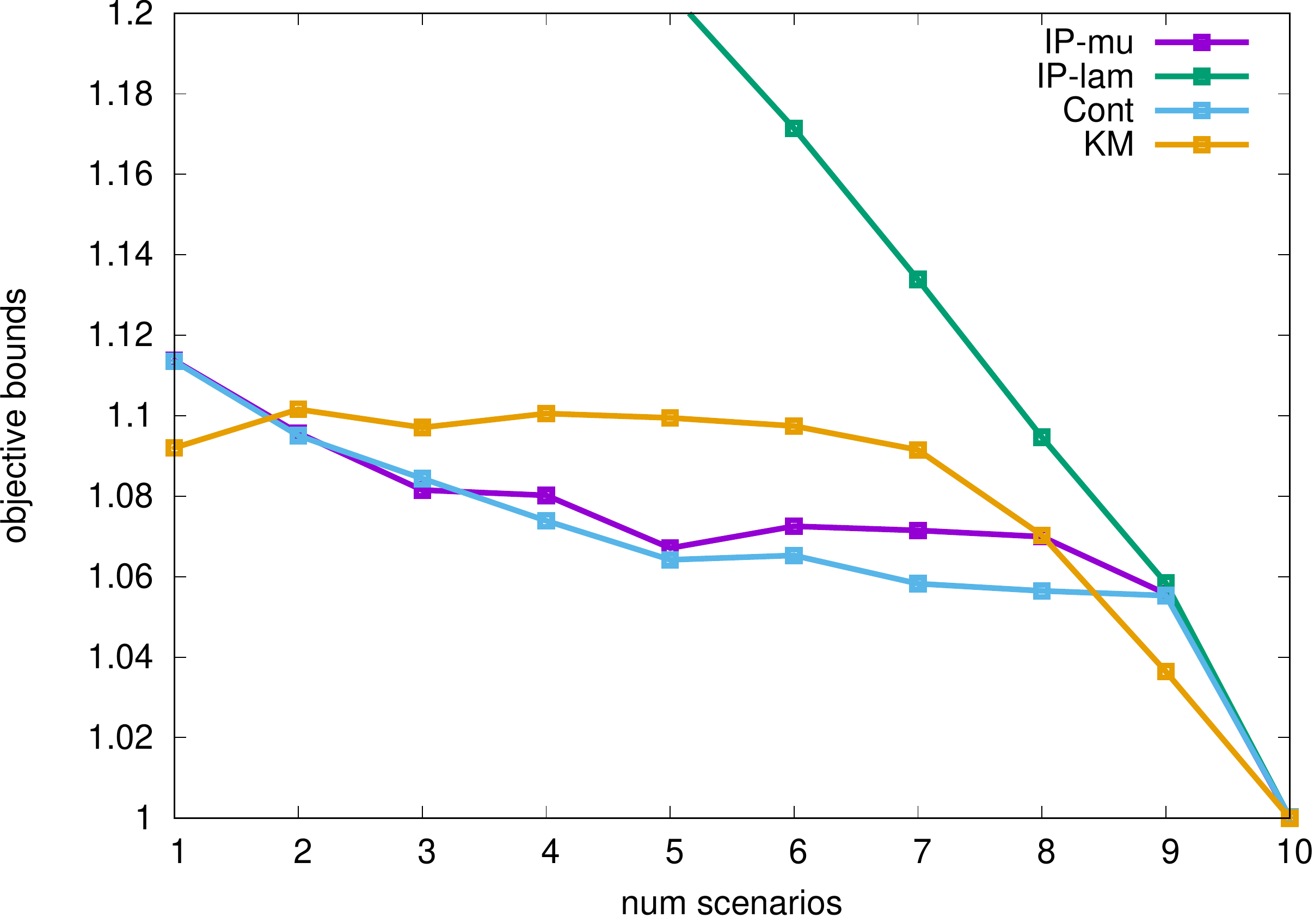}}
\subfigure[$n=20$, $N=50$]{\includegraphics[width=0.3\textwidth]{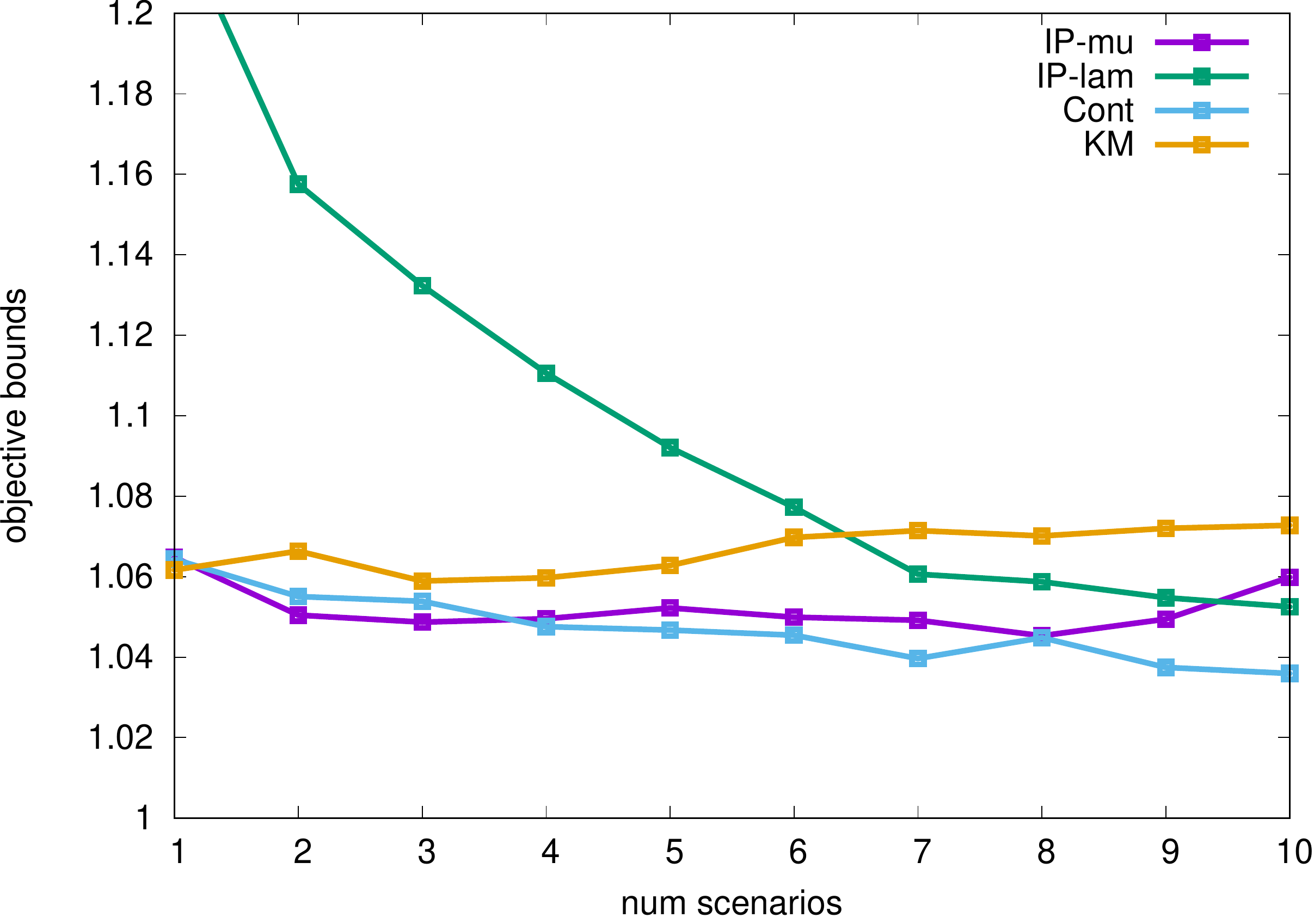}}%
\subfigure[$n=150$, $N=50$]{\includegraphics[width=0.3\textwidth]{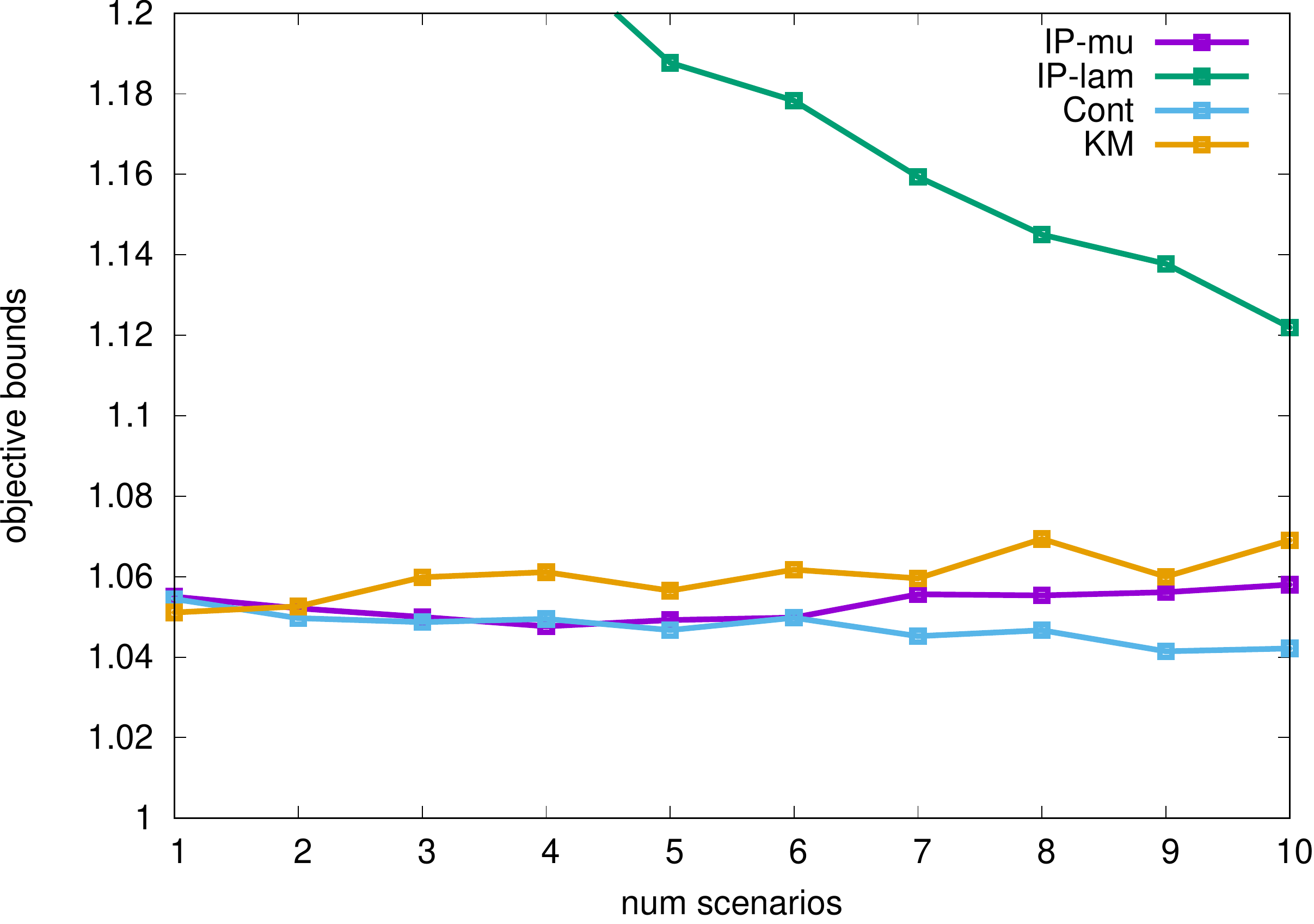}}
\end{center}
\caption{One-stage vertex cover, average objective ratios.}\label{fig:1st-obj-vc}
\end{figure}

\subsubsection{Results on Two-Stage Robust Optimization}
\label{subsec:exptwo}

We now consider two-stage robust optimization problems, where we compare the performance of the scenario reduction method IP with that of K-means. The setup follows the same lines as for our experiments with one-stage problems. In Figure~\ref{fig:2st-agg}, we show average times to construct the reduced uncertainty sets. Recall that problem IP allows for a smaller integer programming formulation than IP-$\lambda$ with only $N$ binary variables. Indeed we find that solution times are now competitive to those of K-means (with the caveat that here the computation time is the sum of 1000 runs). 

\begin{figure}[htbp]
\begin{center}
\subfigure[$n=20$, $N=10$]{\includegraphics[width=0.3\textwidth]{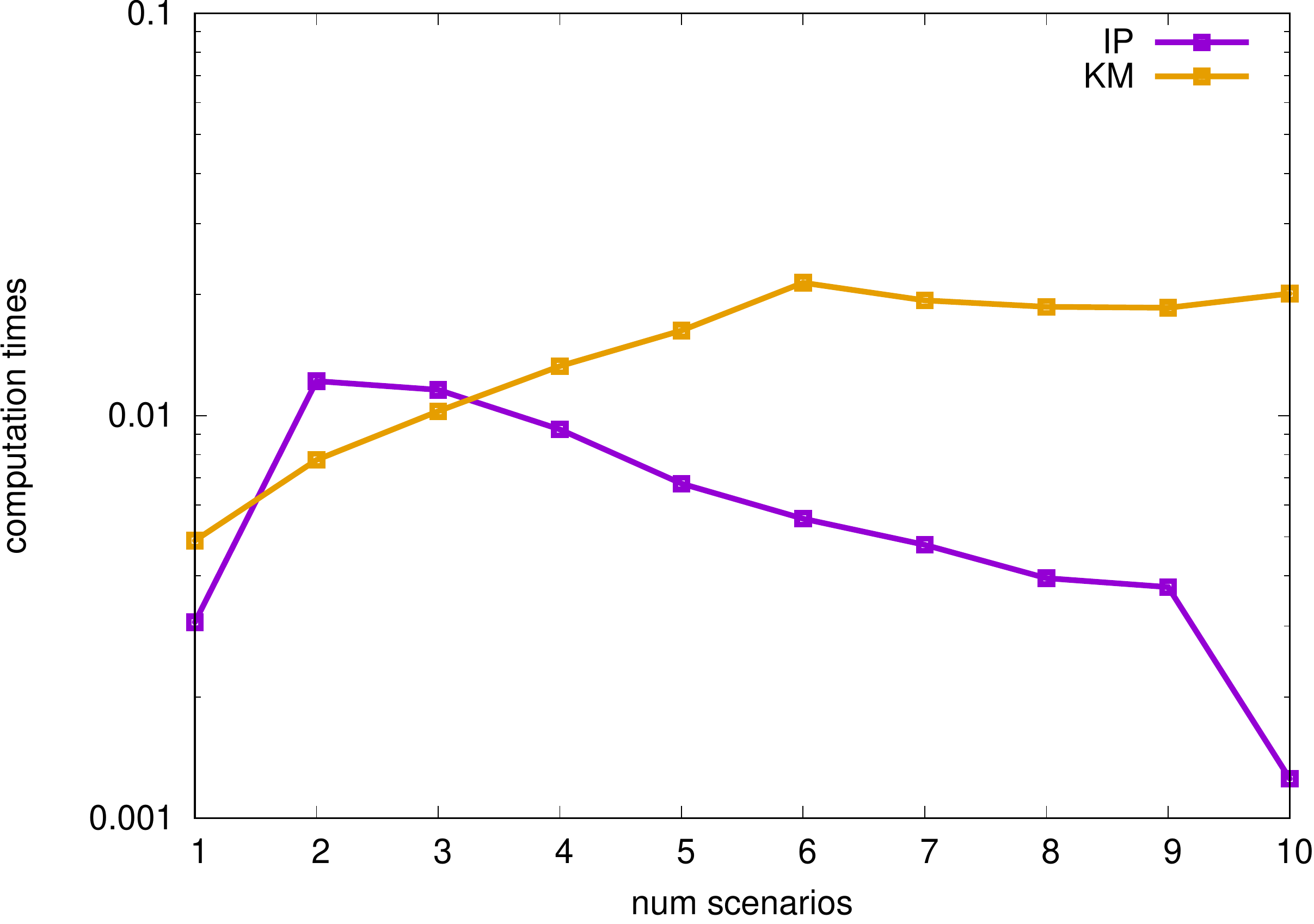}}%
\subfigure[$n=150$, $N=10$]{\includegraphics[width=0.3\textwidth]{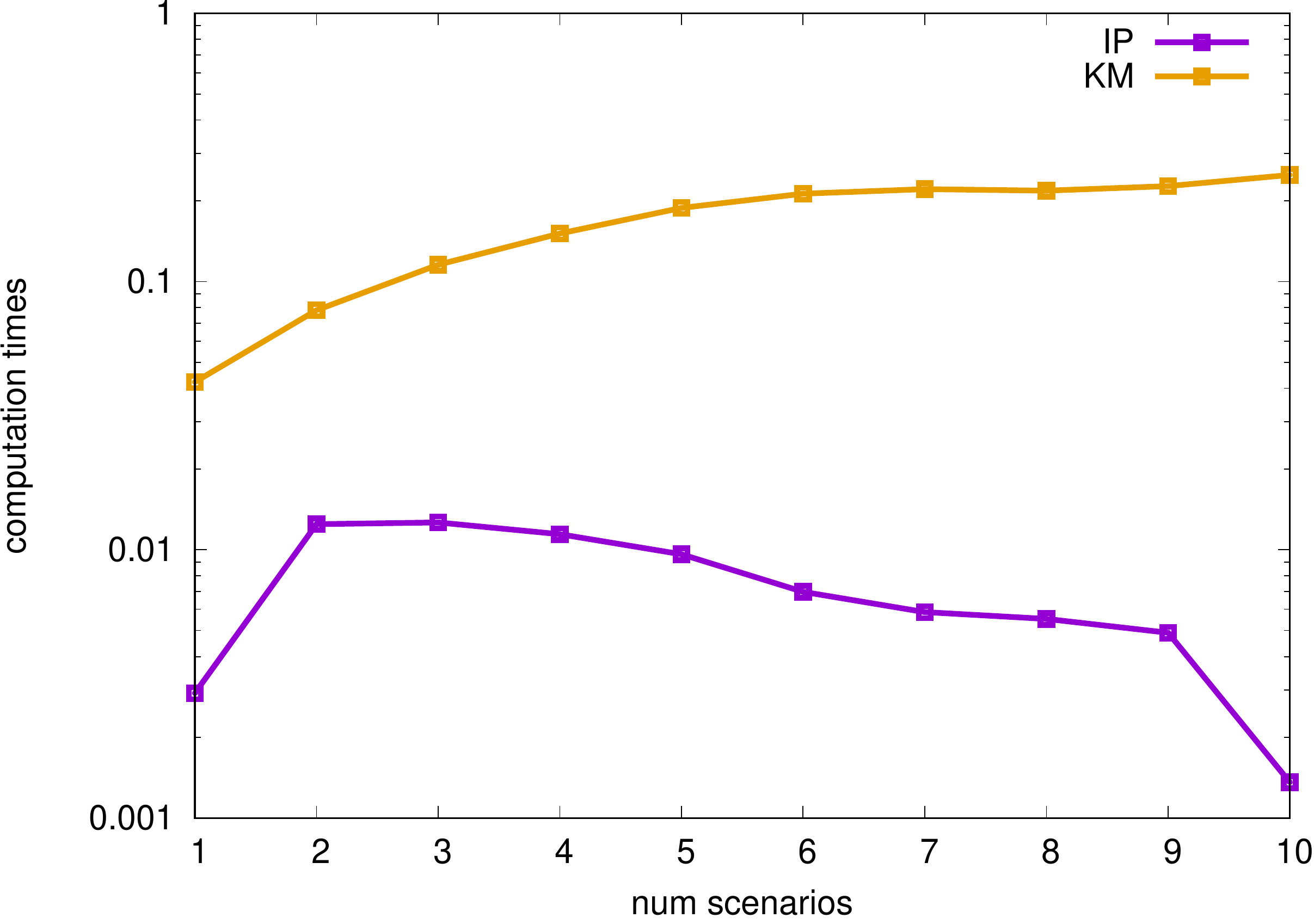}}
\subfigure[$n=20$, $N=50$]{\includegraphics[width=0.3\textwidth]{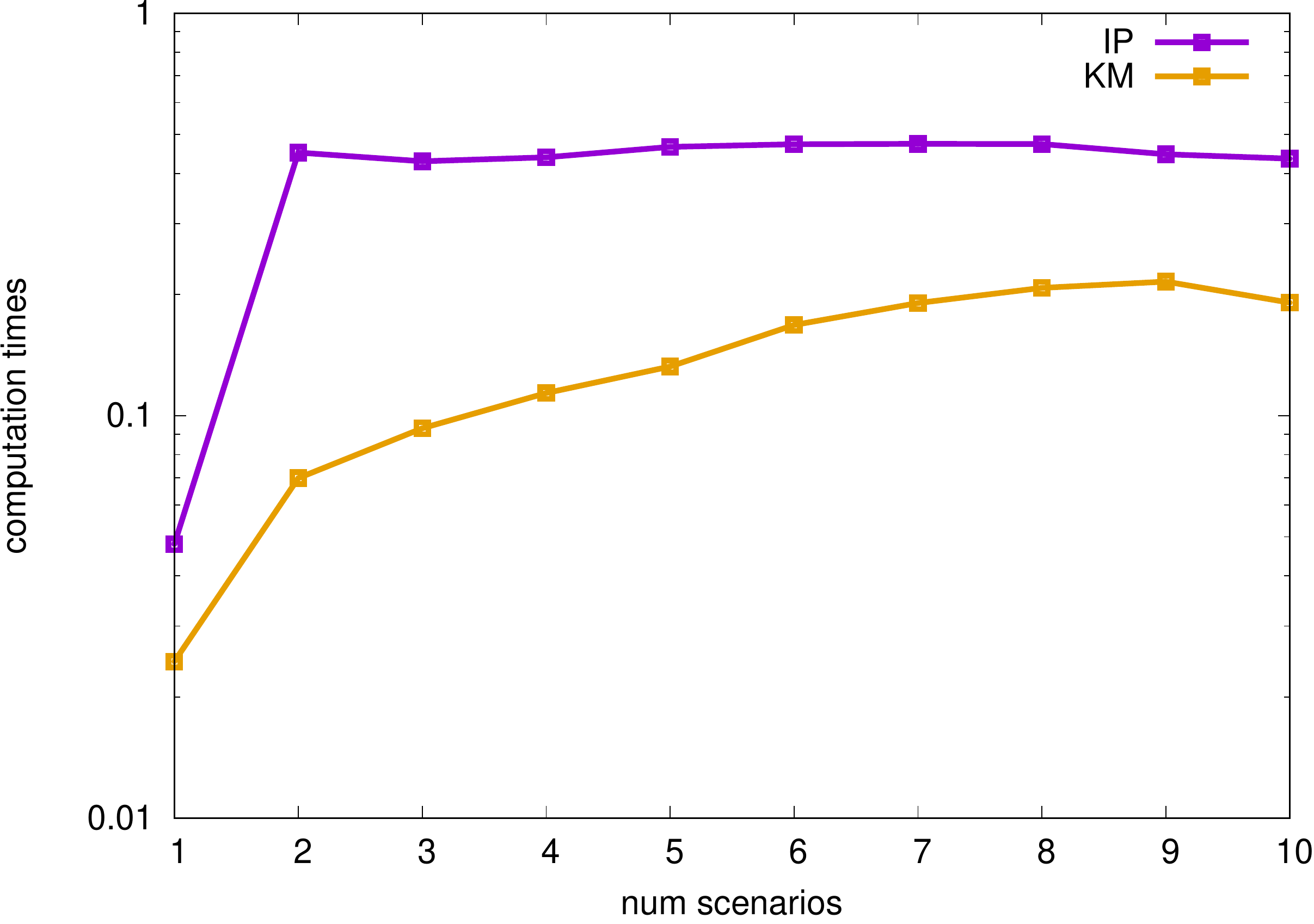}}%
\subfigure[$n=150$, $N=50$]{\includegraphics[width=0.3\textwidth]{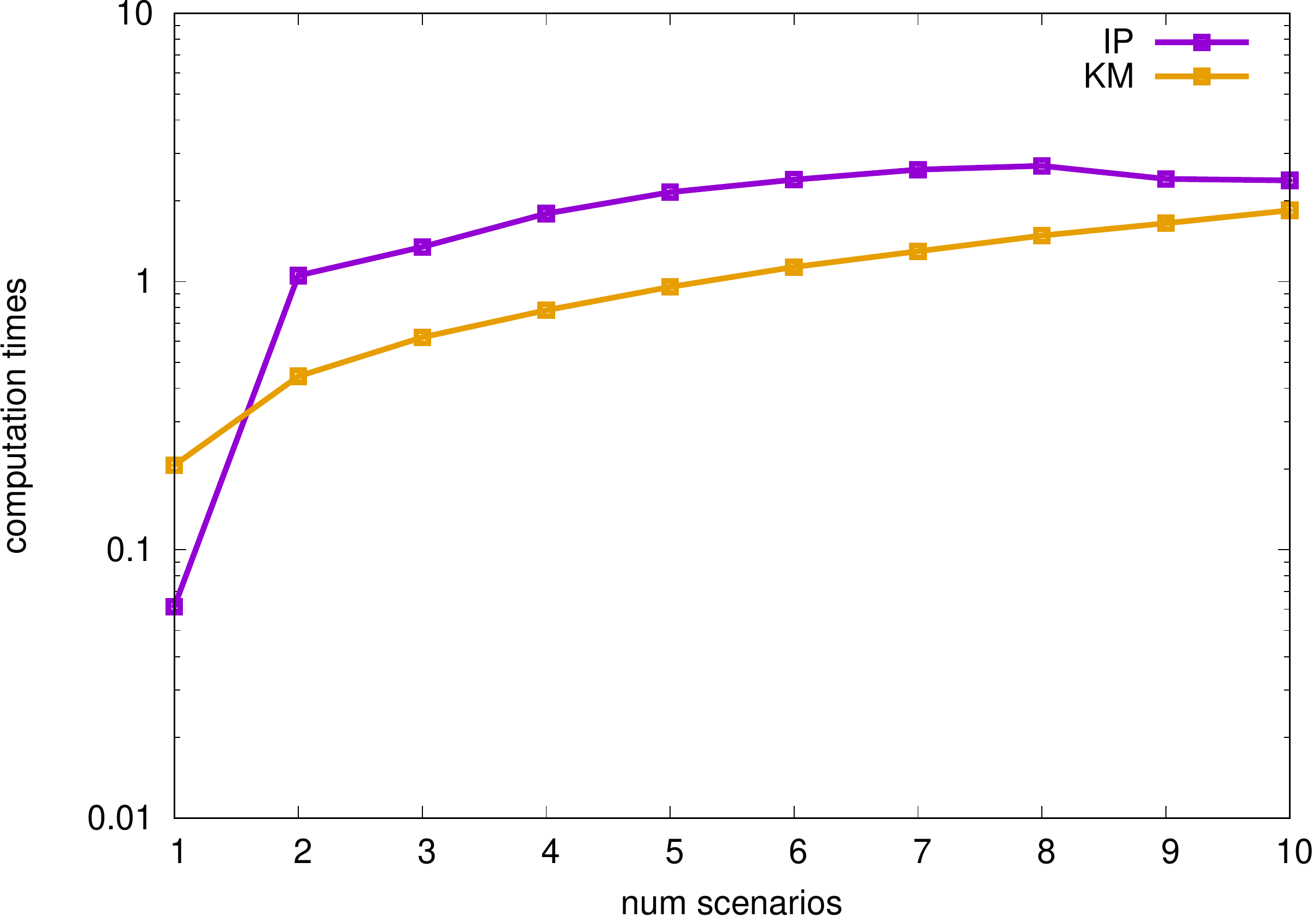}}
\end{center}
\caption{Two-stage aggregation times.}\label{fig:2st-agg}
\end{figure}

In Figure~\ref{fig:2st-obj-sel} we compare the performance of the resulting robust solutions for two-stage selection problems. We can observe that IP gives better solutions overall, where differences are less relevant for low-dimensional problems with $n=20$, but become more pronounced for high-dimensional problems with $n=150$. 

\begin{figure}[htbp]
\begin{center}
\subfigure[$n=20$, $N=10$]{\includegraphics[width=0.3\textwidth]{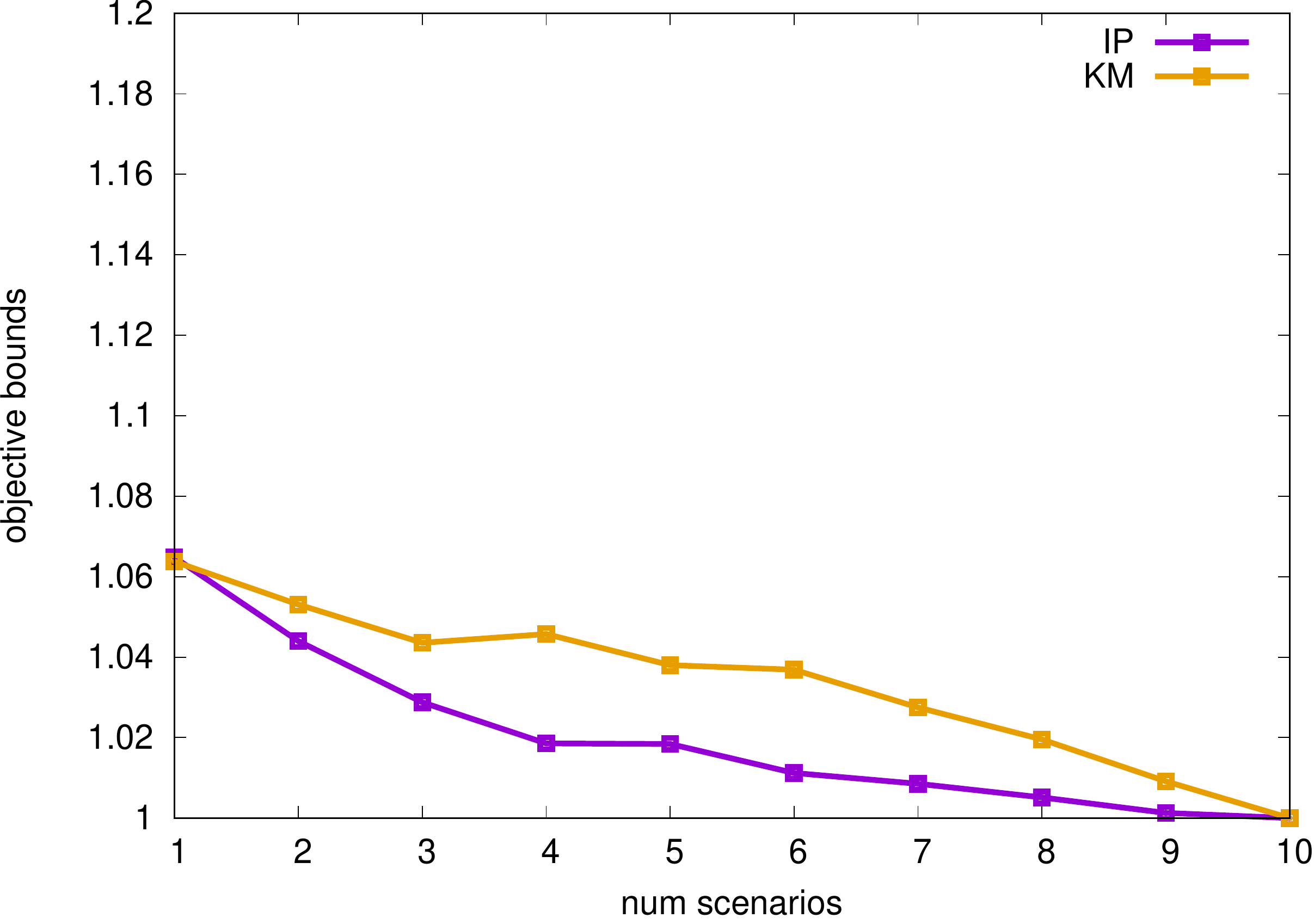}}%
\subfigure[$n=150$, $N=10$]{\includegraphics[width=0.3\textwidth]{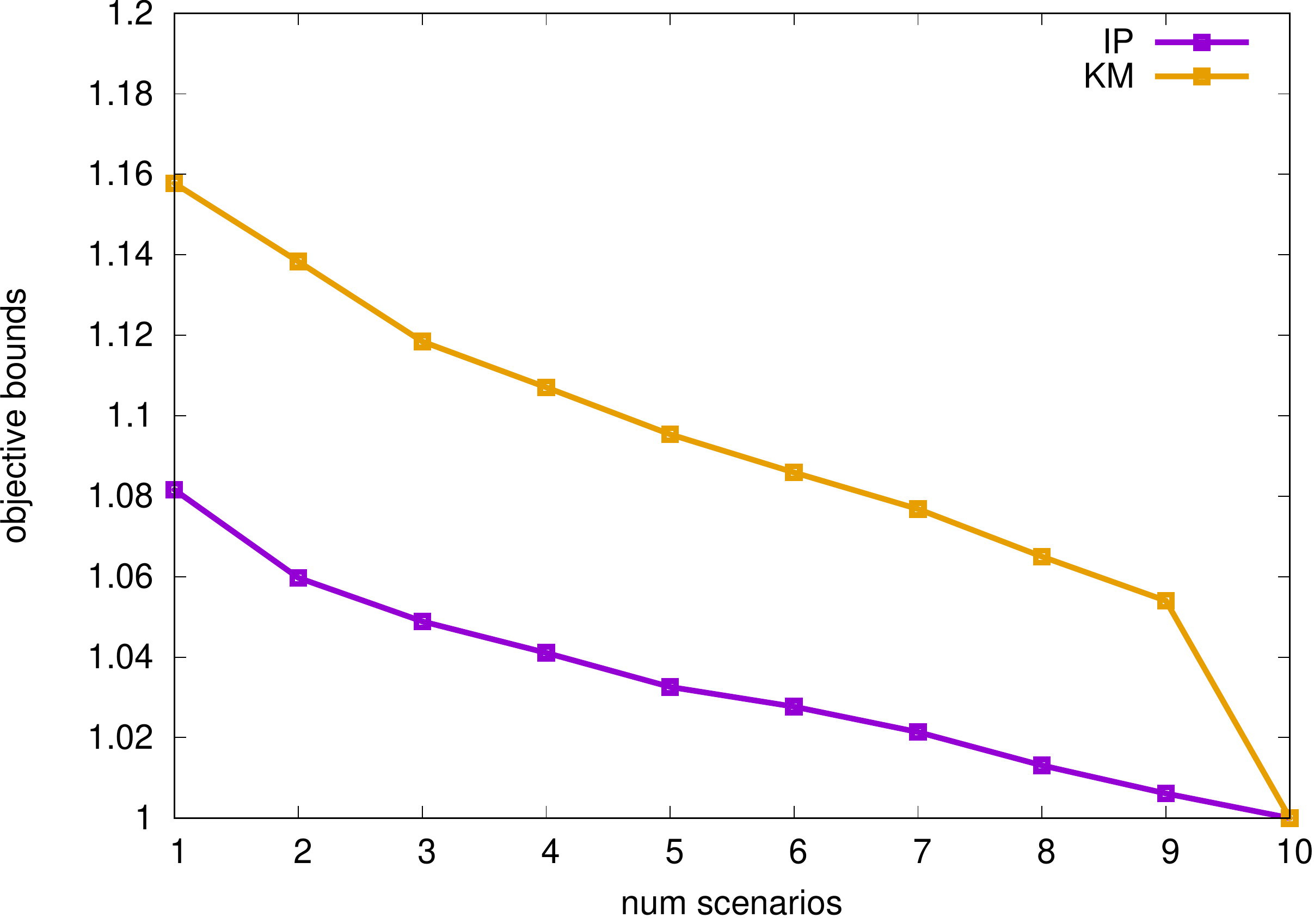}}
\subfigure[$n=20$, $N=50$]{\includegraphics[width=0.3\textwidth]{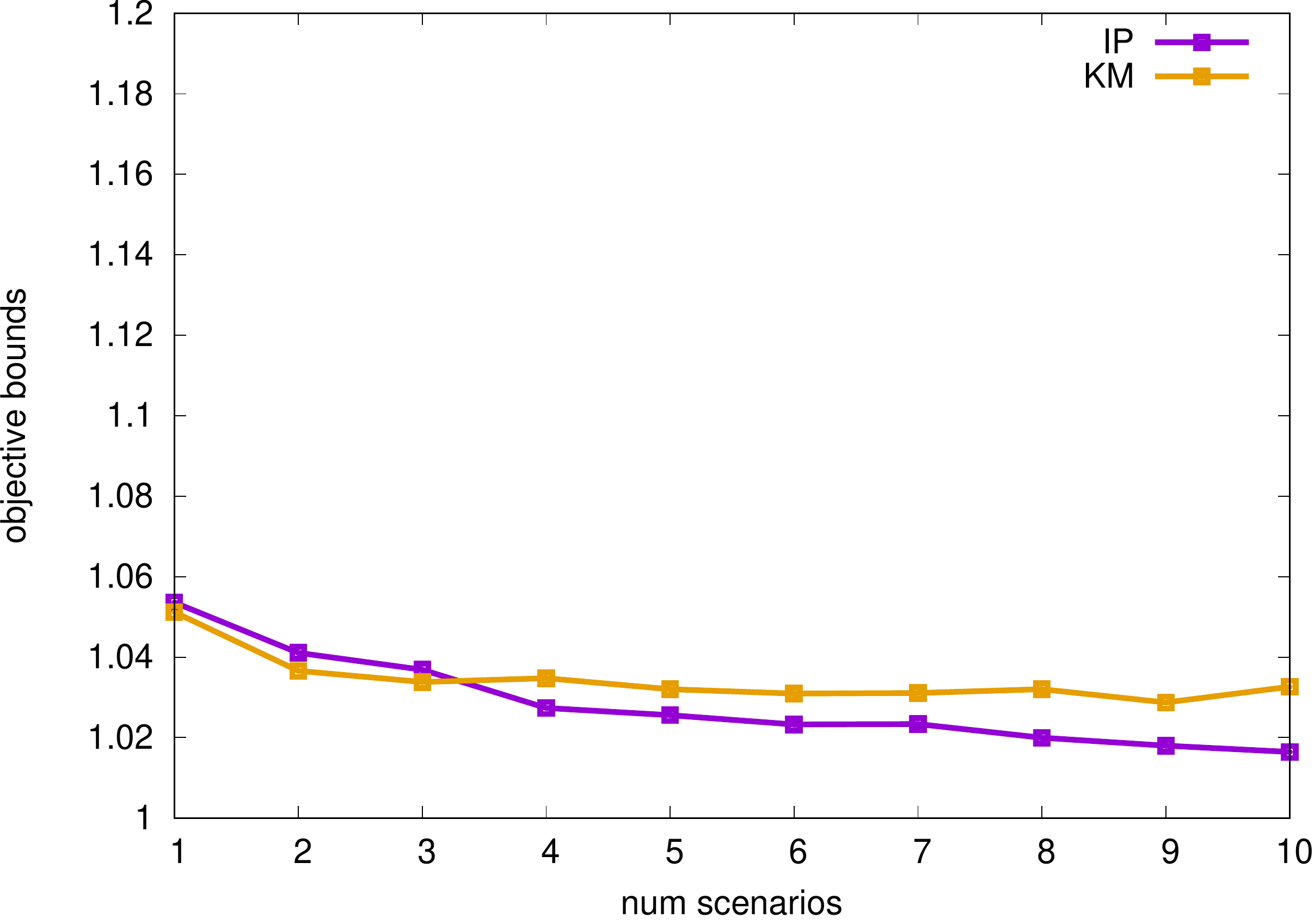}}%
\subfigure[$n=150$, $N=50$]{\includegraphics[width=0.3\textwidth]{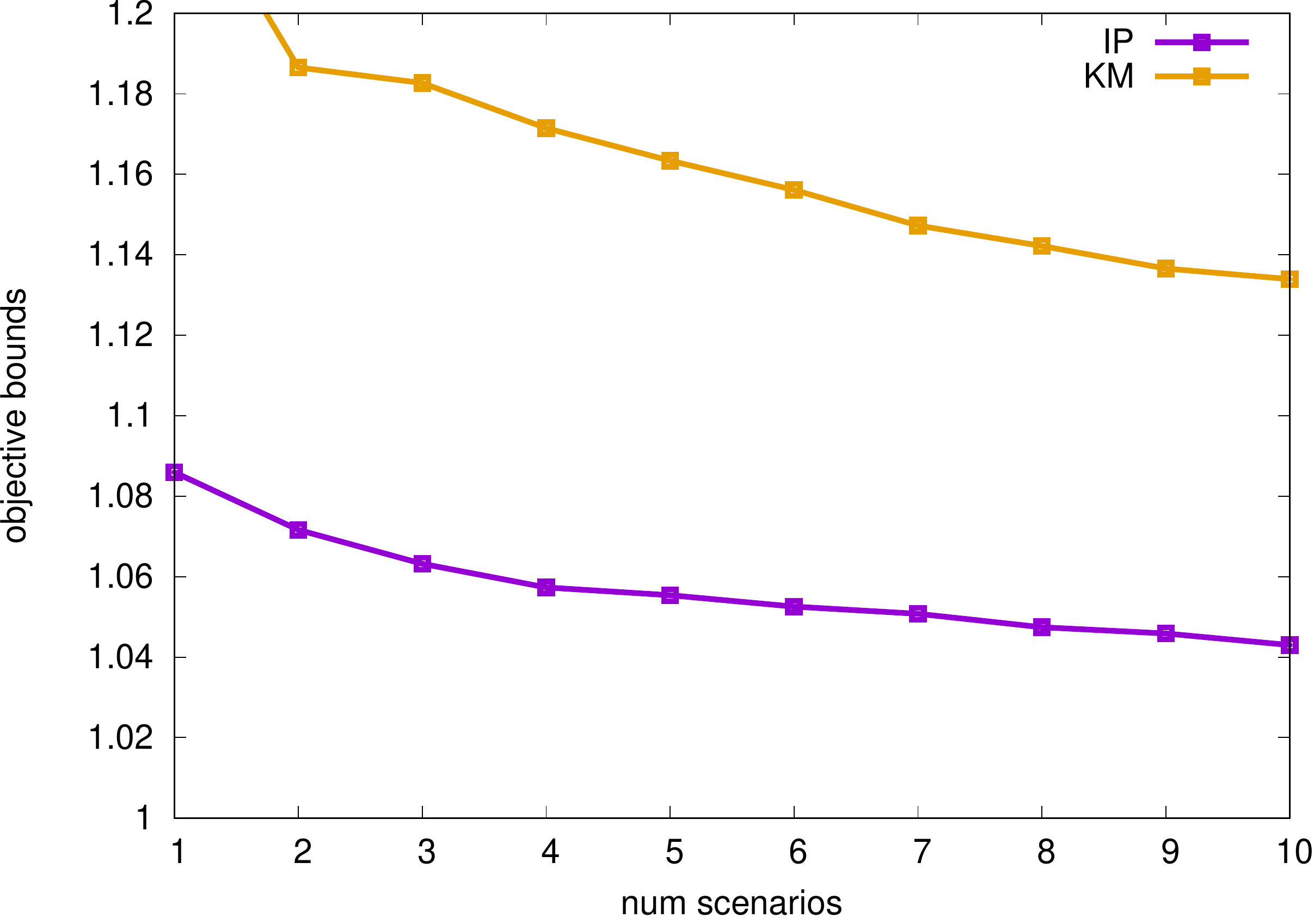}}
\end{center}
\caption{Two-stage selection, average objective ratios.}\label{fig:2st-obj-sel}
\end{figure}

While the same uncertainty sets are used for two-stage vertex cover problems, the benefits of using IP are not as clear-cut in this case, see Figure~\ref{fig:2st-obj-vc}. For the case of high problem dimension and small uncertainty set ($n=150$, $N=10$), IP outperforms K-means for all values of $K$ in the plot. For the other cases, solutions based on K-means clustering tend to be better for small values of $K$, but IP tends to improve more as $K$ increases, while the curve for K-means remains flat, especially for $N=50$. This is not surprising, given that as with method IP-$\lambda$, method IP is restricted to choose a subset of existing scenarios. If few scenarios can be chosen (i.e., if $K$ is small), then there may be advantage is choosing scenarios from the convex hull to better represent multiple scenarios simultaneously -- even if this is not reflected in the approximation guarantee. However, this only seems relevant for the smallest values of $K$.

\begin{figure}[htbp]
\begin{center}
\subfigure[$n=20$, $N=10$]{\includegraphics[width=0.3\textwidth]{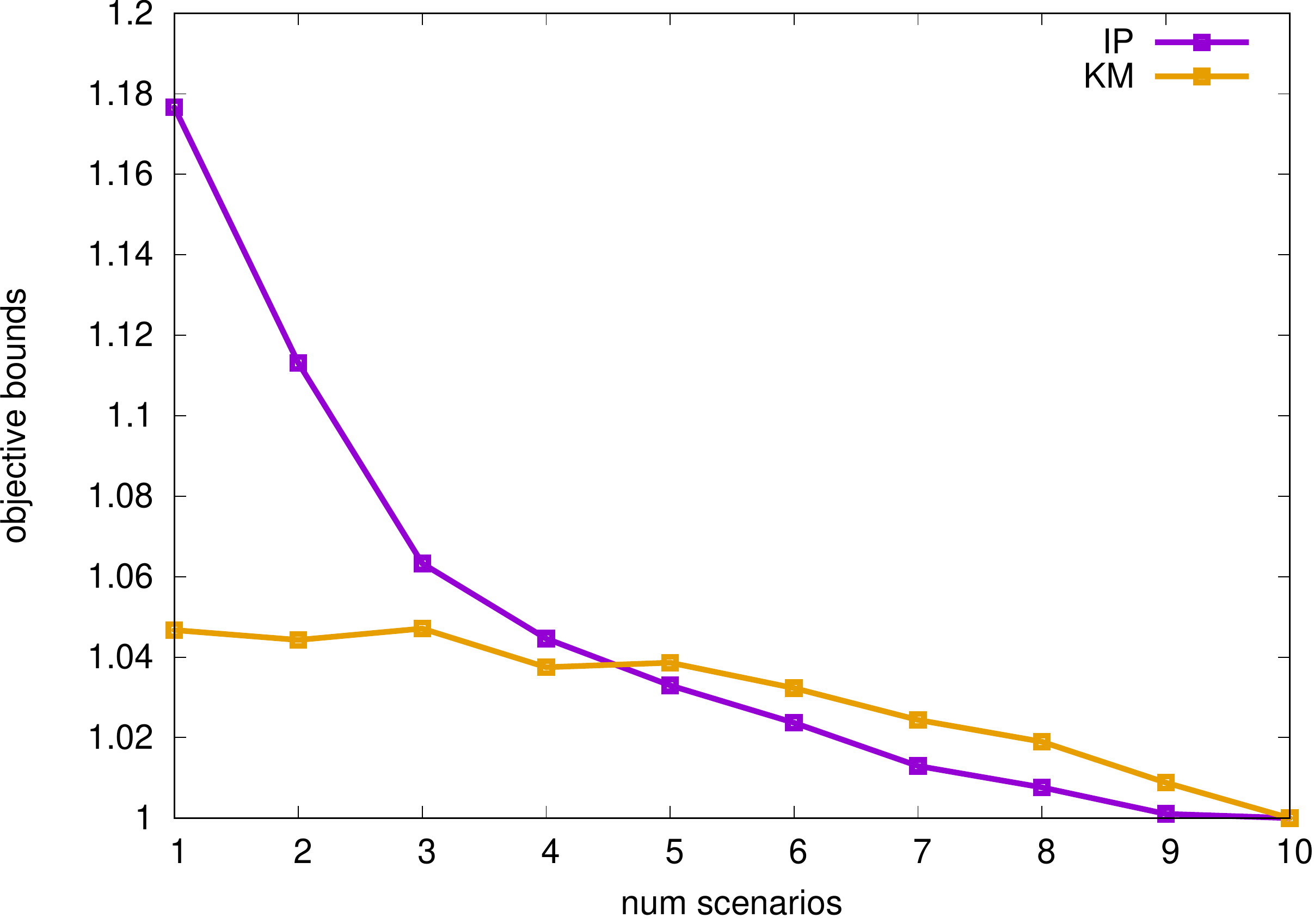}}%
\subfigure[$n=150$, $N=10$]{\includegraphics[width=0.3\textwidth]{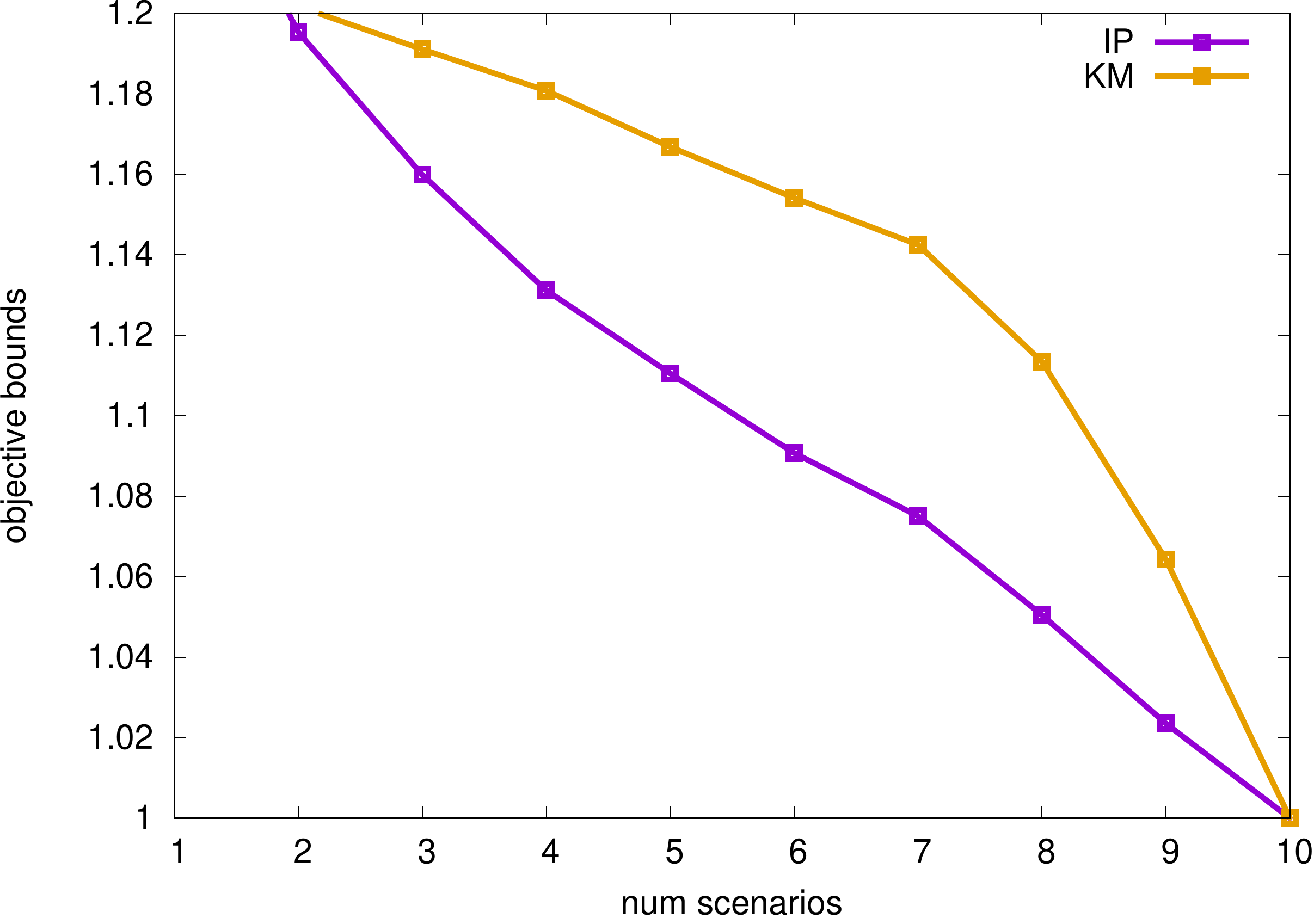}}
\subfigure[$n=20$, $N=50$]{\includegraphics[width=0.3\textwidth]{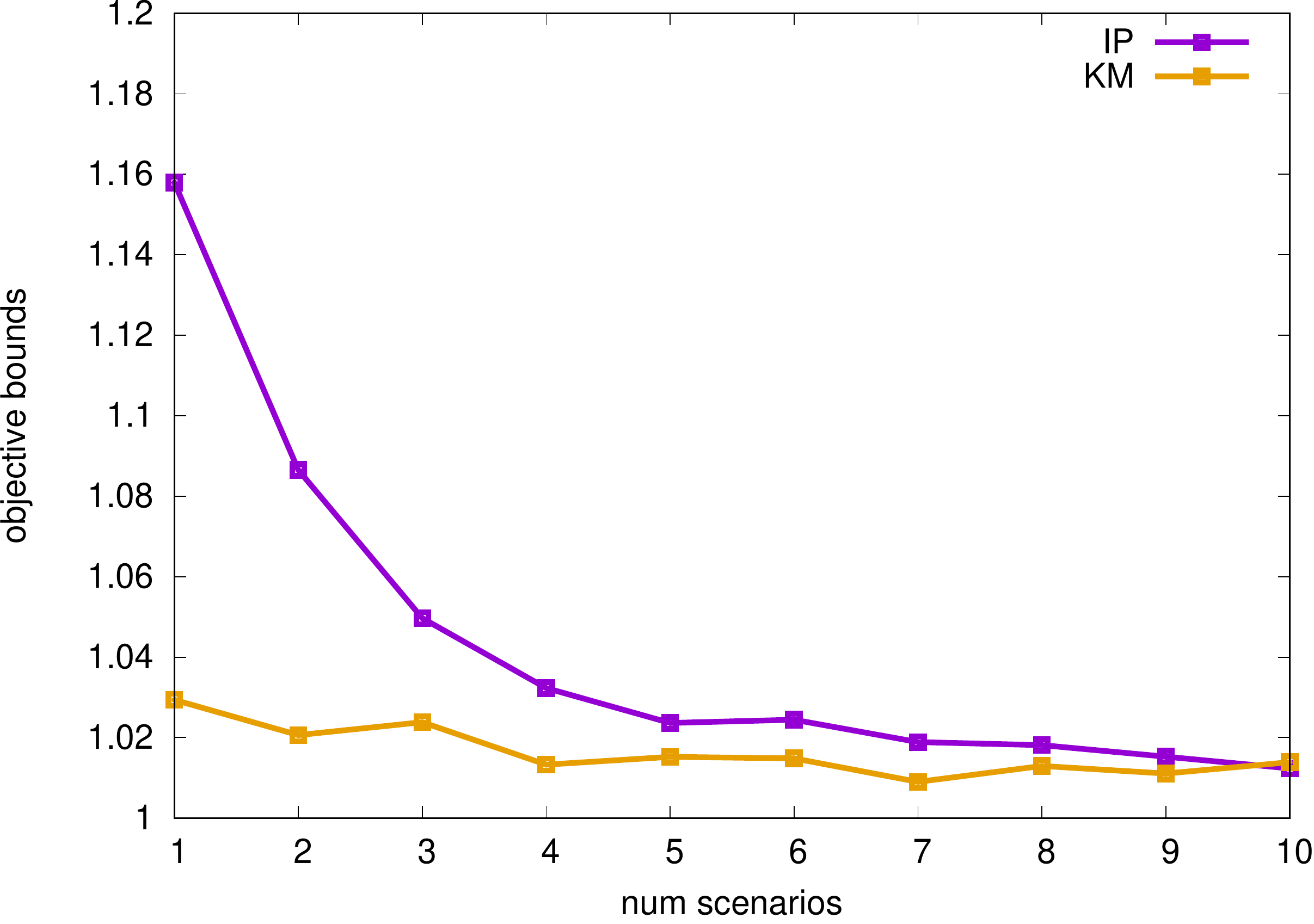}}%
\subfigure[$n=150$, $N=50$]{\includegraphics[width=0.3\textwidth]{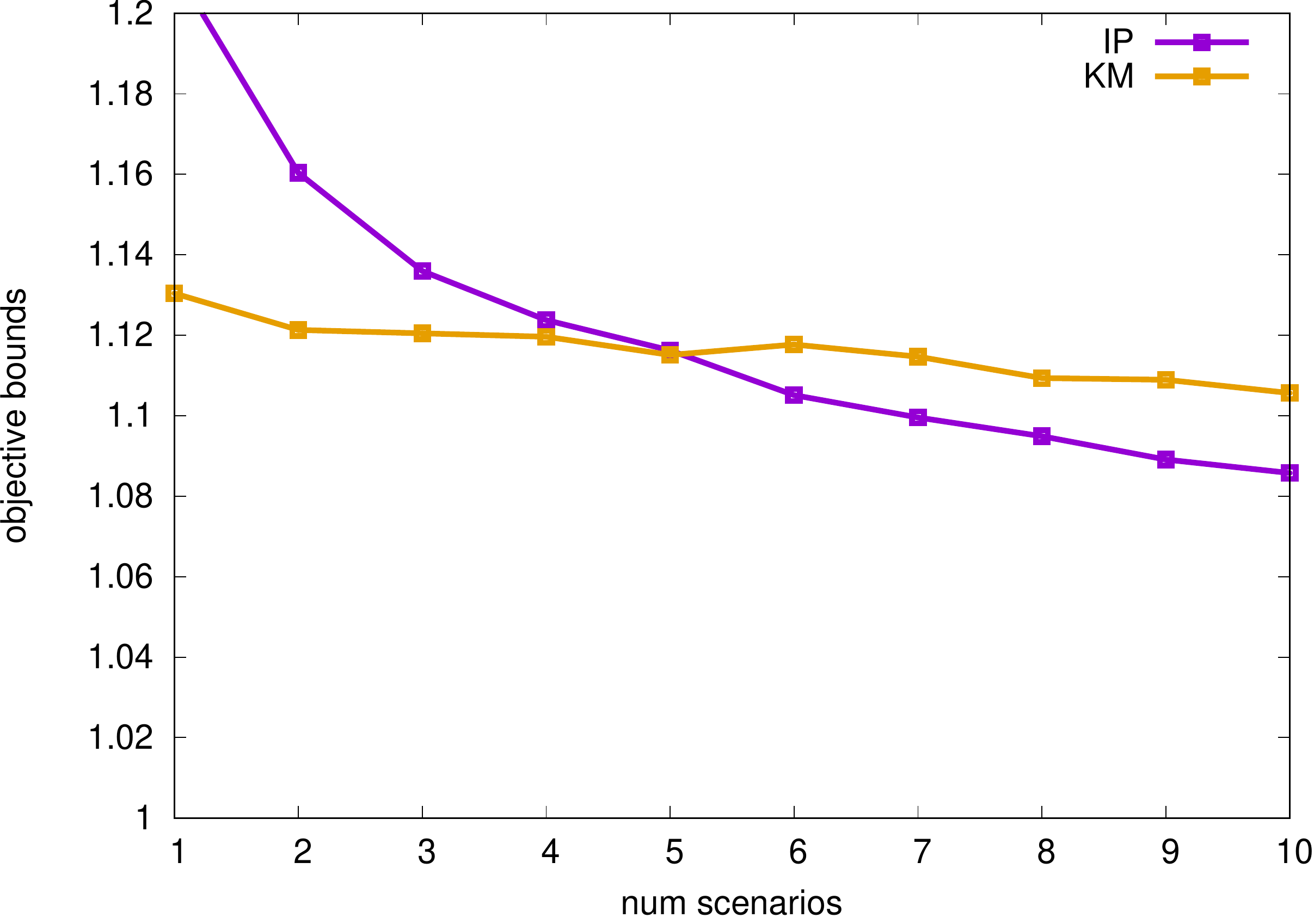}}
\end{center}
\caption{Two-stage vertex cover, average objective ratios.}\label{fig:2st-obj-vc}
\end{figure}

Overall, our scenario reduction methods give a valuable tool to find improved solutions even if data is generated in a way where both our methods and the K-means method find reduced uncertainty sets whose worst-case objective values have a high correlation to the original worst-case objective values. This advantage is likely to increase further on uncertainty sets that are less suitable to reduce using the K-means method.

\section{Conclusions}
\label{sec:conclusions}

Over the last years, the use of data for optimization problems has seen increasing attention in the research community. In the robust optimization setting, several data-driven approaches to modeling uncertainty sets have been developed. Usually, data that is observed in practice comes as discrete data points, e.g., one scenario may give measurements of the parameters of the optimization problem at one specific point in time. On the one hand, using more data in the optimization process should lead to more finely calibrated models and thus improved performance of robust solutions; on the other hand, robust optimization problems, in particular two-stage problems, tend to get harder to solve, the more scenarios are involved.

It is therefore a natural question to consider if it possible to reduce the number of scenarios of an uncertainty set while remaining as true as possible to the original problem. In the area of stochastic optimization, where the use of probability distributions enables us to better estimate the effect of removing scenarios from consideration, such approaches are well-studied. To the best of our knowledge, no general principled methodology has been studied for robust optimization, where the removal of a single scenario may have dramatic effects on the worst-case performance of a solution.

In this paper we propose to reduce the scenario set of a robust optimization problem with the aim of bounding the approximation ratio of the resulting solution as tightly as possible. This means that we can calculate scenario sets of any desired size that give a guarantee on the relative increase in objective value for the resulting robust solution in comparison to the robust solution of the original problem. Our approach only considers the scenario data and can be applied to any set of feasible solutions, provided that all variables are non-negative. Following this idea, we derived several models that can be applied to one-stage optimization problems, and a model that can be used in the case of two-stage robust optimization. While the latter problem is actually inapproximable, we can avoid this issue by providing instance-dependant rather than a-priori approximation guarantees.

In computational experiments we show that the reduced uncertainty sets we construct yield a higher corrleation on robust objective values than can be achieved using K-means, in particular so if the data contains outliers. Using selection and vertex cover problems as testbeds, we observed that our methods can result in robust solutions that provide a considerable improvement in performance when compared to using K-means, at the cost of increased effort when calculating reduced scenario sets.

Several further directions of research should be considered. On the one hand, is it possible to further improve the approximation guarantees that clustering approaches provide? While our method has the advantage that it can be applied to any robust problem, further specialization may bring stronger guarantees, i.e., incorporating the structure of $\X$ may lead to stronger results. On the other hand, is it possible to find clusterings of similar quality using less computation time? All of our approaches are based on the solution of integer linear programs or sequences of linear programs. It may be possible to avoid the solution of such problems altogether to reach methods that are similarly scalable as the K-means algorithm.

\FloatBarrier
\appendix

\section{Complexity Proofs}
\label{sec:proofs}

\begin{theorem}\label{th:hardness}
The decision version of Cont is NP-complete.
\end{theorem}
\begin{proof}
Consider the following yes/no decision version of Cont: Given scenarios $\cU = \{\pmb{c}^1,\ldots,\pmb{c}^N\}$ as well as a number of clusters $K$ and an approximation guarantee $t$, is there a reduced set $\mathcal{C}=\{\hat{\pmb{c}}^1,\ldots,\hat{\pmb{c}}^K\}$ in the convex hull of $\cU$, such that for all $i\in[N]$ there exists a $\hat{\pmb{c}}\in\text{conv}(\C)$ such that $\pmb{c}^i \le t\cdot \hat{\pmb{c}}$?

We use a reduction from the dominating set problem, which is stated as follows. Given an undirected graph $G=(V,E)$, is there a subset of nodes $V'\subseteq V$ with cardinality $K$, such that for each $v\in V$, either $v\in V'$ or there is an edge $\{v,v'\}\in E$ for some $v'\in V'$?

Let an instance of dominating set be given, with $V=\{v_1,\ldots,v_n\}$. We construct a cluster-aggregation problem with the same value of $K$ and $t=4/3$ that is a yes-instance, if and only if the dominating set problem is a yes-instances.

We show a small example of the proposed reduction in Figure~\ref{fig:red}. We consider a problem dimension $n$. We create $N=2n$ scenarios of two different types. Scenarios of the first type, denoted as $\tilde{\pmb{c}}^i$ for $i\in[n]$, are 0 everywhere, except for the $i$th entry, which is 16. Scenarios of the second type, denoted as $\pmb{c}^i$ for $i\in[n]$, are 9 everywhere, except for those entries corresponding to nodes that can be covered by node $v_i\in V$, where we set the value to be 12. Note that for $t=4/3$, $\pmb{c}^j \le t\pmb{c}^i$ for all $i,j\in[n]$. Furthermore, we have $\tilde{\pmb{c}}^j \le t\pmb{c}^i$ for all nodes $v_j$ that can be covered by node $v_i$.
Intuitively, the idea of this construction is that we need to choose the right scenarios of the second type to cover all scenarios of the first type.

\begin{figure}[htb]
\begin{center}
\includegraphics[width=0.7\textwidth]{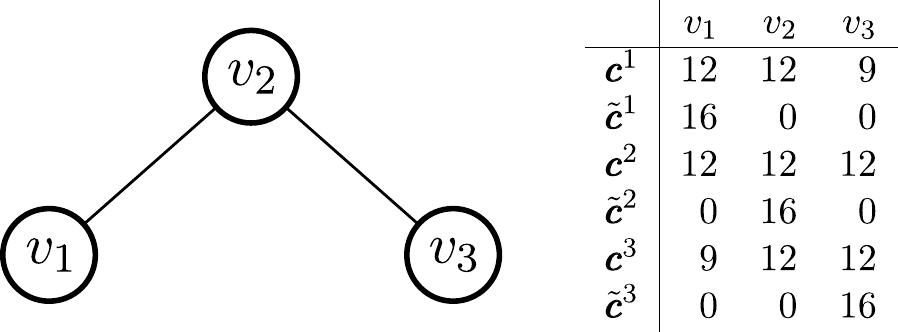}
\caption{Example reduction.}\label{fig:red}
\end{center}
\end{figure}

We now prove that the Cont instance is a yes-instance, if and only if the dominating set instance is a yes-instance.

For the first direction, let us assume that there is a dominating set $V'=\{v_{i_1},\ldots,v_{i_K}\}$ of cardinality $K$. We construct a solution to Cont by choosing $\C=\pmb{c}^{i_1},\ldots,\pmb{c}^{i_{K}}$. By construction, it holds that all of the $N$ scenarios are dominated by at least one of these scenarios scaled by factor $t=4/3$.

For the other direction, let us assume that there exists $\C=\hat{\pmb{c}}^1,\ldots,\hat{\pmb{c}}^K$ in the convex hull of $\cU$ such that each scenario $\pmb{c}\in\cU$ is dominated by a convex combination of scenarios in $\C$ scaled by $t=4/3$. Let us write $\hat{\pmb{c}}^k = \sum_{i\in[n]} \lambda^k_i \pmb{c}^i + \sum_{i\in[n]} \tilde{\lambda}^k_i \tilde{\pmb{c}}^i$ with $\sum_{i\in[n]} \lambda^k_i + \tilde{\lambda}^k_i = 1$. 

We first show that we can assume $\tilde{\lambda}^k_i=0$ for all $i\in[n]$ and $k\in[K]$. To this end, let us assume that $\tilde{\lambda}^k_i > 0$ for some $i\in[n]$ and some $k\in[K]$. Consider scenario $\tilde{\pmb{c}}^i$ and its dominating convex combination of scenarios in $\C$, given as $\sum_{k\in[K]} \mu_k\hat{\pmb{c}}^k$. Observe that scenario $\tilde{\pmb{c}}^i$ is zero everywhere, except for dimension $i$. Hence, if any such choice of values $\mu_k$ exists, we can assume without loss of generality that $\mu_k=1$ for the one scenario $\hat{\pmb{c}}^k$ that has the largest value in dimension $i$ (by assumption, at least 12). This means that each scenario $\tilde{c}^i$ is dominated by the scaled version of a single scenario $\hat{\pmb{c}}^k$. So let us assume that $\tilde{\pmb{c}}^i$ is not covered by another scenario in $\C$. We further distinguish the following cases.
\begin{itemize}
\item There is no other $j\neq i$ with $\tilde{\lambda}^k_j > 0$. This means that for all $j\neq i$, $\hat{c}^k_j < 12$. Hence, the scenario $\hat{\pmb{c}}^k$ only covers scenario $\tilde{\pmb{c}}^i$ within $t=4/3$. We can thus use scenario $\pmb{c}^i$ instead and remain feasible.

\item There is some $j\neq i$ with $\tilde{\lambda}^k_j > 0$. Without loss of generality, let $\tilde{\lambda}^k_i \ge \tilde{\lambda}^k_j$. We have that $\tilde{\lambda}^k_j \le 1/2$. Hence, $\hat{c}^k_j < 12$ and scenario $\tilde{\pmb{c}}^j$ must already be covered by another scenario in $\C$. We can thus set $\tilde{\lambda}^k_j=0$ and continue in the same way as in the first case.
\end{itemize}
We conclude that we can assume that $\tilde{\lambda}^k_i=0$ for all $i\in[n]$ and $k\in[K]$, i.e., we only need to consider convex combinations of scenarios of type $\pmb{c}^i$.

So let us assume that $\lambda^k_i > 0$ for some $i\in[n]$. We show that $\lambda^k_i=1$. To this end, let us assume further assume that $\lambda^k_j>0$ for some $j\neq n$. This means that for any node $v_\ell\in V$ not covered by both $v_i$ and $v_j$, we have $\hat{c}^k_\ell < 12$. Hence, to dominate scenario $\tilde{\pmb{c}}^\ell$, the convex combination of scenarios with $\tilde{\pmb{c}}^\ell \le t\sum_{k\in[K]}\mu_k \hat{\pmb{c}}^k$ has $\mu_k = 0$. In other words, the convex combination of multiple scenarios $\pmb{c}^i$ and $\pmb{c}^j$ means that only those scenarios $\tilde{\pmb{c}}^i$ can be covered, which are already possible to cover by using only $\pmb{c}^i$ or only $\pmb{c}^j$.

Overall, we have showed that there is an optimal solution to Cont where all $\lambda$-variables are binary and all $\mu$-variables are binary. By construction, $V'=\{v_{i_1},\ldots,v_{i_K}\}$ is therefore a dominating set, which completes the proof.

\end{proof}

\begin{corollary}
The decision versions of IP-$\mu$, IP-$\lambda$ and IP are NP-complete.
\end{corollary}
\begin{proof}
In the proof of Theorem~\ref{th:hardness}, we constructed an instance where we can assume all $\mu$ and $\lambda$ variables to be binary in an optimal solution. Hence, the same proof applies to these problem variants as well.
\end{proof}

\section{Additional Experimental Results}

While the plots in Sections~\ref{subsec:expone} and \ref{subsec:exptwo} focus on the time to solve the reduction problem and the resulting objective value, here we show additional data on the time that is required to solve the resulting robust optimization problems. Recall that a time limit of 60 seconds was used in the experiments, which truncates the average solution times. In Figures~\ref{selonesoltime} and \ref{vconesoltime}, we show solution times for the one-stage problem, while Figures~\ref{seltwosoltime} and \ref{vctwosoltime} show results on two-stage problems.

\begin{figure}[htbp]
\begin{center}
\subfigure[$n=20$, $N=10$]{\includegraphics[width=0.3\textwidth]{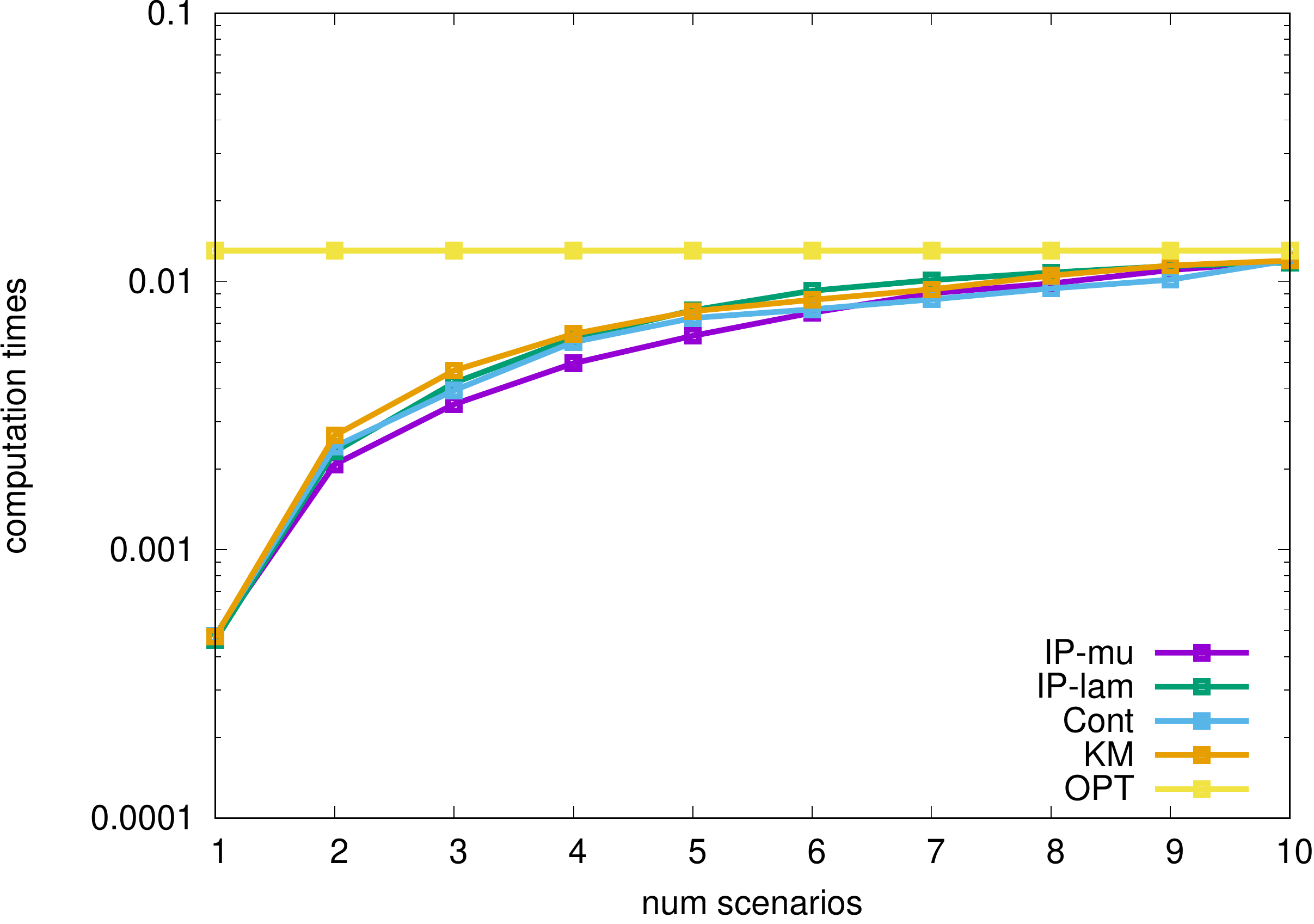}}%
\subfigure[$n=150$, $N=10$]{\includegraphics[width=0.3\textwidth]{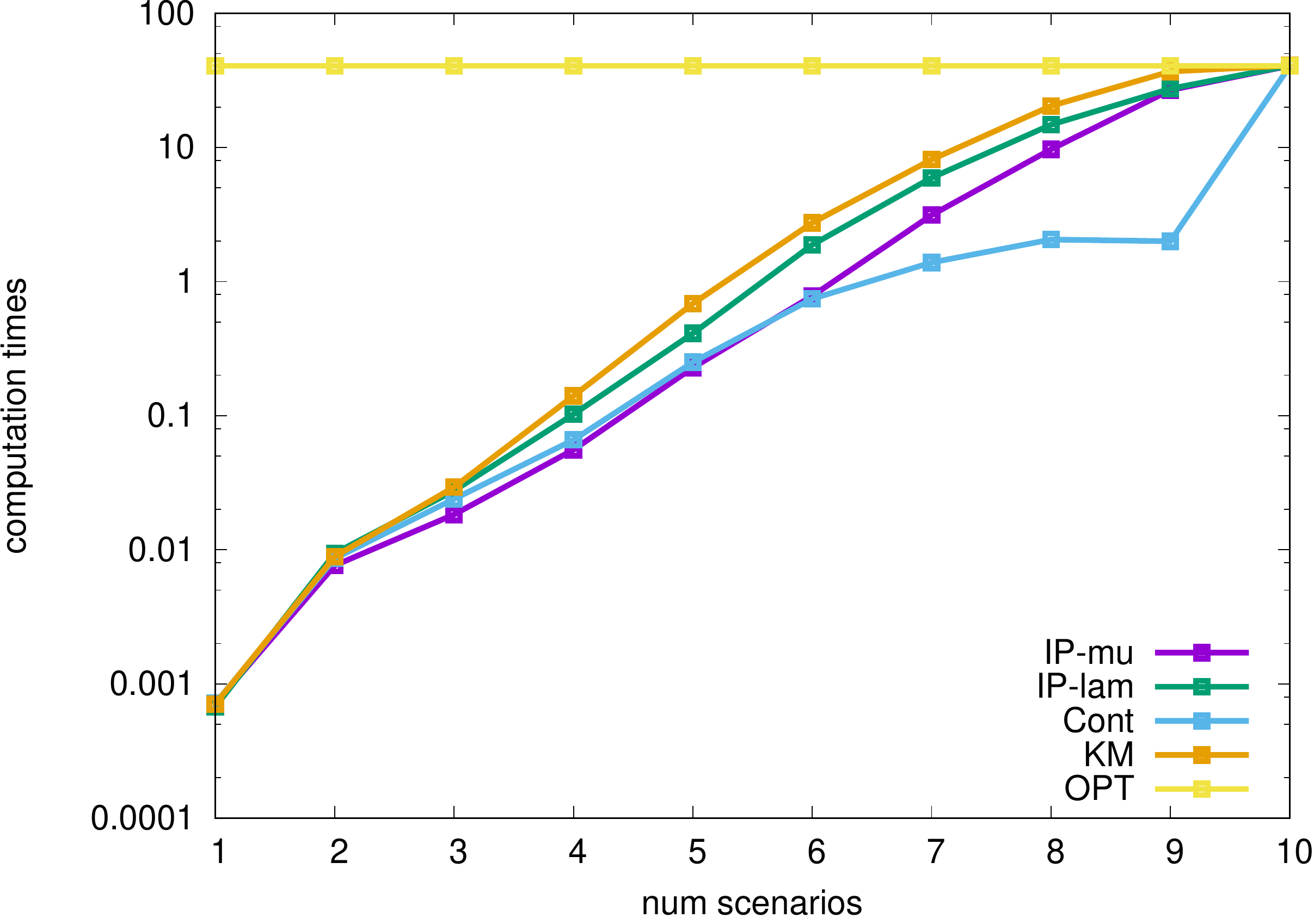}}
\subfigure[$n=20$, $N=50$]{\includegraphics[width=0.3\textwidth]{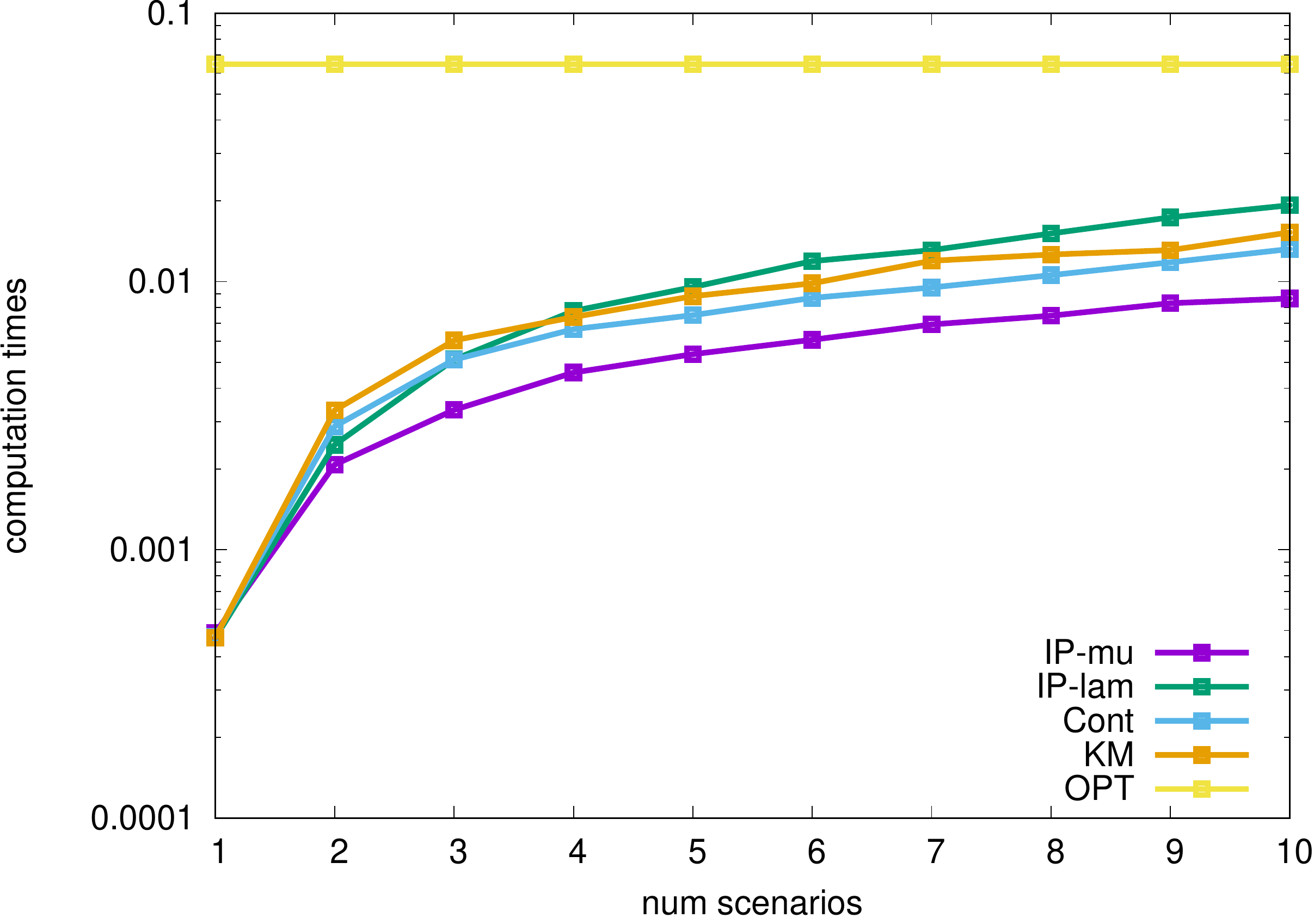}}%
\subfigure[$n=150$, $N=50$]{\includegraphics[width=0.3\textwidth]{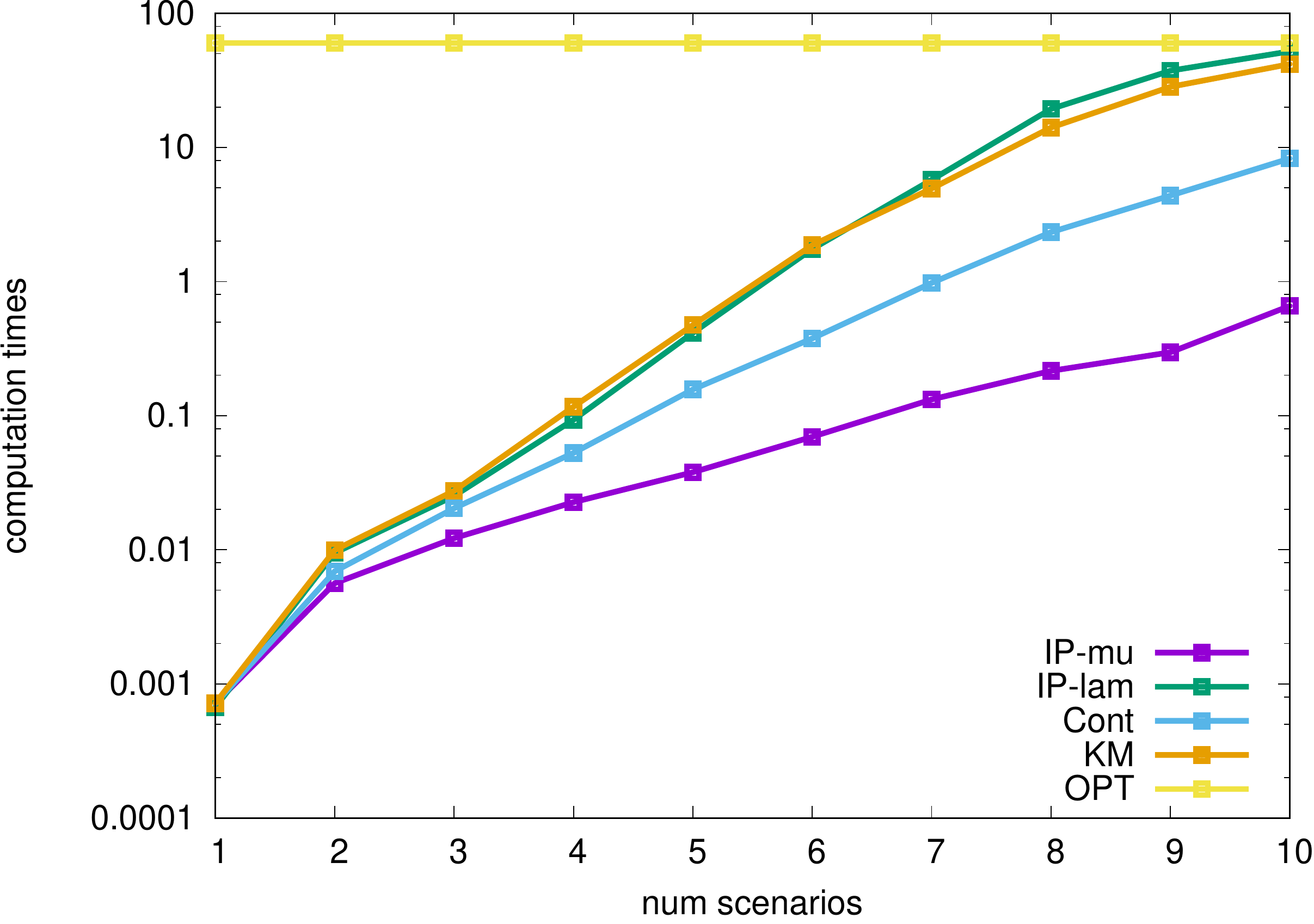}}
\end{center}
\caption{One-stage selection, average solution times.}\label{selonesoltime}
\end{figure}

\begin{figure}[htbp]
\begin{center}
\subfigure[$n=20$, $N=10$]{\includegraphics[width=0.3\textwidth]{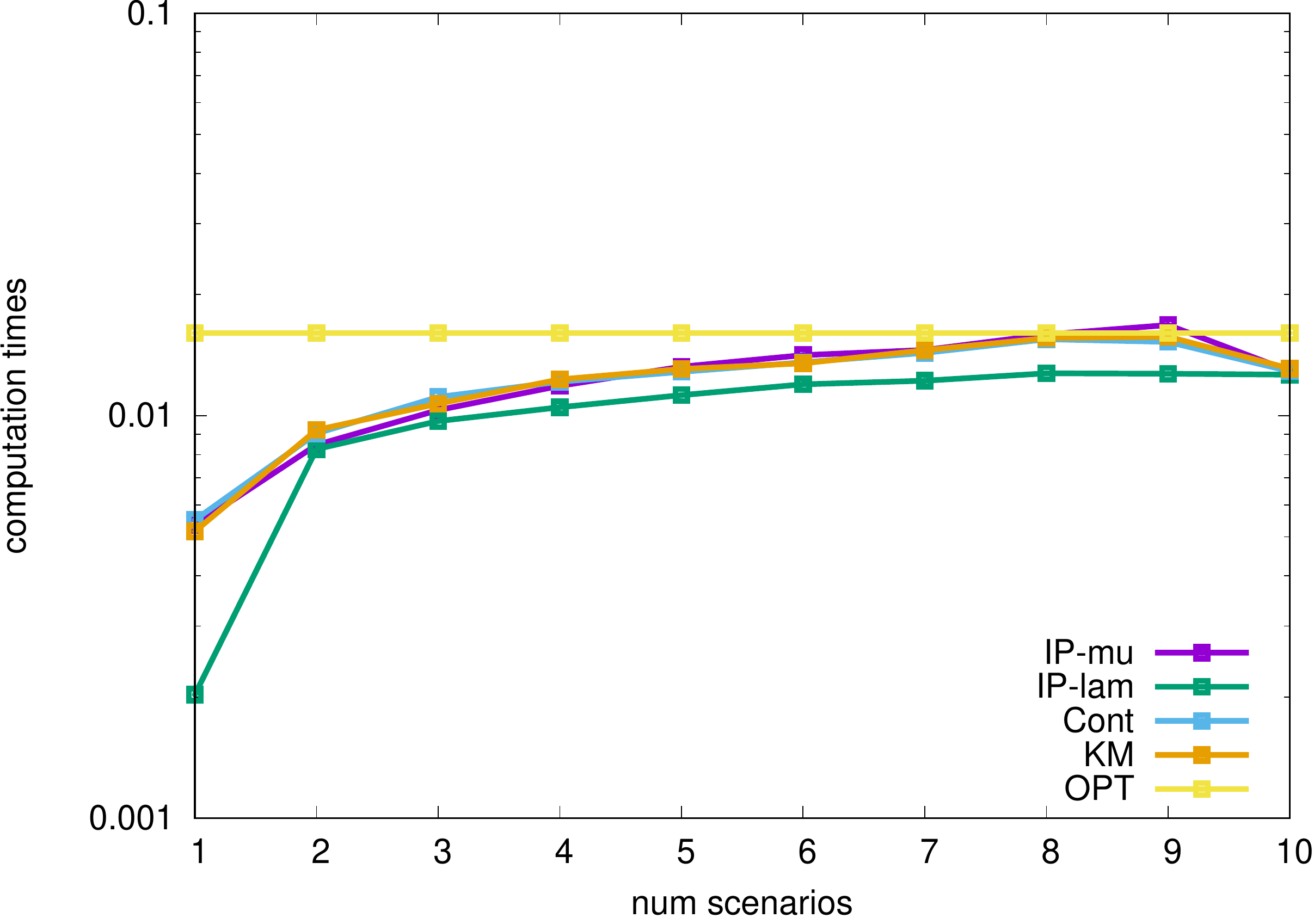}}%
\subfigure[$n=150$, $N=10$]{\includegraphics[width=0.3\textwidth]{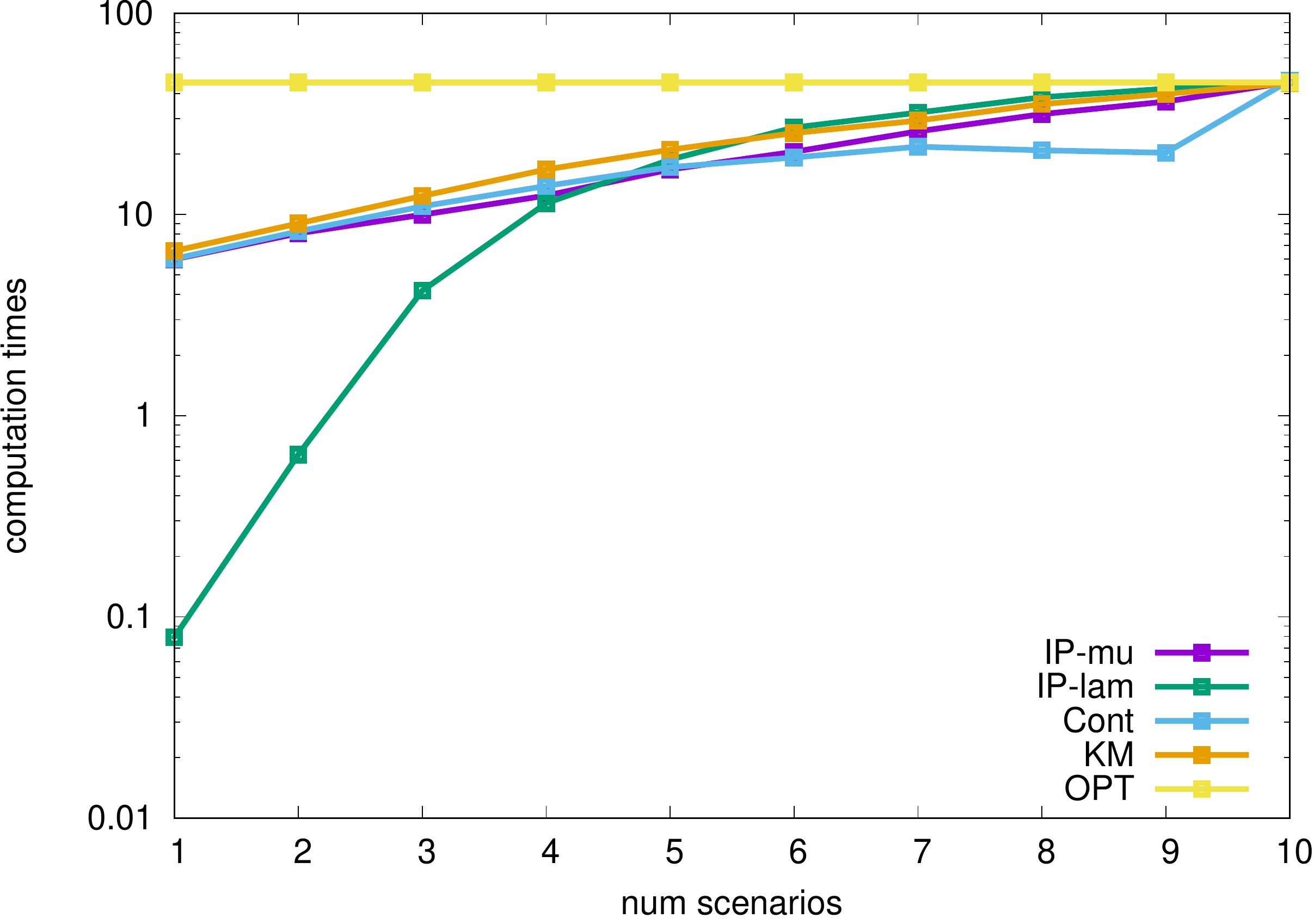}}
\subfigure[$n=20$, $N=50$]{\includegraphics[width=0.3\textwidth]{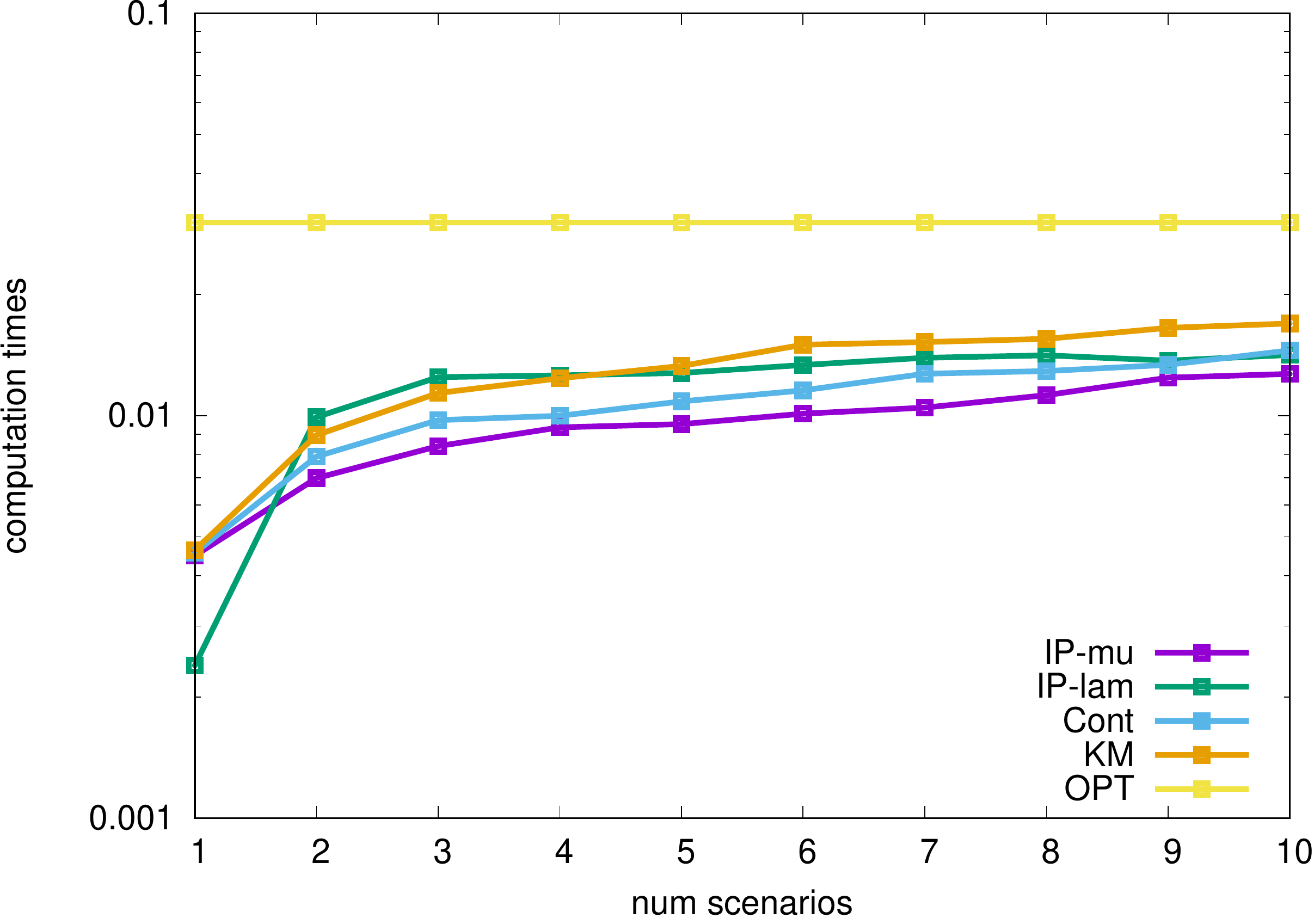}}%
\subfigure[$n=150$, $N=50$]{\includegraphics[width=0.3\textwidth]{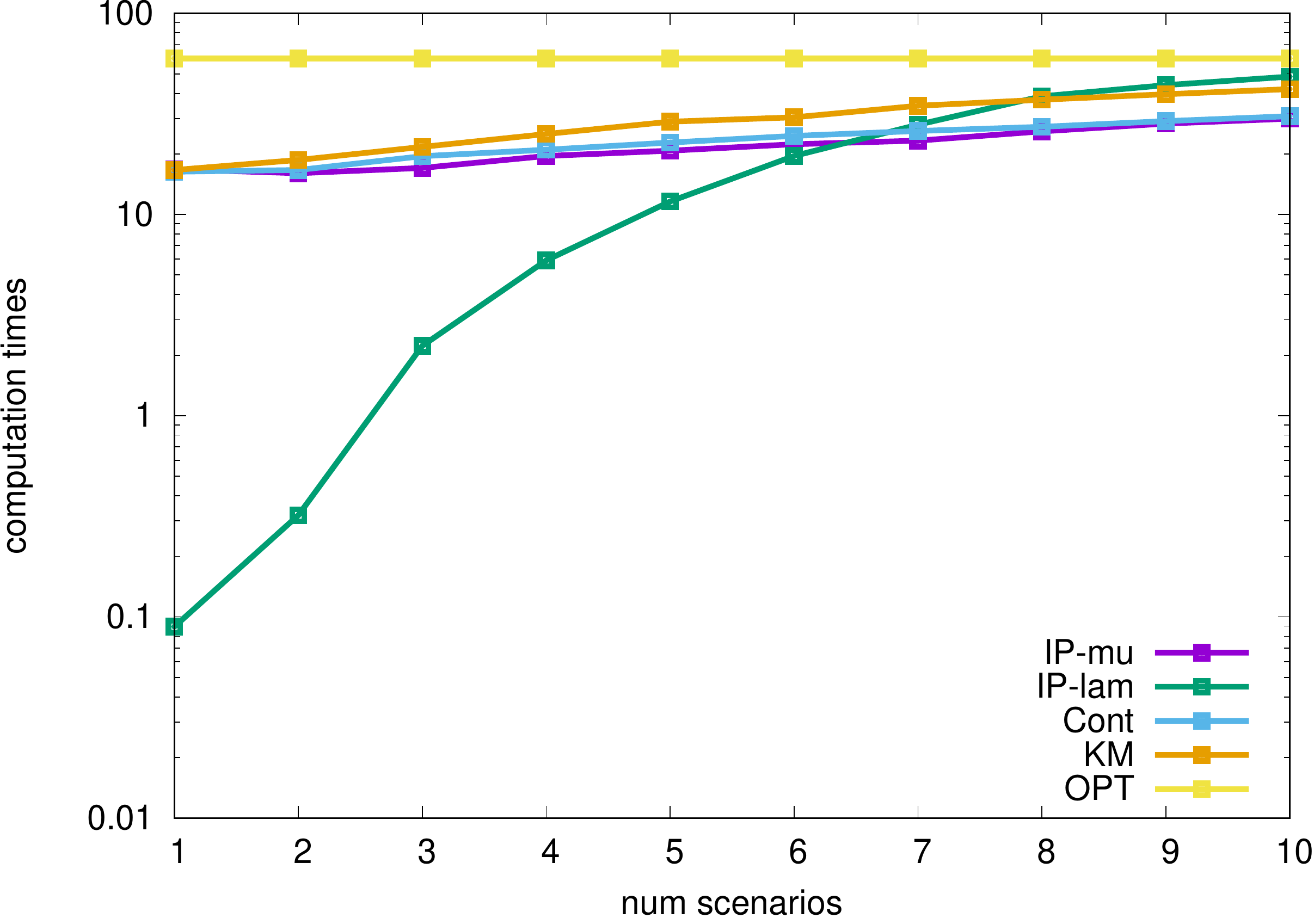}}
\end{center}
\caption{One-stage vertex cover, average solution times.}\label{vconesoltime}
\end{figure}

\begin{figure}[htbp]
\begin{center}
\subfigure[$n=20$, $N=10$]{\includegraphics[width=0.3\textwidth]{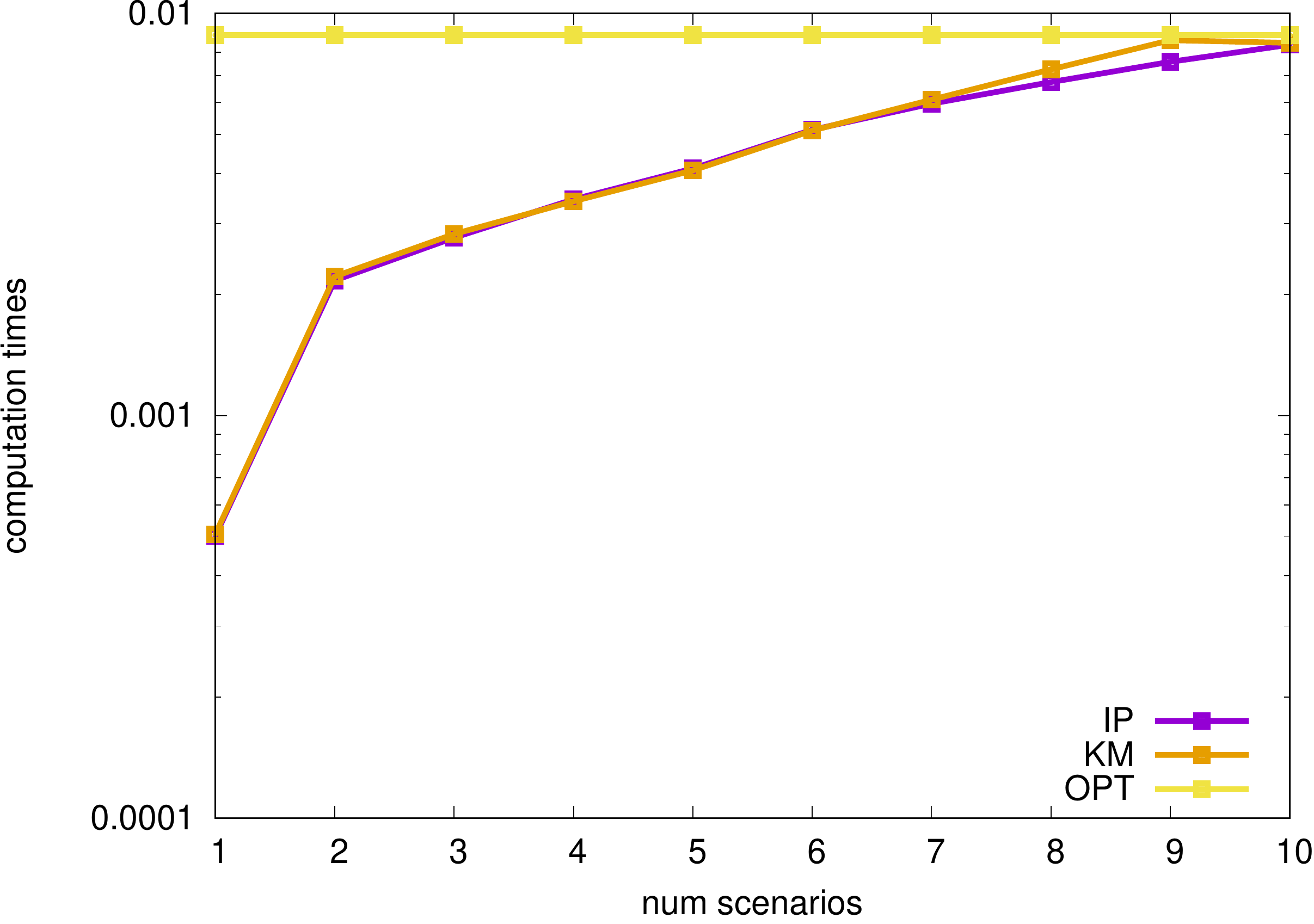}}%
\subfigure[$n=150$, $N=10$]{\includegraphics[width=0.3\textwidth]{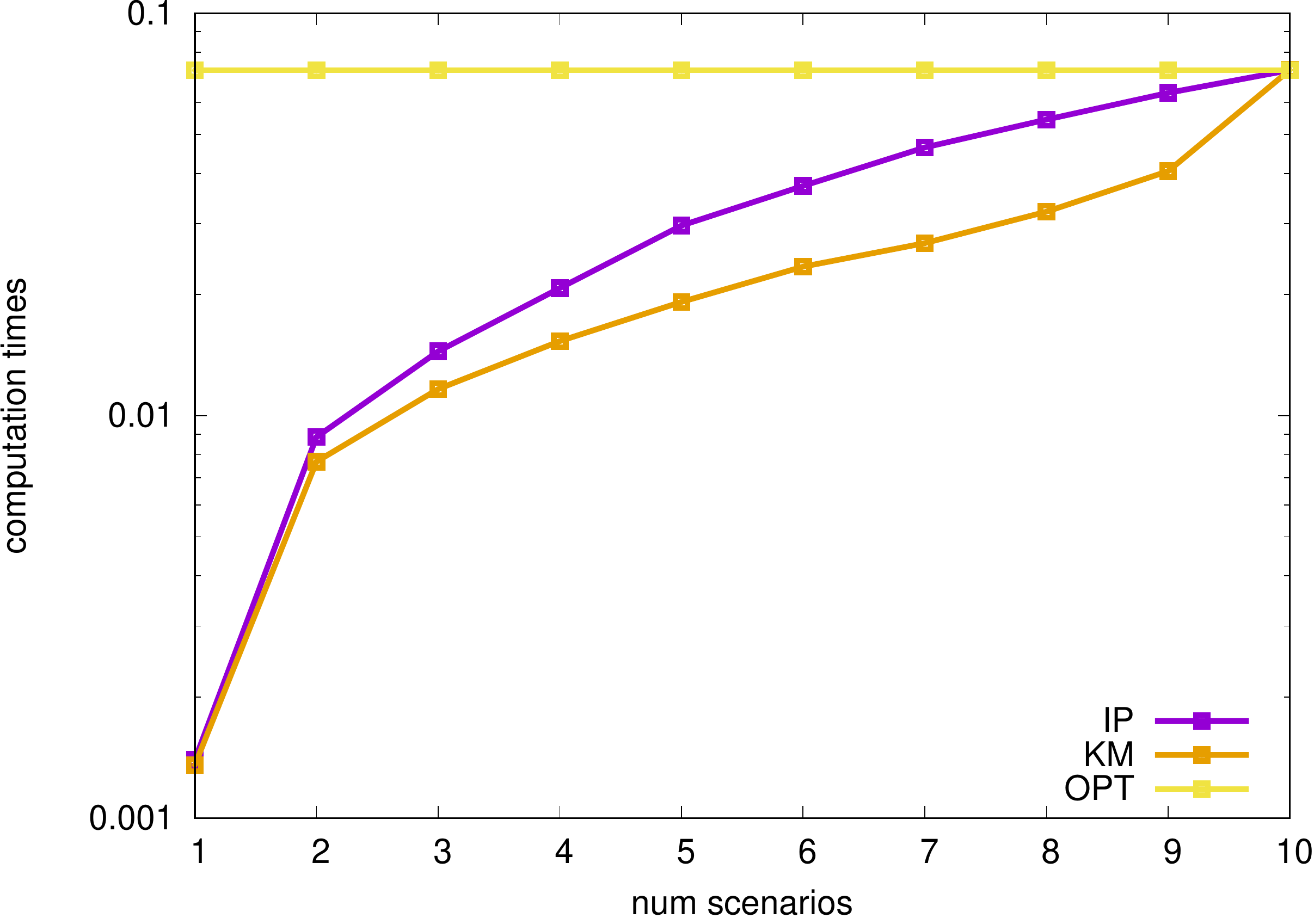}}
\subfigure[$n=20$, $N=50$]{\includegraphics[width=0.3\textwidth]{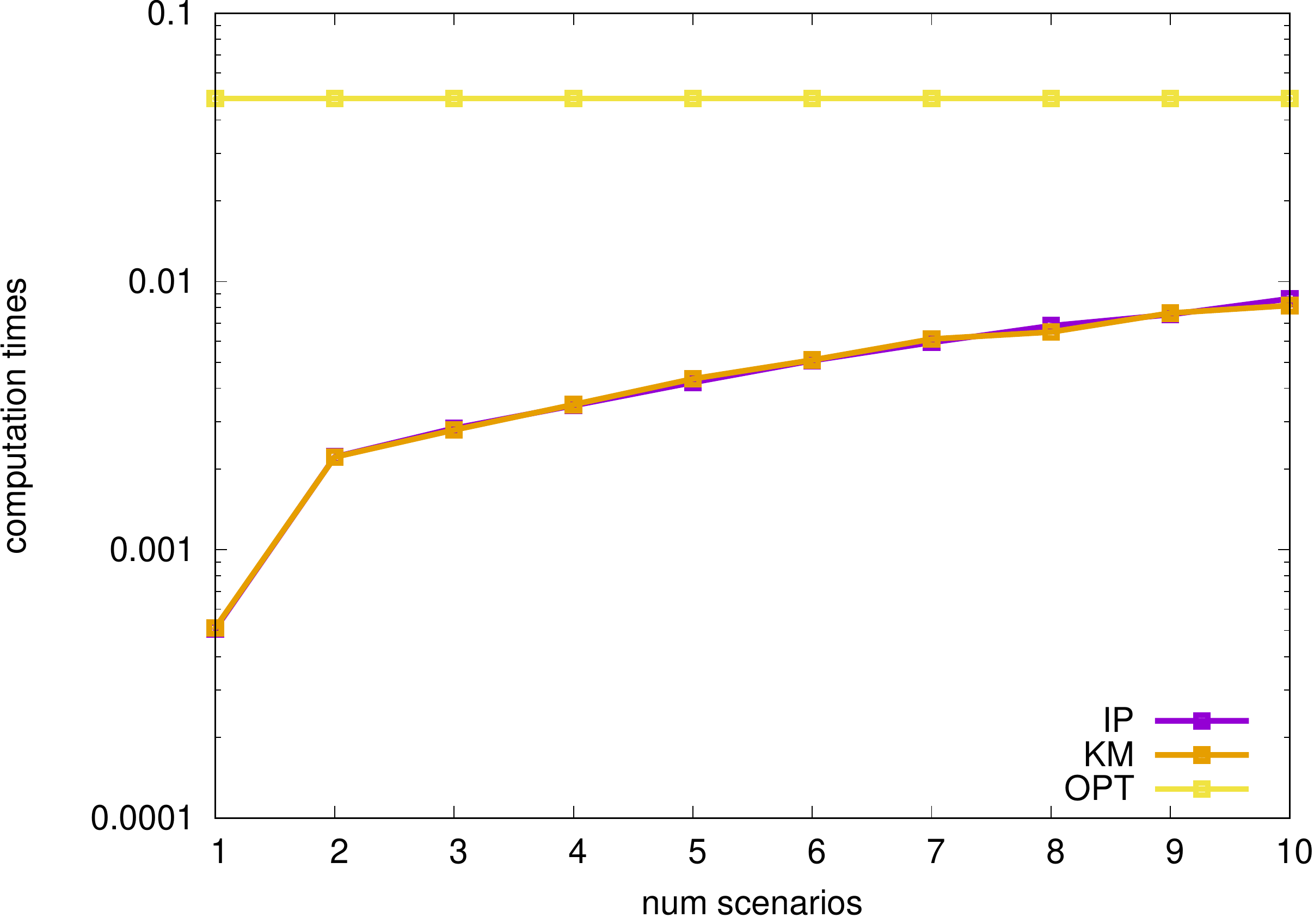}}%
\subfigure[$n=150$, $N=50$]{\includegraphics[width=0.3\textwidth]{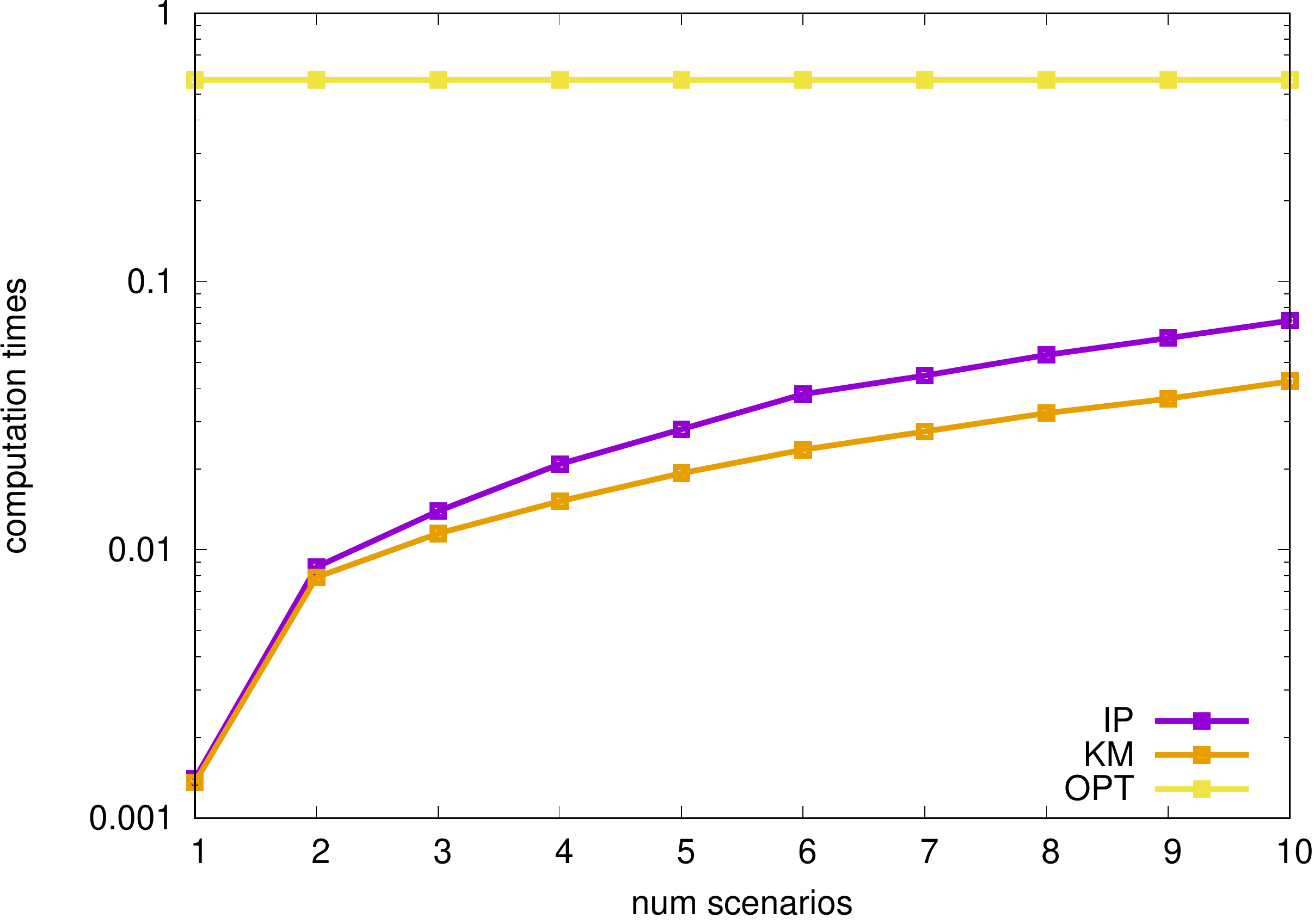}}
\end{center}
\caption{Two-stage selection, average solution times.}\label{seltwosoltime}
\end{figure}

\begin{figure}[htbp]
\begin{center}
\subfigure[$n=20$, $N=10$]{\includegraphics[width=0.3\textwidth]{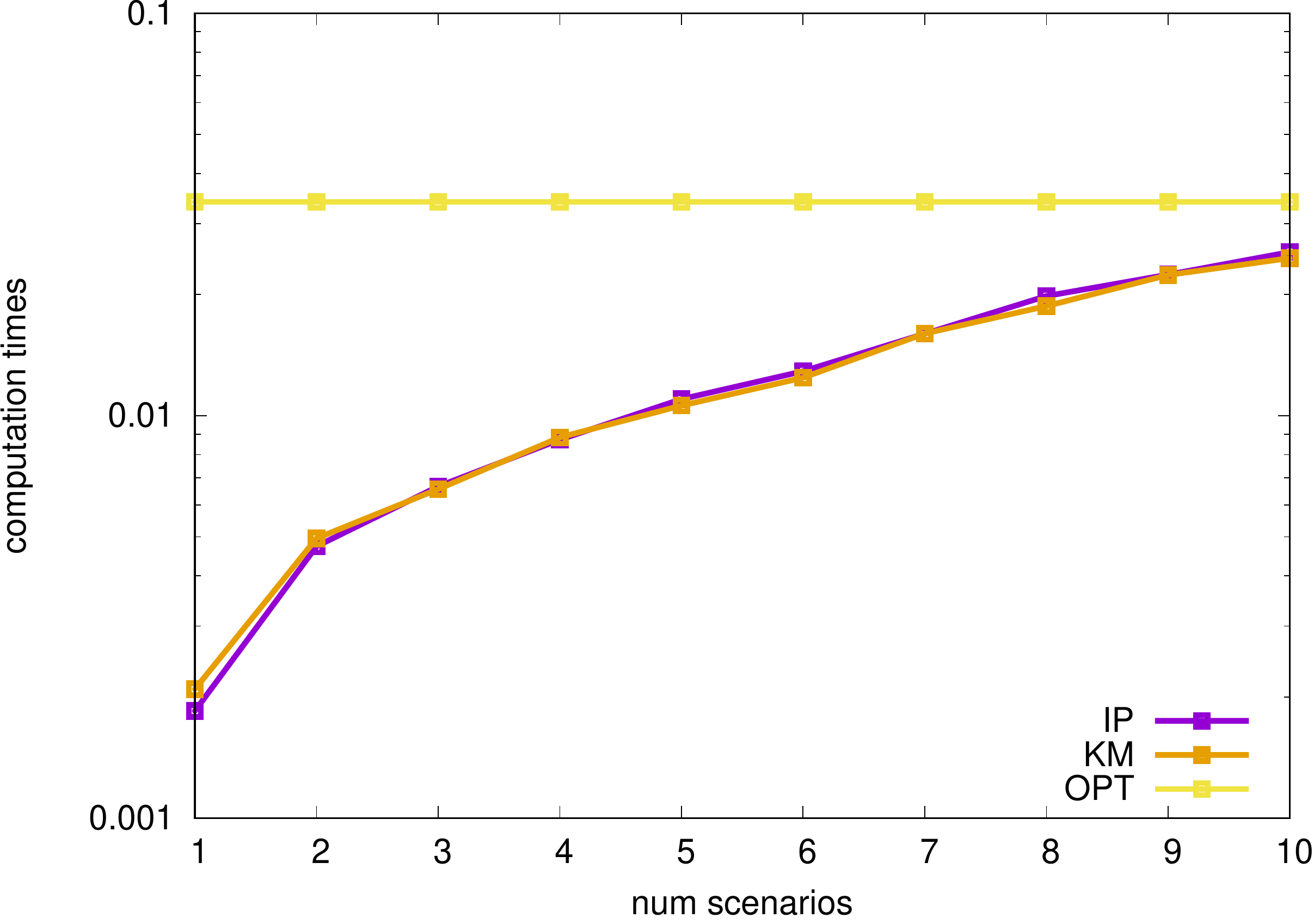}}%
\subfigure[$n=150$, $N=10$]{\includegraphics[width=0.3\textwidth]{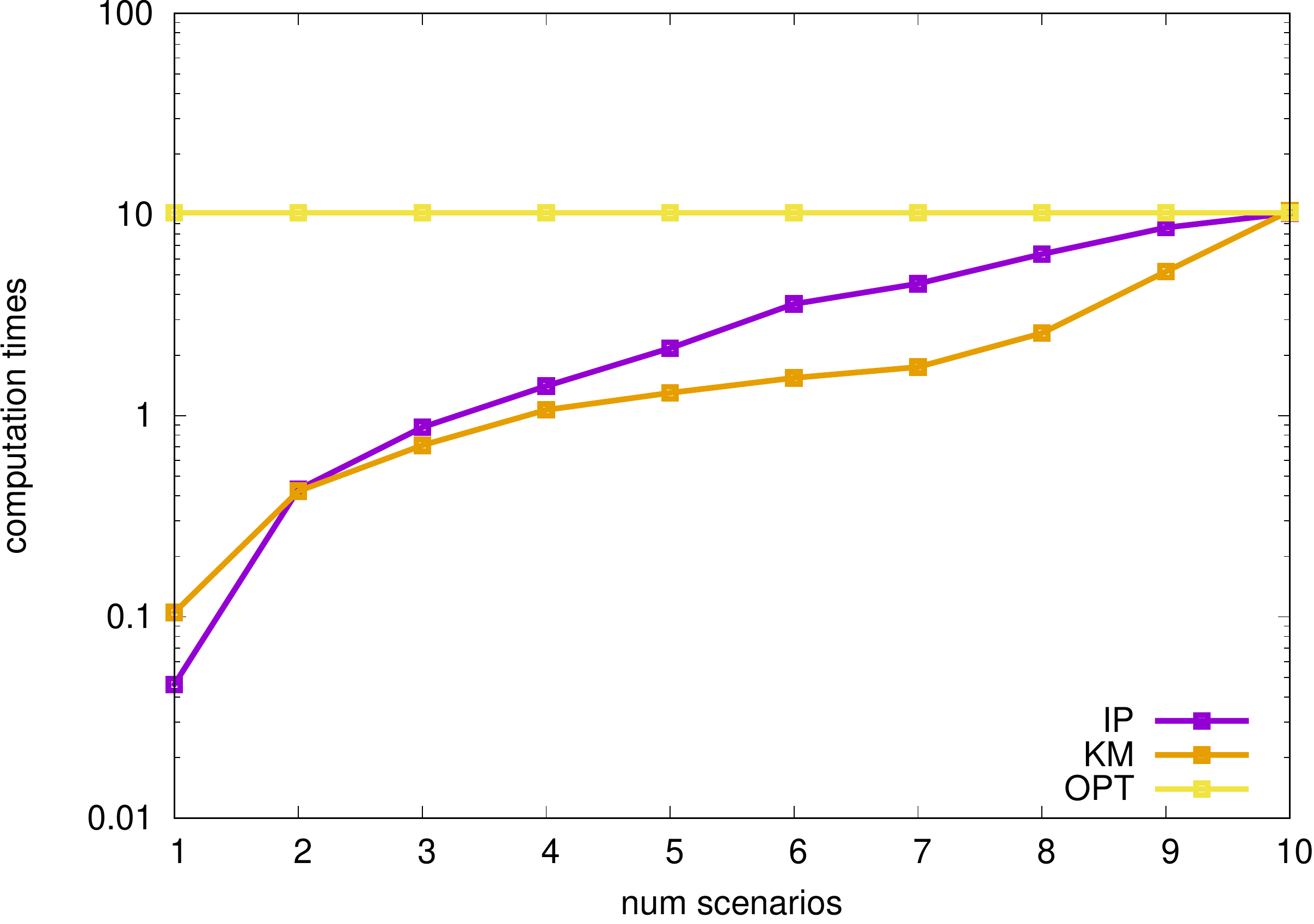}}
\subfigure[$n=20$, $N=50$]{\includegraphics[width=0.3\textwidth]{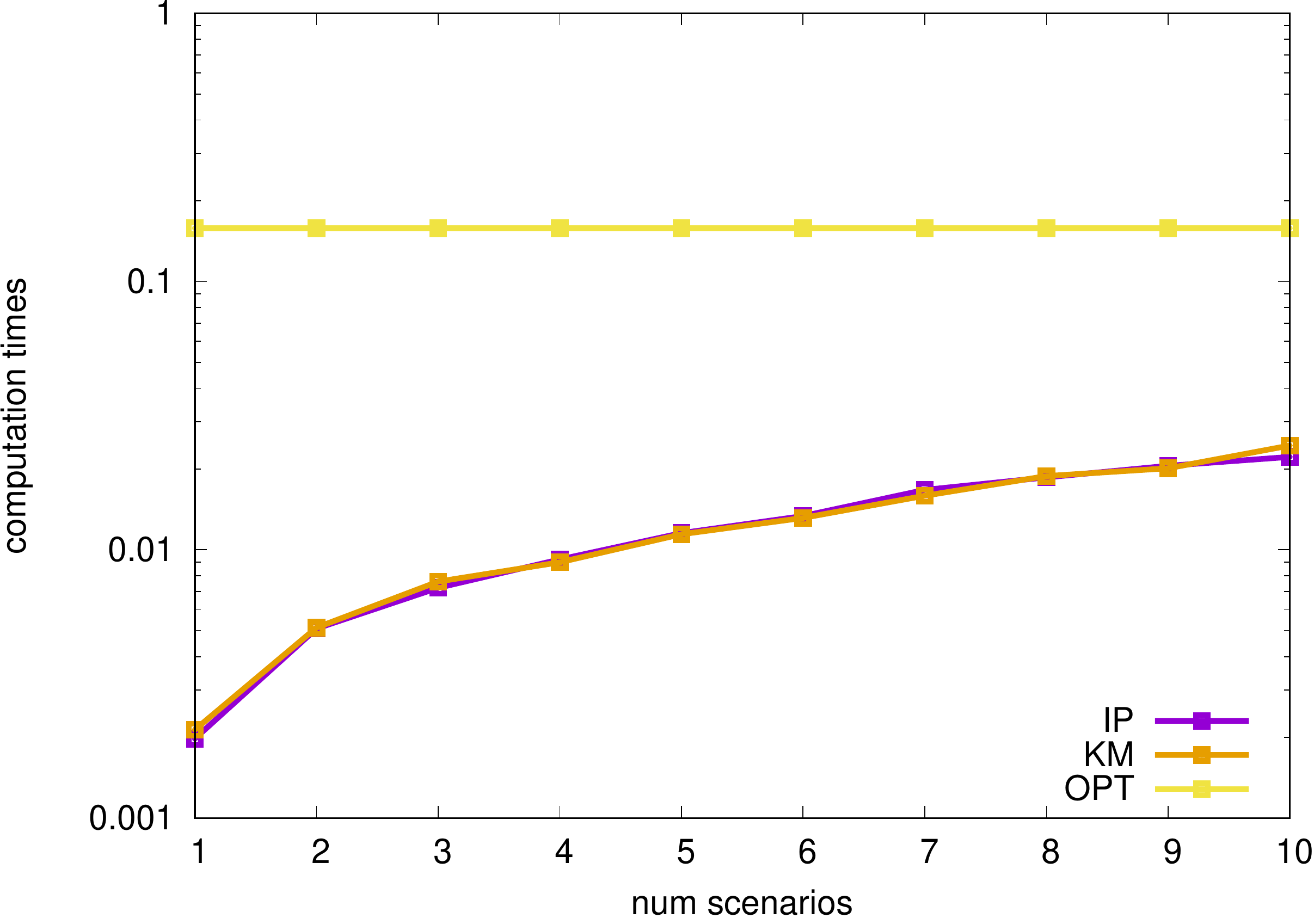}}%
\subfigure[$n=150$, $N=50$]{\includegraphics[width=0.3\textwidth]{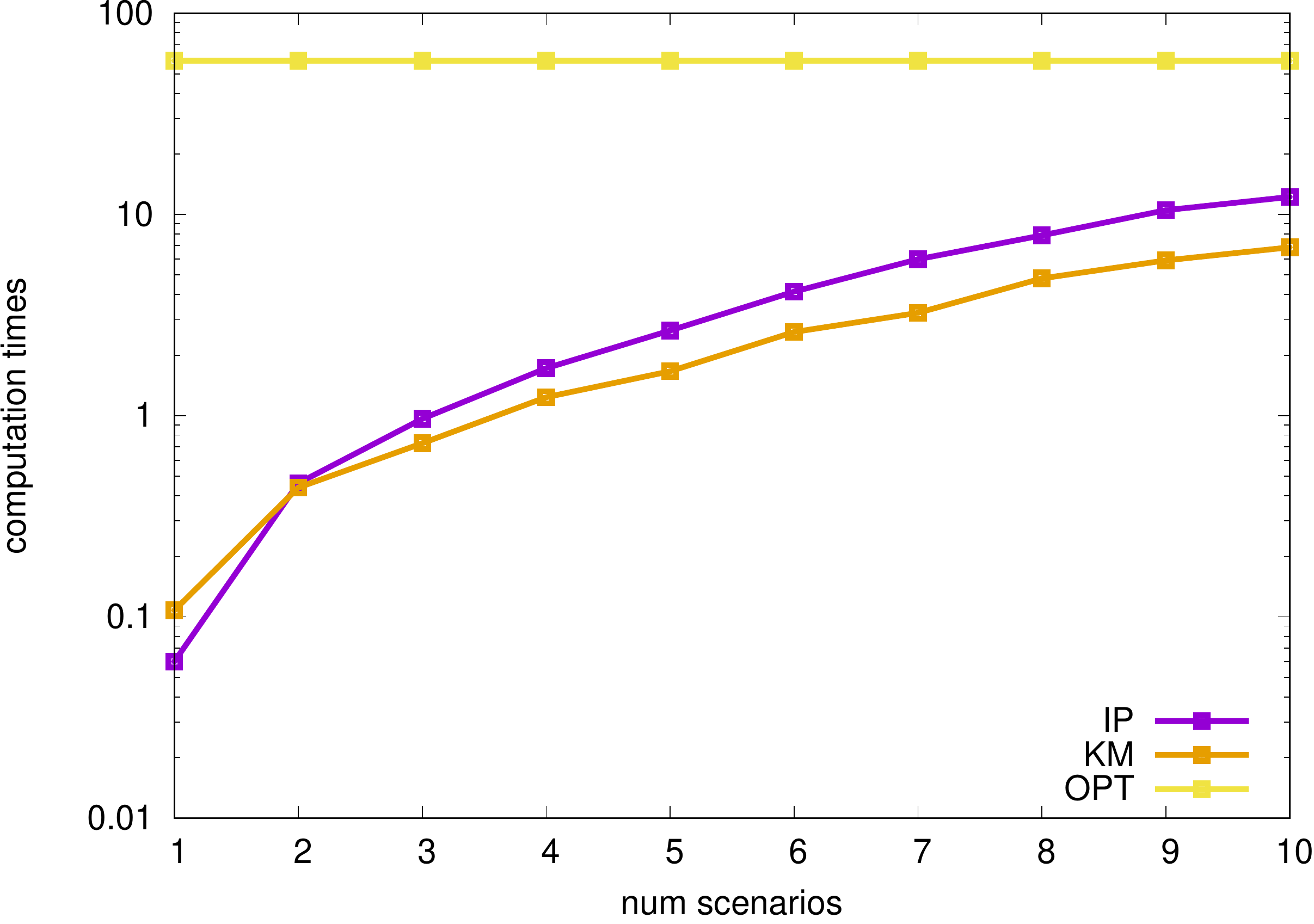}}
\end{center}
\caption{Two-stage vertex cover, average solution times.}\label{vctwosoltime}
\end{figure}

\end{document}